\documentclass[bj, preprint]{imsart}
\usepackage{amsmath}
\usepackage{amsthm}
\usepackage{amssymb}
\usepackage[stable]{footmisc}
\usepackage[numbers]{natbib}
\usepackage[colorlinks=true,linkcolor=blue,citecolor=blue]{hyperref}
\usepackage{hypernat}
\usepackage{verbatim}
\usepackage{textcomp}
\usepackage{enumitem}
\usepackage{graphicx}
\usepackage{color}

\definecolor{light-gray}{gray}{0.5}

\newcommand{\R}{{\mathbb R}}
\newcommand{\E}{{\mathbb E}}
\renewcommand{\P}{{\mathbb P}}
\newcommand{\N}{{\mathbb N}}

\newcommand{\U}{{\cal U}}
\newcommand{\eps}{\varepsilon}

\DeclareMathOperator{\Var}{Var}

\DeclareMathOperator{\trace}{trace}

\DeclareMathOperator{\diag}{diag}
\DeclareMathOperator{\rank}{rank}
\DeclareMathOperator{\s}{span}

\newtheorem{theorem}{Theorem}[section]
\newtheorem{proposition}[theorem]{Proposition}
\newtheorem{lemma}[theorem]{Lemma}
\newtheorem{corollary}[theorem]{Corollary}
\newtheorem{remark}[theorem]{Remark}
\newtheorem*{remark*}{Remark}
\newtheorem{example}[theorem]{Example}
\newtheorem{condition*}{Condition}

\numberwithin{equation}{section}

\newcounter{rcnt}[section]
\renewcommand{\thercnt}{(\roman{rcnt})}
\newcommand{\rem}[1]{}

\begin{document}

\sloppy


\title{On conditional moments 
	of \\high-dimensional random vectors
	given  \\lower-dimensional projections
}
\runauthor{Steinberger, Leeb}
\runtitle{Conditional moments of high-dimensional random vectors}

\begin{aug}

\begin{aug}
  \author{\fnms{Lukas}  \snm{Steinberger}\ead[label=e1]{lukas.steinberger@univie.ac.at}}
   \and
  \author{\fnms{Hannes} \snm{Leeb$^\ast$}\ead[label=e2]{hannes.leeb@univie.ac.at}}

  \affiliation{University of Vienna}

  \address{Department of Statistics, University of Vienna\\
		Oskar-Morgenstern-Platz 1, A-1090 Vienna, Austria,\\ 
		Tel.: +43-1-4277-38620\\
          \printead{e1,e2}}

\end{aug}



\end{aug}

\begin{abstract}
One of the most widely used properties of the multivariate
Gaussian distribution, besides its tail behavior, is the fact 
that conditional means are linear
and that conditional variances are constant. 
We here show that this property is also shared,  in an approximate
sense, by a large class of non-Gaussian distributions.
We allow for several conditioning variables
and we provide explicit non-asymptotic results,
whereby we extend earlier findings of \citet{Hal93a} and \citet{Lee12a}.
\end{abstract}


\maketitle



\section{Introduction}
\label{s1}

\subsection{Informal summary}
\label{summary}

The property of the multivariate Gaussian law, that
conditional means are linear and that conditional variances
are constant, 
is used by several fundamental statistical methods,
even if these methods per se do not require Gaussianity:
the generic linear model is built on the assumption
that the conditional mean of the response
is linear in the (conditioning) explanatory variables; 
and the generic homoskedastic
linear model rests on the additional assumption that the
conditional variance is constant.
Linear conditional means and/or constant conditional variances
are also assumed, for example, by 
methods for sufficient dimension reduction such as 
SIR \citep{Li91a} or SAVE \citep{Coo91a},
or by certain imputation techniques \citep{Qu10a}.
Elliptically contoured  distributions
are characterized by linear conditional means \citep{Eat86a}.
And methods for spatial statistics such as Kriging rely on 
Gaussianity mainly through the property that conditional means are
linear and that conditional variances are constant.\footnote{
	Distributions with
	linear conditional means and/or constant conditional
	variances are also studied, for example, in
	\citep{Bry05a, Coh12a, Jon98a, Plu83a, Tar04a, Wes93a}.
}
But even though these properties are widely used, in a sense the only 
distribution that has both linear conditional means and constant conditional
variances is the Gaussian (see also Section~\ref{outline}).

In this paper, we show that conditional means are approximately linear
and that conditional variances are approximately constant,
for a large class of multivariate distributions, when the conditioning
is on lower-dimensional projections. To illustrate our results,
consider a random $d$-vector $Z$ that has a Lebesgue density, 
and a $d\times p$ matrix $B$.
Conditional on $B'Z$, we show that the mean of $Z$ is  
linear in $B'Z$, and that the variance/covariance matrix of $Z$ is
constant, in an approximate sense.
Typically, our approximation error bounds are small if
$d$ is sufficiently large relative to $p$.
Our results extend recent findings of \citep{Lee12a}, where the
case $p=1$ is considered (which, from a modeling perspective, 
covers only models with one explanatory variable).
More precisely, we extend and refine the results of \citep{Lee12a}
in three directions: First, we allow for the case where $p>1$,
thereby also proving a result that is outlined in \citep[Sect. 5]{Hal93a}.
Second, we derive non-asymptotic and explicit
error bounds that hold for fixed $d$, whereas \citep{Hal93a} and
\citep{Lee12a} only give asymptotic results that hold as $d\to\infty$;
cf. Theorem~\ref{thm:main}.
And third, we also give asymptotic results where $p$ is allowed to
increase with $d$; see Corollary~\ref{asymptotics}.
In many cases, our error bounds go to zero if $p/\log d \to 0$.

The rest of the paper is organized as follows:
We continue this section with a more detailed description
of the results that we derive.
Our main results are then stated in Section~\ref{sec:main}. In 
Section~\ref{sec:ExEx} we provide a number of examples where 
the assumptions of our main theorem are satisfied and we discuss 
further extensions of our work.
Finally, Section~\ref{proof:main} gives a high-level description of the proof.
The more technical low-level parts of the proof as well as the proofs of 
Section~\ref{sec:ExEx} are collected in the supplementary material \citep{Ste16b}.

\subsection{Outline of results}
\label{outline}

Consider a random $d$-vector $Z$ that has a Lebesgue density,
and that is centered and standardized so that $\E Z = 0$ and
$\E Z Z' = I_d$.  And take a $d\times p$ matrix $B$ with orthonormal columns.
[While we do rule out degenerate distributions, the requirement
that $Z$ is centered and standardized, and the requirement that
the columns of $B$  are orthonormal, are inconsequential; cf.
Remark~\ref{inconsequential} as well as Section~\ref{sec:cov}.]
Our objective is to show that
the conditional mean and the conditional variance of $Z$ given $B'Z$
are close to what they would be if $Z$ were Gaussian. In the following, we use 
the notation $\|\cdot\|$ to denote the Euclidean norm of vectors and the spectral 
norm of matrices; the meaning of $\|\cdot\|$ will always be clear from the context.

Instead of the conditional mean and variance, it will be convenient
to focus on the first two conditional moments, i.e., on
$\E[Z \|B'Z]$ and on $\E[Z Z'\|B'Z]$.
If both the expressions
$$
	\Big\| 
		\E\left[ Z \| B' Z \right]
		\;-\; B B' Z
	\Big\|
\quad\text{ and } \quad
	\Big\| 
		\E\left[ Z Z'\| B' Z \right]
		\;- \;
		(I_d - BB' + BB' Z Z' BB')
	\Big\|
$$
are equal to zero, then the conditional mean of $Z$ given $B'Z$
is linear in $B'Z$, and the corresponding conditional variance is
constant in $B'Z$. But the only distribution, which satisfies this
for all $B$, is the Gaussian law; cf. the discussion in \citep[p. 466]{Lee12a}.
We will show that a weaker form of this requirement,
namely that the expressions in the preceding display are
close to zero in probability for most $B$, is satisfied by a much larger
class of distributions, provided mainly that $d$ is sufficiently large
relative to $p$.

For the case where $p=1$, it was shown in \citep{Lee12a}, for each $t>0$, that
\begin{align}\label{mean}
& \sup_{B \in \mathbb G} \P\left( \Big\| \E[Z \| B'Z] - B B' Z\Big\| > t \right) 
\qquad\text{ and }
\\
\label{variance}
&\sup_{B \in \mathbb G} \P\left( \Big\| \E[Z Z'\| B'Z] - 
	(I_d - BB' +  BB' Z Z' BB') \Big\| > t \right) 
\end{align}
converge to zero as $d\to\infty$,
under some conditions, where the sets $\mathbb G$ are collections
of $d\times p$ matrices with orthonormal columns that become large
as $d\to\infty$. More precisely, for $\nu_{d,p}(\cdot)$ denoting the
uniform distribution on the set of all such matrices
(i.e., the Haar measure on the Stiefel manifold $\mathcal V_{d,p}$),
the sets $\mathbb G$ satisfy $\nu_{d,p}(\mathbb G)\to 1$ as $d\to\infty$.
[Obviously, $\mathbb G$ depends on $d$ and also on $p$, although
this dependence is not shown explicitly in our notation.]
In the case where $p=1$ covered in \citep{Lee12a}, the sets $\mathbb G$ are
collections of unit-vectors, and $\nu_{d,1}(\cdot)$ is the
uniform distribution on the unit sphere, in $\R^d$.
We derive a non-asymptotic version of this result, i.e.,
explicit upper bounds on \eqref{mean} and \eqref{variance},
and also on $1-\nu_{d,p}(\mathbb G)$, that hold
for fixed $d$ and $p$, where we allow for $p>1$;
see Theorem~\ref{thm:main}.
Moreover, we also provide an asymptotic result where our
upper bounds go to zero as $d\to\infty$, where $p$ may increase with $d$;
cf. Corollary~\ref{asymptotics}.
In many cases, our upper bounds are small provided that $p/\log d$ is small.
Both our non-asymptotic and our asymptotic result, i.e., both
Theorem~\ref{thm:main} and Corollary~\ref{asymptotics}, hold uniformly
over classes of distributions for $Z$, as outlined in Remark~\ref{uniformity}.

Of course, our results rely on further conditions on the distribution
of $Z$ (in addition to the existence of a Lebesgue density and
the requirements that $\E Z = 0$ and $\E Z Z' = I_d$).
In particular, we require that the mean of certain  
functions of $Z$, and of i.i.d. copies
of $Z$, is bounded; see the bounds \ref{c.Sk}.(\ref{c.sup}) and \ref{c.density},
as well as  the attending discussion in Section~\ref{sec:main}.
And we require that certain moments of $Z$
are close to what they would be in the
Gaussian case; see \ref{c.Sk}.(\ref{c.monom1}-\ref{c.monom2}).
From a statistical perspective, we stress that our results rely
on bounds that can be estimated from appropriate data, as outlined
in the discussion leading up to Theorem~\ref{thm:main}.
One particularly simple example, where these bounds hold,
and where the error bounds in Theorem~\ref{thm:main} get small as $d$ gets 
large, is the case where the components of $Z$ are independent,
with bounded marginal densities and bounded
marginal moments of sufficiently large order;
see Example~\ref{Ex:1}.
Finally, we emphasize that \ref{c.Sk} and \ref{c.density}
do not require that the components of $Z$ are independent.

The results in this paper
demonstrate that the requirement of linear conditional means and
constant conditional variances (which is quite restrictive as
discussed in the second paragraph of this subsection)
is actually satisfied, in an approximate sense,
by a rather large class of distributions. Some implications,
namely to sparse linear modeling, and to sufficient dimension reduction
methods like SIR or SAVE, are discussed in \citep[Sect. 1.4]{Lee12a}.
And while the discussion in \citep{Lee12a} is hampered by the
fact that only situations with $p=1$, i.e., 
{\em only models with one explanatory variable}, 
are covered in that paper, our results show that these
considerations extend also to the case where $p>1$, i.e.,
to more complex models with several explanatory variables.
\\

\begin{remark}\normalfont\label{inconsequential}
\begin{list}{\thercnt}{
        \usecounter{rcnt}
        \setlength\itemindent{5pt}
        \setlength\leftmargin{0pt}
        \setlength\partopsep{0pt} 
        }
\item
Our requirements, that the random $d$-vector $Z$ is centered and
standardized, and that the matrix $B$ has orthonormal columns,
are inconsequential in the following sense:
Consider a random $d$-vector $Y$ such that $\E[Y] = \mu$ and
$\Var[Y ] = \Sigma$  are both well-defined and finite, and such that
$Y$ has a Lebesgue density (which also entails that $\Sigma$ is invertible).
Moreover, consider a $d\times p$ matrix $A$ with linearly independent columns. 
If $Y$ were Gaussian, we would have 
$\E[Y\|A'Y] = \mu+\Sigma A (A' \Sigma A)^{-1} A' (Y - \mu)$. In general,
one easily verifies that
\begin{align*}
&\Big\|
\E[Y\|A'Y] - (\mu+\Sigma A (A' \Sigma A)^{-1} A'( Y - \mu)) 
\Big\|
\;\; \leq \;\; \|\Sigma\|^{1/2}
\,\Big\|
	\E[Z \|B'Z] - B B' Z
\Big\|
\end{align*}
holds 
for $Z = \Sigma^{-1/2}(Y-\mu)$ and $B=\Sigma^{1/2} A (A' \Sigma A)^{-1/2}$.
Note that
$Z$ has a Lebesgue density;  that $Z$ is centered and standardized so 
that $\E[Z] =0$ and $\E[Z Z'] = I_d$; 
and that the columns of $B$ are orthonormal.
In particular, we see that the conditional mean of $Y$ given $A' Y$ is
approximately linear if 
the same is true for the conditional mean of $Z$ given $B'Z$, provided
only that the largest eigenvalue of $\Sigma$ is not too large.
A similar consideration applies, mutatis mutandis, to the
conditional variance of $Y$ given $A'Y$ and that of $Z$ given $B'Z$.
For further details, in particular about the role of $\Sigma$, see 
Section~\ref{sec:cov}.

\item
Conditioning on $B'Z$ is equivalent to conditioning 
on $B B'Z$, which is the orthogonal projection of $Z$ onto the 
column space of $B$. Therefore, we could formulate 
Theorem~\ref{thm:main}  
for collections of $p$-dimensional subspaces $S$ of 
$\R^d$ (elements of the Grassmann manifold $\mathcal G_{d,p}$) 
instead of matrices $B$ (from the 
Stiefel manifold $\mathcal V_{d,p}$), and thus replace \eqref{mean} by 
\begin{align}\nonumber
\sup_{S\in \mathbb H} \P\left( \Big\| \E[Z \| P_SZ] - P_S Z\Big\| > t \right),
\end{align}
with $P_S$ denoting the orthogonal projection matrix for the subspace $S$.
Here, $\mathbb H$ denotes the image of the set $\mathbb G \subseteq {\mathcal V}_{d,p}$
from \eqref{mean}
under the mapping that maps a matrix $B$ into its column space $S$.
Note that the image of the Haar measure on ${\mathcal V}_{d,p}$
under this mapping is the Haar measure on ${\mathcal G}_{d,p}$;
see also
\citep[Theorem~2.2.2(iii)]{Chi03a}. In a similar manner, one can also 
write \eqref{variance} in terms of the Grassmann manifold, namely as
\begin{align*}
\sup_{S \in \mathbb H} \P\left( \Big\| \E[Z Z'\| P_S Z] - 
	(I_d -P_S  + P_S Z (P_S Z)') \Big\| > t \right). 
\end{align*}
\end{list}
\end{remark}


\section{Results}
\label{sec:main}

We first present our main non-asymptotic result, i.e., Theorem~\ref{thm:main},
and the bounds \ref{c.Sk} and \ref{c.density} that it relies on.
These  bounds depend on a constant $k$ that will be chosen as needed later.
In Corollary~\ref{asymptotics} and the attending discussion,
we then present asymptotic scenarios in which the constants in \ref{c.Sk} and
\ref{c.density} can be controlled, such that the error bounds
in Theorem~\ref{thm:main} become small.
Throughout the following, consider a random $d$-vector  $Z$ that has a 
Lebesgue density and that satisfies $\E Z = 0$ 
as well as  $\E Z Z' = I_d$. For $k \in \N$, write $Z_1,\dots, Z_k$
for i.i.d. copies of $Z$, and write
$S_k$ for the $k\times k$ Gram-matrix $S_k = (Z_i' Z_j/d)_{i,j=1}^k$.
For $g\ge 0$, a monomial of degree $g$ in the elements of $S_k-I_k$ is an 
expression of the form $G = \prod_{\ell=1}^g (S_k-I_k)_{i_\ell,j_\ell}$ for 
$(i_\ell, j_\ell)\in\{1,\dots, k\}^2$, $1\le \ell\le g$ (with the convention that $G=1$
in case $g=0$). We say that $G$ has a linear (resp. quadratic) factor if one of 
the pairs, say $(i_1, j_1)$, occurs exactly once (resp. twice).

\begin{enumerate}
        \setlength\leftmargin{-20pt}
\renewcommand{\theenumi}{(b\arabic{enumi})}
\renewcommand{\labelenumi}{\textbf{\theenumi}}

\item \label{c.Sk} Fix $k\in \N$.

\begin{enumerate}[leftmargin=5pt]
\renewcommand{\theenumii}{\alph{enumii}}
\renewcommand\labelenumii{(\theenumii)}
\makeatletter \renewcommand\p@enumii{} \makeatother

	\item \label{c.sup} There are constants $\eps\in[0,1/2]$ and $\alpha\ge 1$ so that 
	$ \E \| \sqrt{d}(S_k - I_k)\|^{2 k + 1 + \eps} 
		\leq \alpha$.
	
	\item \label{c.monom1} There are constants $\beta>0$ and $\xi\in(0,1/2]$ that satisfy the following:
	For any monomial $G = G(S_k-I_k)$ in the elements of $S_k-I_k$,
	whose degree $g$ satisfies $g \leq 2 k$, 
	we have $| d^{g/2} \E G - 1 | \leq \beta/d^{\xi}$ 
	if $G$ consists
	only  of quadratic factors in elements above the diagonal,
	and $| d^{g/2} \E G | \leq \beta /d^{\xi}$
	if $G$ contains a linear factor.
	
	\item \label{c.monom2} The constants $\beta$ and $\xi$ in (\ref{c.monom1}) also satisfy the following: 
	Consider two monomials $G = G(S_k-I_k)$ and $H = H(S_k-I_k)$ 
	of degree $g$ and $h$, respectively, 
	in the elements of $S_k-I_k$. If $G$ is given by 
	$Z_1'Z_2Z_2'Z_3\dots Z_{g-1}'Z_gZ_g'Z_1/d^g$, if $H=\prod_{\ell=1}^h (S_k-I_k)_{i_\ell,j_\ell}$
	 with $\{1,\dots, g\} \subseteq \{i_1,j_1,\dots, i_h,j_h\}$, and 
	if $2\le h<g\le k$, then $| d^g\E GH | \leq \beta /d^\xi$.
	\end{enumerate}

\item \label{c.density} 
For fixed $k\in \N$, there is a constant $D\geq 1$, such that the 
following holds true:
If $R$ is an orthogonal $d\times d$ matrix, 
then a marginal density of the first $d-k+1$ components of $RZ$ is 
bounded by $\binom{d}{k-1}^{1/2} D^{d-k+1}$.
\end{enumerate}


The bounds in \ref{c.Sk} and \ref{c.density} essentially guarantee
that moments of certain functions of the Gram matrix $S_k$ are either
bounded (in \ref{c.Sk}.(\ref{c.sup}) and \ref{c.density}) or not too 
different from what they would be if $Z$ were Gaussian
(in \ref{c.Sk}.(\ref{c.monom1}-\ref{c.monom2})). We will impose \ref{c.Sk} and \ref{c.density}
with $k=2$ when considering conditional means, and with $k=4$ when considering conditional variances.
Clearly, \ref{c.Sk} becomes stronger as $k$ increases.
The specific requirements in \ref{c.Sk} are minimal for our current method of proof and
the bound in \ref{c.density} is chosen in such a way that certain constants $\gamma_1$ and $\gamma_2$
appearing in Theorem~\ref{thm:main} do not depend on the dimension $d$.
Other methods of proof,
if such can be found, may rely on different conditions.
For further discussion and specific examples 
where our conditions apply, see Section~\ref{sec:ex}.

The bounds in \ref{c.Sk} are non-asymptotic versions of 
condition (t1) in \citep{Lee12a}, and the bound in \ref{c.density}
coincides with condition (t2) in that reference.
The bounds in \ref{c.Sk}.(\ref{c.monom1}-\ref{c.monom2})
are written as $\beta/d^{\xi}$, because in Corollary~\ref{asymptotics}
we will consider situations where these bounds hold for constants
$\beta$ and $\xi$ that either are both independent of $d$, or that are such
that $\beta$ is independent of $d$ while $\xi$ depends on $d$ 
so that $1/d^\xi \to 0$. In \ref{c.density}, note that the upper bound on the marginal densities can increase in $d$.
The bound in~\ref{c.density} appears to be qualitatively different from 
\ref{c.Sk} in that it does not directly impose restrictions on moments involving the 
standardized Gram matrix $S_k - I_k$.
However, \ref{c.density} is used only to bound the
$p$-th moment of $\det S_l^{-4(k+1)}$ for $l=1,\dots, k$;
cf. Lemma~\ref{p4.l1} and the proof of Proposition~\ref{p4} in 
Appendix~\ref{appendixD} of the supplement. Just like Condition~\ref{c.Sk}, the requirement of 
a uniform bound on $\max_{l\le k}\E\det S_l^{-4p(k+1)}$, becomes more restrictive if $k$ increases.
From a statistical perspective, we note that the moment-bounds discussed
here can be estimated from a sample of independent copies of $Z$.
Indeed, population means like $\E \| S_k -I_k\|^{2 k + 1 + \eps}$,
$\E G$, $\E G H$, or $\E \det S_l^{-4 p (k+1)}$ as above are readily
estimated by appropriate sample means. In this sense, we rely on bounds
that can be estimated from data. 

\begin{theorem}\label{thm:main}
For fixed $d$, consider a random $d$-vector $Z$ that has a Lebesgue density $f_Z$
and that is standardized such that $\E Z = 0$ and $\E Z Z' = I_d$.
\begin{list}{\thercnt}{
        \usecounter{rcnt}
        \setlength\itemindent{5pt}
        \setlength\leftmargin{0pt}
        \setlength\partopsep{0pt} 
        }
\item \label{theoremA}
Suppose that \ref{c.Sk}.(\ref{c.sup}-\ref{c.monom1}) 
and \ref{c.density} hold with $k=2$.
Then, for each $p < d$ and for each 
$\tau\in(0,1)$, there is a  Borel set $\mathbb G \subseteq \mathcal V_{d,p}$
such that \eqref{mean} is bounded by
\begin{align}\label{eq:boundMean}
\frac{1}{t} d^{-\tau \xi_1} \,+\, \frac{\gamma_1}{1 - \tau} 
	\frac{p}{ 3 \xi_1 \log d}
\end{align}
for each $t>0$, 
and such that
\begin{align}\label{eq:boundG1}
\nu_{d,p}\left( \mathbb G^c\right)  \quad\leq \quad \kappa_1 \,
	d^{-\tau \xi_1 \left(1 - \frac{\gamma_1 }{\tau}
	\frac{p}{\xi_1 \log d}\right)},
\end{align}
where $\xi_1$ is given by
$\xi_1 = \min\{\xi, \eps/2 + 1/4, 1/2\}/3$ and $\gamma_1 = \max\{g_1, 6+2\log(2D\sqrt{\pi e}) \}$. 
Here, the constant $\kappa_1$ depends only on $\alpha$ and $\beta$ and $g_1$ is a global constant.

\item \label{theoremB}
Suppose that \ref{c.Sk}.(\ref{c.sup}-\ref{c.monom2}) 
and \ref{c.density} hold with $k=4$.
Then, for each $p < d$ and for each $\tau\in(0,1)$, 
there is a Borel set $\mathbb G\subseteq \mathcal V_{d,p}$ so that both \eqref{mean}
and \eqref{variance} are bounded by 
\begin{align}\label{eq:boundVariance}
\frac{1}{t}d^{-\tau \xi_2} + \frac{\gamma_2}{1-\tau}\frac{p}{5 \xi_2 \log{d}}
\end{align}
for each $t>0$, 
and such that 
\begin{align}\label{eq:boundG2}
\nu_{d,p}(\mathbb G^c) \quad \leq \quad 2 \kappa_2\, 
		d^{-\tau \xi_2\left(1-\frac{\gamma_2}{\tau}
			\frac{p}{\xi_2 \log{d}} \right)}.
\end{align}
Here, $\xi_2$ is given by $\xi_2 = \min\{\xi, \eps/2 + 1/4, 1/2\}/5$ and
$\gamma_2 = \max\{g_2, 10 + 4\log(2D\sqrt{\pi e})\}$. The constant $\kappa_2$ depends
only on $\alpha$ and $\beta$ and $g_2$ is a global constant.

\item \label{theoremC}
The set $\mathbb G$ in both parts~\ref{theoremA} and \ref{theoremB} can be chosen to have the following additional properties:
$\mathbb G$ is right-invariant under the action of the orthogonal group of order $p$ and it is left orthogonally equivariant, i.e., $\mathbb G = \mathbb G(f_Z)$ depends on the distribution of $Z$ in such a way that $\mathbb G(f_{RZ}) = R \mathbb G(f_Z)$, for every $d\times d$ orthogonal matrix $R$.
\end{list}

The constants $g_1$, $g_2$, $\kappa_1$ and $\kappa_2$ in part \ref{theoremA}
and \ref{theoremB} can be obtained explicitly upon detailed
inspection of the proof.
\end{theorem}

With Theorem~\ref{thm:main}, we aimed to obtain 
the best possible upper bounds for \eqref{mean}, \eqref{variance}
and $\nu_{d,p}(\mathbb G^c)$ that our current technique of proof delivers.
It is likely that better bounds can be obtained
under stronger assumptions (like in the case where the components of $Z$
are independent) 
together with an alternative method of proof.
In particular, when bounding \eqref{mean} 
in Theorem~\ref{thm:main}\ref{theoremA},
the term 
$\frac{\gamma_1}{1-\tau} \frac{p}{3 \xi_1 \log d}$
 is obtained by bounding
$\P( \|B' Z\|^2 > (1-\tau)3 \xi_1 \log (d) / \gamma_1)$ 
using Chebyshev's inequality;
cf. the proof  of Lemma~\ref{final} in the supplement.
Under appropriate additional assumptions on the tails of
$\|B' Z\|$, this bound can be dramatically improved.
The bound on  both \eqref{mean} and \eqref{variance}
in Theorem~\ref{thm:main}\ref{theoremB} can be improved
in a similar fashion (cf. Section~\ref{sec:bounds}).
When proving Theorem~\ref{thm:main},
we derive upper bounds for \eqref{mean} and \eqref{variance}, on the 
one hand, and for  $\nu_{d,p}(\mathbb G^c)$, on the other hand, that are 
antagonistic in the sense that one can be reduced at the expense of the other
(namely in the proof of Lemma~\ref{final}).
For Theorem~\ref{thm:main}, we have balanced these bounds so that both are
of the same leading order in $d$, i.e., $d^{-\tau\xi_1}$ in 
part~\ref{theoremA} and $d^{-\tau\xi_2}$ in part~\ref{theoremB}.

\begin{remark}\normalfont \label{uniformity}
Because the error bounds in Theorem~\ref{thm:main} depend on
$Z$ only through the constants that occur in \ref{c.Sk} and \ref{c.density},
the theorem a fortiori {\em holds uniformly} over the class of all distributions
for $Z$ that satisfy \ref{c.Sk} and \ref{c.density}. For example:
Fix constants $\eps$, $\alpha$, $\beta$ and $\xi$ as in 
\ref{c.Sk},
fix $D$ as in \ref{c.density}, and write
${\cal Z}$ for the class of all random $d$-vectors $Z$ that satisfy
the bounds \ref{c.Sk}.(\ref{c.sup}-\ref{c.monom2}) 
and \ref{c.density} for $k=4$, that have a
Lebesgue density, and that are centered and standardized.
Then, for each $Z \in {\cal Z}$ and for each $p<d$, there exits
a Borel set $\mathbb G \subseteq \mathcal V_{d,p}$ (that depends on $Z$), so that
\eqref{mean}, \eqref{variance}
and also 
$\nu_{d,p}(\mathbb G^c)$ are bounded as in Theorem~\ref{thm:main}\ref{theoremB}.
Similar considerations also apply, mutatis mutandis, to 
Theorem~\ref{thm:main}\ref{theoremA} and to the following corollary.
\end{remark}

\begin{remark}\normalfont
Our results provide conditions under which conditional means are approximately linear and
conditional variances are approximately constant, provided that $p/\log{d}$ is small. Theorem~\ref{thm:main} 
provides such a statement for a fixed distribution of $Z$ and for many $B$. By a slight change of perspective,
this also leads to a similar statement that holds for fixed $B$ and many distributions of $Z$, cf. \citep{Ste16c}.
We can not deal with a fixed matrix $B$ and a fixed distribution of $Z$ with our methods. Whether, say,
the conditional variance is approximately constant for given $B$ and $Z$ depends on the particulars of $B$ and $Z$, 
irrespective of $p$ and $d$. A few trivial examples, however, are well known. For instance, if the distribution of $Z$ is 
spherically symmetric, then the conditional expectation of $Z$ given $B'Z$ is exactly linear for every matrix $B$. Moreover, 
the conditional expectation is linear and the conditional variance is constant if the components of $Z$ are independent 
and $B = (e_{j_1},\dots, e_{j_p})$, where $e_j$ is the $j$-th element of the standard basis in $\R^d$. See also 
Section~2.3.4. in Chapter~2 of \citep{Ste15a} for a non-trivial example with $p=1$ and $d=2$.
\end{remark}

\begin{corollary}\label{asymptotics}
For each $d$, consider a random $d$-vector $Z^{(d)}$ that has a Lebesgue density
and that satisfies $\E Z^{(d)} = 0$ and $\E Z^{(d)} Z^{(d)'} = I_d$. 
And for each $d$, 
suppose that \ref{c.Sk}.(\ref{c.sup}-\ref{c.monom2}) 
and \ref{c.density} hold with $Z^{(d)}$ replacing $Z$
and with $k=4$, such that the constants $\eps$, $\alpha$,
$\beta$, and $D$ in these bounds do not depend on $d$, while
the constant $\xi = \xi_d$ in \ref{c.Sk} may depend on $d$
as long as $d^{-\xi_d} \to 0$ as $d\to\infty$. Moreover, consider
a sequence of integers $p_d < d$ such that $p_d / (\xi_d \log d) \to 0$.
Then Theorem~\ref{thm:main}\ref{theoremB} applies for each $d$, 
with $Z^{(d)}$ and $p_d$ replacing $Z$ and $p$, respectively, 
and the error bounds provided in
the theorem go to zero as $d\to\infty$.
\end{corollary}

Corollary~\ref{asymptotics} provides an asymptotic version of
Theorem~\ref{thm:main}\ref{theoremB}. Similarly, an asymptotic version of
Theorem~\ref{thm:main}\ref{theoremA} can also be obtained, mutatis
mutandis.
This provides a direct extension of Theorem~2.1 of
\citep{Lee12a} from the case $p=1$ covered in that reference to the
case where $p>1$, also allowing for $p$ to grow with $d$. [Indeed,
it is elementary to verify that conditions (t1) and (t2) with $k=4$
in \citep{Lee12a} imply that the conditions of the corollary are
satisfied with $\eps = 0$ and for some sequence $\xi_d$ such that
$d^{-\xi_d}\to 0$ as $d\to\infty$. And if (t1) and (t2) hold with $k=2$,
one obtains conditions that imply an asymptotic version of
Theorem~\ref{thm:main}\ref{theoremA}.]

If Corollary~\ref{asymptotics} applies with constants $\xi_d$ 
satisfying $\xi_d \to \xi_\infty > 0$ (e.g.,  in the case
where the $\xi_d$ do not depend on $d$, which is also discussed in Example~\ref{Ex:1}), the corollary's requirement on $p_d$
reduces to $p_d = o(\log d)$. In that case, the bounds on $\nu_{d,p}(\mathbb G^c)$
in Theorem~\ref{thm:main} are of polynomial order in $d$.
But if Corollary~\ref{asymptotics} applies with $\xi_d \to 0$,
then the stronger requirement $p_d = o( \xi_d \log d)$ is needed,
and the bounds on $\nu_{d,p}(\mathbb G^c)$ 
in Theorem~\ref{thm:main} can be of slower order in $d$.
Note that $d^{-\xi_d}\to 0$ entails that $\xi_d \log d \to \infty$,
so that the constant sequence $p_d = p$ always satisfies
the growth condition in Corollary~\ref{asymptotics}.


\section{Examples and extensions}
\label{sec:ExEx}

\subsection{Examples}
\label{sec:ex}

In this section we discuss a few simple examples of multivariate distributions for which our assumptions~\ref{c.Sk} and \ref{c.density} are satisfied and explicit values for the quantities $\eps$ and $\xi$ can be given. First, we consider the case of a product distribution on $\R^d$ with moments of sufficiently high order and bounded component densities. For the proof we refer the reader to Example~A.1 in \cite{Lee12b}.

\begin{example}[Leeb 2013]
\label{Ex:1}
Suppose that the random $d$-vector $Z = (z_1,\dots, z_d)'$ has independent components and satisfies $\E Z = 0$, $\E ZZ' = I_d$, and fix $k\in\N$. 

\begin{list}{\thercnt}{
        \usecounter{rcnt}
        \setlength\itemindent{5pt}
        \setlength\leftmargin{0pt}
        \setlength\partopsep{0pt} 
        }
\item \label{Ex:1A}
If $\E |z_i|^{4k+4} \le \mu_{4k+4}$, for some universal constant $\mu_{4k+4}>0$ and for all $i=1,\dots, d$, then the bounds in \ref{c.Sk}.(\ref{c.sup}-\ref{c.monom1}) hold with $k$ as chosen here, with $\eps=\xi=1/2$ and the constants $\alpha$ and $\beta$ depend only on $k$ and $\mu_{4k+4}$.

\item \label{Ex:1B}
If $\E |z_i|^{2k+1} \le \mu_{2k+1}$, for some universal constant $\mu_{2k+1}>0$ and for all $i=1,\dots, d$, then the bounds in \ref{c.Sk}.(\ref{c.monom2}) hold with $k$ as chosen here, with $\xi=1/2$ and the constant $\beta$ depends only on $k$ and $\mu_{2k+1}$.

\item\label{Ex:1C}
If all the marginal Lebesgue densities of the components of $Z$ exist and are bounded by a constant $D\ge1$ then Condition~\ref{c.density} holds true for the same constant $D$ and every value of $k\in\{1,\dots, d\}$.

\end{list}
\end{example}

From Example~\ref{Ex:1} we see, in particular, that if the random vector $Z$ has independent components with bounded densities and bounded $12$-th marginal moments, then the bounds of Theorem~\ref{thm:main}\ref{theoremA} hold, with $\xi_1 = 1/6$ (note that $k$ has to be chosen as $k=2$ here). If the components of $Z$ even have $20$ marginal moments bounded by a universal constant, then also the bounds of Theorem~\ref{thm:main}\ref{theoremB} hold, with $\xi_2 = 1/10$ (in this case $k=4$).

The assumptions of Theorem~\ref{thm:main}, however, are not limited to product distributions, as the following examples show. See Appendix~\ref{sec:AppExEx} in the supplementary material \citep{Ste16b} for the proofs.

\begin{example}
\label{Ex:2}
Suppose that the random $d$-vector $Z$ satisfies $\E Z = 0$ and $\E ZZ' = I_d$.

\begin{list}{\thercnt}{
        \usecounter{rcnt}
        \setlength\itemindent{5pt}
        \setlength\leftmargin{0pt}
        \setlength\partopsep{0pt} 
        }
\item \label{Ex:2A}
If $R$ is a fixed $d\times d$ orthogonal matrix and $Z$ satisfies any of the bounds \ref{c.Sk}.(\ref{c.sup},\ref{c.monom1},\ref{c.monom2}) or \ref{c.density} for some values of $k$, $\alpha$, $\beta$, $\eps$ and $\xi$, then the random vector $Z^* = RZ$ satisfies the same bound with the same constants. 

\item \label{Ex:2B}
If $r$ is a scalar random variable taking values in $\{-1,1\}$ that is independent of $Z$, and $Z$ satisfies any of the bounds \ref{c.Sk}.(\ref{c.sup},\ref{c.monom1},\ref{c.monom2}) or \ref{c.density} for some values of $k$, $\alpha$, $\beta$, $\eps$ and $\xi$, then the random vector $Z^*=rZ$ satisfies the same bound with the same constants. 
\end{list}
\end{example}

Examples~\ref{Ex:1} and \ref{Ex:2} can be combined to produce many multivariate distributions with dependent components that still satisfy the assumptions of Theorem~\ref{thm:main}. For instance, if $Z$ has independent non-Gaussian components with moment and density bounds as in Example~\ref{Ex:1} and $R$ is orthogonal with no zero entry, then, by the Darmois-Skitovich Theorem, $Z^*=RZ$ can not have independent components. Alternatively, if $Z = (z_1,\dots, z_d)'$ is as in Example~\ref{Ex:1} and such that, say, the first two components have non-symmetric distributions, then the first two components of $Z^*= rZ = (z_1^*,\dots, z_d^*)'$, for some non-degenerate random variable $r$ with values in $\{-1,1\}$, may be dependent. Indeed, for example, take $z_1 \thicksim z_2 \thicksim \text{Exp}(1)-1$ and note that $\P(z_2^*<-1 | z_1^*>1) = 0 \ne \P(z_2^*<-1)$.

As our last example we discuss a specific case of a spherical distribution. Recall that every spherically symmetric distribution with independent components must be Gaussian. So every spherical non-Gaussian distribution constitutes an example of a multivariate distribution with dependent components. Also, if $Z$ is spherical, then $\E[Z\|B'Z] = BB'Z$, almost surely, for every $B\in\mathcal V_{d,p}$. Hence, the following example is only of interest in connection with Theorem~\ref{thm:main}\ref{theoremB} on the conditional second moment of $Z$, since Theorem~\ref{thm:main}\ref{theoremA} is trivially true in this case.

\begin{example}
\label{Ex:3}
If $Z$ is uniformly distributed on the $d$-ball of radius $\sqrt{d+2}$, then $\E Z = 0$ and $\E ZZ' = I_d$. Moreover, for $k\in\{2,4\}$, at least for all sufficiently large $d$, $Z$ satisfies Conditions~\ref{c.Sk} and \ref{c.density} with constants $\eps=\xi=1/2$, $D=1$, and constants $\alpha$ and $\beta$ that depend only on $k$.
\end{example}

Finally, it is worth mentioning that in the case of spherically symmetric $Z$ the structure of the set $\mathbb G$ from Theorem~\ref{thm:main} simplifies dramatically. Indeed, from Theorem~\ref{thm:main}\ref{theoremC} we see that if $Z$ is spherical, then $\mathbb G$ is both left and right-invariant under the action of the appropriate orthogonal groups and thus is either empty or equal to the whole Stiefel manifold $\mathcal V_{d,p}$.

\subsection{Improved bounds}
\label{sec:bounds}

At the current state of research we can not say if the bounds provided by Theorem~\ref{thm:main} are tight, or at least if they are of the optimal rate in $d$, in the sense that this rate is achieved for some multivariate distribution satisfying conditions~\ref{c.Sk} and \ref{c.density}. However, there are certain distributions for which the bounds of the theorem can be improved substantially.\footnote{Moreover, for each specific distribution, there are typically subsets of the set $\mathbb G$ from the theorem, for which the probabilities in \eqref{mean} and \eqref{variance} are substantially smaller than their respective upper bounds \eqref{eq:boundMean} and \eqref{eq:boundVariance}. For instance, if $Z$ has independent components and the columns of $B$ are elements of the standard basis in $\R^d$, then both probabilities in \eqref{mean} and \eqref{variance} are equal to zero, for all $t>0$.}   

First, consider the bounds \eqref{eq:boundMean} and \eqref{eq:boundVariance}, which are only of logarithmic order in $d$. As mentioned in Section~\ref{sec:main}, they can be improved considerably if one imposes an appropriate condition on $Z$. Here, we only consider \eqref{eq:boundMean} as an example. This bound is derived in the proof of Lemma~\ref{final}\ref{final.i} by the following simple argument involving the cut-off point $M_d = \sqrt{3\xi_1(1-\tau)(\log d)/\gamma_1}$. For $t>0$, 
\begin{align}
&\P\left( \|\E[Z\|B'Z] - BB'Z\| > t \right)\notag\\
&\quad \le \quad
\P\left(\|\E[Z\|B'Z] - BB'Z\| > t , \|B'Z\| \le M_d\right) 
\,\, + \,\, \P\left( \|B'Z\| > M_d \right) \notag\\
&\quad \le\quad
\frac{1}{t} 
\int\limits_{\|x\| \le M_d} 
	\|\E[Z\|B'Z=x] - Bx\|
\,d\P_{B'Z}(x) \,\, +\,\,
p/M_d^2.	\label{eq:Exbound1}
\end{align}
In the proof of Theorem~\ref{thm:main}\ref{theoremA} we choose $\mathbb G\subseteq \mathcal V_{d,p}$ such that for $B\in\mathbb G$ the bound in \eqref{eq:Exbound1} turns into \eqref{eq:boundMean}, while, at the same time, $\nu_{d,p}(\mathbb G^c)$ is bounded as in \eqref{eq:boundG1}. Of course, using Markov's inequality to bound $\P\left( \|B'Z\| > M_d \right)$ in the preceding display is far from optimal if we have more information on the tails of $Z$. 

Suppose now that the random vector $Z$, in addition to the assumptions of Theorem~\ref{thm:main}\ref{theoremA}, also satisfies the sub-Gaussian tail condition
\begin{align}\label{eq:subGauss}
\E \exp(\alpha'Z) \le \exp(\|\alpha\|^2\sigma^2/2),
\end{align}
for every $\alpha \in\R^d$ and some constant $\sigma>0$.\footnote{This is satisfied, for instance, if $Z$ has independent components which all satisfy the one dimensional analogue of \eqref{eq:subGauss} with the same value of $\sigma$.} Under this condition, the tail inequality for quadratic forms by \cite{Hsu12} applies and yields
$$
\P(\|B'Z\|^2> \sigma^2(p + 2\sqrt{p}s + 2s^2)) \le e^{-s^2},
$$
for all $s>0$. Suppose that $p < M_d^2/(8\sigma^2)$. Since this restriction also entails that $p < M_d^2/\sigma^2$, the equation $\sigma^2(p + 2\sqrt{p}s + 2s^2) = M_d^2$ has a real positive solution $s_0 = -\sqrt{p}/2 + \sqrt{M_d^2/(2\sigma^2) - p/4}$. Thus, after expanding the square and rearranging terms, we obtain
\begin{align*}
s_0^2 &\quad=\quad \frac{M_d^2}{2\sigma^2} - \sqrt{p\left(\frac{M_d^2}{2\sigma^2} - \frac{p}{4}\right)}
\quad\ge\quad
\frac{M_d^2}{4\sigma^2} \quad=\quad 
 (\log{d}) \frac{3}{4} \frac{\xi_1(1-\tau)}{\sigma^2 \gamma_1},
\end{align*}
where we have used our restriction on $p$ again. Hence,
\begin{align*}
&\P(\|B'Z\|^2> M_d^2) \quad\le\quad e^{-s_0^2} 
\quad\le\quad d^{-\frac{3}{4} \frac{\xi_1(1-\tau)}{\sigma^2 \gamma_1}},
\end{align*}
and we have managed to replace the term in \eqref{eq:boundMean} that is only of logarithmic order in $d$ by something that is decreasing polynomially in $d$. However, since the squared cut-off point $M_d^2$ is only of logarithmic order in $d$, the condition that $p < M_d^2/(8\sigma^2)$ still requires $p/\log{d}$ to be small. At the moment, we do not see a way how to increase the cut-off point to polynomial order in $d$ without simultaneously ruining the bound in \eqref{eq:boundG1}.

Concerning the bounds \eqref{eq:boundG1} and \eqref{eq:boundG2}, we believe that polynomial rates in $d$ of arbitrarily high order can be achieved under more restrictive assumptions than those maintained here and upon using a more elaborate method of proof. First results in that direction, regarding only the conditional expectation, are in preparation, cf. \cite{Mil15}.

\subsection{The case of a general covariance matrix}
\label{sec:cov}

The proof of Theorem~\ref{thm:main} crucially relies on the assumptions that $\E Z = 0$ and $\E ZZ' = I_d$.
However, Theorem~\ref{thm:main}, as it stands, can already be used to generalize substantially beyond the mean zero and unit covariance case. In particular, we can provide a large class of positive definite covariance matrices such that for each element $\Sigma$ of that class the conclusions of Theorem~\ref{thm:main} remain valid, provided, of course, that all the relevant quantities are modified to reflect the general covariance structure $\Sigma$. The key to this extension is the following observation.

If $Y$ is Gaussian with mean $\mu\in\R^d$ and positive definite covariance matrix $\Sigma$, and $A\in\mathcal V_{d,p}$, then $\E[Y\| A'Y] = \mu + \Sigma^{1/2} P_{\Sigma^{1/2}A} \Sigma^{-1/2}(Y-\mu)$ and $\E[(Y-\mu)(Y-\mu)'\|A'Y] = \Sigma^{1/2}\left[I_d-P_{\Sigma^{1/2}A} + P_{\Sigma^{1/2}A}\Sigma^{-1/2}(Y-\mu)(Y-\mu)'\Sigma^{-1/2}P_{\Sigma^{1/2}A}\right]\Sigma^{1/2}$, where $P_{\dots}$ is the projection matrix corresponding to the column span of the matrix in the subscript. These are our target quantities. 
Now assume that $Y$ is not necessarily Gaussian but
satisfies
$Y = \mu + \Sigma^{1/2}Z$, with $Z$ as in Theorem~\ref{thm:main}. 
One easily verifies that
\begin{align*}
\E[Y\|A'Y] - (\mu + \Sigma^{1/2} P_{\Sigma^{1/2}A} \Sigma^{-1/2}(Y-\mu))
=
\Sigma^{1/2} \left(\E[Z\|B'Z] - BB'Z\right),
\end{align*}
and
\begin{align*}
\E[(Y-\mu)&(Y-\mu)'\|A'Y] \\
&\quad- \Sigma^{1/2}\left(I_d-P_{\Sigma^{1/2}A} + P_{\Sigma^{1/2}A}\Sigma^{-1/2}(Y-\mu)(Y-\mu)'\Sigma^{-1/2}P_{\Sigma^{1/2}A}\right)\Sigma^{1/2}\\
&=
\Sigma^{1/2} \left(\E[ZZ'\|B'Z] - (I_d-BB' + BB'ZZ'BB')\right)\Sigma^{1/2},
\end{align*}
where $B = \Sigma^{1/2}A(A'\Sigma A)^{-1/2} \in\mathcal V_{d,p}$. Ideally, the norm of these quantities should become small if $d$ is large. Ignoring the additional scaling by the matrix $\Sigma^{1/2}$ of these error terms\footnote{Whether the scaling by $\Sigma^{1/2}$ matters depends on the specific context of application for these results. Also, the problem can always be circumvented by imposing a boundedness assumption on $\|\Sigma\|$. However, in the context of \cite{Ste16c}, for example, no such bound is required.}, there remains the question of whether the theorem also applies to
\begin{align}
\P\left(\Big\|\E[Z\|B'Z] - BB'Z\Big\|>t\right) \label{eq:mean2}
\end{align}
and
\begin{align}
\P\left(\Big\|\E[ZZ'\|B'Z] - (I_d-BB' + BB'ZZ'BB')\Big\|>t\right), \label{eq:variance2}
\end{align}
instead of \eqref{mean} and \eqref{variance}, i.e., if $B = B(\Sigma,A) \in \mathbb G$. 
This raises two questions: For given $\Sigma$, how large is the collection 
of `good' matrices $A$, i.e., how large is the set of $A$
for which $B(\Sigma,A) \in \mathbb G$?
And: How large is the family of matrices $\Sigma$ for which
the collection of `good' matrices $A$ is large?
The latter question is answered by the next result, the proof of which is deferred to Appendix~\ref{sec:AppExEx} 
in the supplement. 

\begin{proposition}
\label{prop:Sigma}
If $Z$ satisfies the assumptions of Theorem~\ref{thm:main}\ref{theoremA} (or Theorem~\ref{thm:main}\ref{theoremB}) and $\mathbb G\subseteq \mathcal V_{d,p}$ is the corresponding subset of the Stiefel manifold, then, for each diagonal positive definite matrix $\Lambda$, there exists a collection $\mathbb U(\Lambda) = \mathbb U(\mathbb G,\Lambda) \subseteq \mathcal O_d$ of orthogonal matrices, such that the sets
$$
\mathbb S := \mathbb S(\mathbb G) := \{ U\Lambda U' : \Lambda = \diag(\lambda_i) >0, U \in \mathbb U(\mathbb G,\Lambda)\}
$$
and
$$
\mathbb J(\Sigma) := \mathbb J(\Sigma,\mathbb G) := \{A \in \mathcal V_{d,p} : \Sigma^{1/2}A(A'\Sigma A)^{-1/2} \in \mathbb G\},
$$
have the following properties:

$$
\sup_{\Lambda:\Lambda=\diag(\lambda_i)>0}\nu_{d,d}(\mathbb U^c(\Lambda)) \quad\quad \text{and} \quad\quad \sup_{\Sigma\in\mathbb S}\nu_{d,p}(\mathbb J^c(\Sigma))
$$ 
are bounded by the square root of the right-hand-side of \eqref{eq:boundG1} (resp. \eqref{eq:boundG2}). 
By definition, if $\Sigma$ is any positive definite covariance matrix and $A\in\mathbb J(\Sigma)$, then \eqref{eq:mean2} (resp. \eqref{eq:variance2}) is bounded by \eqref{eq:boundMean} (resp. \eqref{eq:boundVariance}) for every $t>0$.
\end{proposition}

To understand the message of Proposition~\ref{prop:Sigma},
suppose for now that the assumptions of 
Theorem~\ref{thm:main}\ref{theoremA} are satisfied. Then
the set $\mathbb J(\Sigma)$ is constructed such
that the following holds:
If $\Sigma$ 
is any positive definite covariance matrix and $A$ is taken from the 
collection $\mathbb J(\Sigma)$, then, for 
$B = \Sigma^{1/2} A (A' \Sigma A)^{-1/2}$, the expression in 
\eqref{eq:mean2} is bounded by \eqref{eq:boundMean}.
In other words, $\mathbb J(\Sigma)$ is a collection of `good' matrices $A$
as discussed just before the proposition.
Now Proposition~\ref{prop:Sigma} shows 
that $\mathbb J(\Sigma)$ is large provided that
$\Sigma \in \mathbb S$, and also that
the set $\mathbb S$ itself is large.
Similar considerations apply, mutatis mutandis, to the
conditional variance under the assumptions of 
Theorem~\ref{thm:main}\ref{theoremB}.
In short, for a large 
class of $d$-dimensional distributions $Z$ (cf. conditions~\ref{c.Sk} 
and \ref{c.density}), for a large set of covariance matrices $\Sigma$ 
(given by $\mathbb S$) and for most matrices $A$ from the Stiefel manifold 
(those contained in $\mathbb J(\Sigma)$), the first two conditional 
moments of $Y = \mu + \Sigma^{1/2}Z$ given $A'Y$ are close to what they 
would be in the Gaussian case, all provided that $p/\log{d}$ is small.


\section{Proof of Theorem~\ref{thm:main}}
\label{proof:main}

The rest of the paper and the on-line supplementary material comprise the proof
of Theorem~\ref{thm:main}. The basic strategy of the proof is non-standard and
is described in this section. To implement this strategy, we need to deal with several
intricate technical challenges. But those can be handled by standard methods from
multivariate analysis and probability theory. To keep the main paper short, such technical details 
are relegated to the on-line supplementary material.
Our arguments have the same basic structure as those used in \citep{Lee12a}.
To prove Theorem~\ref{thm:main}, however,
the arguments from \citep{Lee12a} require substantial
extensions and modifications, because many of the arguments 
used in that reference are of an asymptotic nature and do not
provide explicit error bounds, and because all of these arguments
rely heavily on the assumption that $p$ is fixed and equal to $1$.

\subsection{Two crucial bounds}
\label{s3.1}

Throughout, fix $d \in \N$ and let $Z$ be as in Theorem~\ref{thm:main}, i.e.,
a random $d$-vector that has a Lebesgue density and that
is standardized so that $\E Z = 0$ and $\E Z Z' = I_d$.
[The particular assumptions of part~\ref{theoremA} and~\ref{theoremB}
of Theorem~\ref{thm:main} will be imposed as needed later.]
We will study the following quantities:
For a positive integer $p < d$, for $x\in\R^p$, and for 
$B\in \mathcal V_{d,p}$, set $\mu_{x|B} = \E [Z \| B'Z = x]$,
$\Delta_{x|B} = \E[ZZ'\|B'Z=x] - (I_d + B(xx'-I_p)B')$, and 
$h(x|B) = \E\left[ \frac{f(W^{x|B})}{\phi(W^{x|B})}\right]$, 
where $f = f_Z$ is a Lebesgue density of $Z$, $W^{x|B} = Bx + (I_d-BB')V$, and 
$\phi$ denotes a Lebesgue density of $V\thicksim N(0,I_d)$. 
Note that these quantities, if 
considered as functions with domain $\R^p \times \mathcal V_{d,p}$, 
can be chosen to be measurable; cf. Lemma~\ref{measurability}.
Finally,  write $\mathcal S_{M,p}$ for the closed ball of radius $M$
in $\R^p$, i.e., $\mathcal S_{M,p} = \{ x \in \R^p:\: \|x\|\leq M\}$.

We now introduce two bounds which will play an essential role
in the proof of Theorem~\ref{thm:main}. In each bound,
the quantity of interest, which will be introduced shortly,
will be bounded by an expression of the form
\begin{align} \label{bound}
p^{2k+1+\eps} e^{g M^2} \left(2D\sqrt{\pi e}\right)^{pk} 
	d^{-\min\{\xi,\eps/2+1/4,1/2\}} \kappa
\end{align}
for some even integer $k$,
where the precise value of the constants in the bound 
will depend
on the context, i.e., these constants will be chosen as needed later.

The first crucial bound implies the first part of Theorem~\ref{thm:main}:
Under the assumptions of the  Theorem~\ref{thm:main}\ref{theoremA},
we will show that
\begin{align}
\sup_{x\in \mathcal S_{M,p}} 
\int 
	\left[ 
		\|\mu_{x|B}\|^2 - \|x\|^2 
	\right] h(x|B)^2 \, 
d\nu_{d,p}(B) \hspace{.5cm} \label{th1a}
\end{align}
is bounded by \eqref{bound}
for $k=2$, for every $M>1$ and every $p\in\N$ such that 
$d> \max\{4(k+p+1)M^4, 2k +p(2k+2)2^{k+3},p^2\}$,
where $\kappa = \kappa_1\geq 1$ is a constant that depends only on
$\alpha$ and $\beta$, and where $g=g_1$ is a global constant.
The remaining constants occurring here, i.e., $\eps$, $\alpha$, $\beta$, 
$\xi$, and $D$, are those that appear in the bounds 
\ref{c.Sk}.(\ref{c.sup}-\ref{c.monom1}) and \ref{c.density} 
imposed by Theorem~\ref{thm:main}\ref{theoremA}.
Once that statement has been derived, the proof of
Theorem~\ref{thm:main}\ref{theoremA}
is easily completed by standard arguments (that are detailed in
Lemma~\ref{final}\ref{final.i}).

The second crucial bound similarly delivers the second part of 
Theorem~\ref{thm:main}:
Under the assumptions of Theorem~\ref{thm:main}\ref{theoremB},
we will show that \eqref{th1a} and
\begin{align}
\sup_{x\in \mathcal S_{M,p}} 
\int \trace{\Delta_{x|B}^k} h(x|B)^k \, d\nu_{d,p}(B) 
	\label{th2a}
\end{align}
are both bounded by \eqref{bound}  
for $k=4$, for 
every $M>1$ and every $p\in \N$ such that 
$d> \max\{4(k+p+1)M^4, 2k +p(2k+2)2^{k+3},p^2\}$,
where $\kappa=\kappa_2\geq 1$ depends only on $\alpha$ and $\beta$,
and where $g =g_2$ is  a global constant.
Again, the remaining constants $\eps$, $\alpha$, $\beta$, 
$\xi$, and $D$ 
are those that appear in the bounds \ref{c.Sk} and \ref{c.density}
imposed by Theorem~\ref{thm:main}\ref{theoremB}.
From this statement, standard arguments complete the proof of
Theorem~\ref{thm:main}\ref{theoremB};
cf. Lemma~\ref{final}\ref{final.ii}.

It turns out that in order to derive the upper bounds 
for \eqref{th1a} and \eqref{th2a},  it will be instrumental to show that
\begin{align}
&
\sup_{x\in \mathcal S_{M,p}} 
\int 
	\left[ 
		h(x|B) - 1
	\right]^2 
~d\nu_{d,p}(B) \label{th1b}
\end{align}
is finite. In particular,  we will need to establish finiteness of
\eqref{th1b} for every $M$ and $p$ as in \eqref{th1a} to derive
the desired bound on \eqref{th1a}, and for every $M$ and $p$ as
in \eqref{th2a} for the bound on \eqref{th2a}.
We will in fact show more than that, namely that \eqref{th1b} is 
also bounded by \eqref{bound}, 
under the assumptions of Theorem~\ref{thm:main}\ref{theoremA} and
for constants as in \eqref{th1a}, and also 
under the assumptions of Theorem~\ref{thm:main}\ref{theoremB} and
for constants as in \eqref{th2a}.

We pause here for a moment to discuss a weaker 
version of Theorem~\ref{thm:main} which also allows us to 
better appreciate the importance of \eqref{th1b}
(the exact role of \eqref{th1b} in 
the main argument will become apparent later, after Proposition~\ref{p1}):
Assume in this paragraph that \eqref{th1a}, \eqref{th2a}, 
and \eqref{th1b} are all bounded by \eqref{bound} with $k=4$,
for each $M>1$, and for {\em each sufficiently large} $d$.
[The other constants in the bound, i.e., 
$p$, $\eps$, $g$, $D$, $\xi$ and $\kappa$,
are assumed to be fixed, independent of of $d$ here.]
Under this assumption, we immediately obtain the following weaker
version of Corollary~\ref{asymptotics}: For each $x\in\R^p$, we have
\begin{align*}
\|\E [Z\|\mathbf B'Z = x]  - \mathbf B x \|^2 
&\quad\stackrel{p}{\longrightarrow} \quad 0, \\
\|\E [ZZ'\|\mathbf B'Z = x] - (I_d + \mathbf B(xx' - I_p)\mathbf B') \| 
&\quad\stackrel{p}{\longrightarrow} \quad 0 
\end{align*}
as $d\to\infty$, if $\mathbf B$ is a random matrix that is uniformly
distributed on ${\mathcal V}_{d,p}$.
After noting that 
$\|\E [Z\|\mathbf B'Z = x]  - \mathbf B x \|^2$ can also be written as
$\|\E [Z\|\mathbf B'Z = x]\|^2 - \|x\|^2$, 
this is an easy consequence of Markov's inequality and Slutzky's 
Lemma.\footnote{
	A proof of the first statement in the preceding display
	was already sketched in \citep{Hal93a} as an immediate generalization 
	of the case where $p=1$ proved therein. See also \citep{Lee12a} 
	for further discussion of that result. 
}
[Choose $M\ge \|x\|$ and observe that the upper bound \eqref{bound} 
converges to zero as $d\to\infty$, 
so that also the three quantities in \eqref{th1a}, \eqref{th2a},
and \eqref{th1b} converge to zero. 
Now convergence of \eqref{th1b} to zero entails
that $h(x|{\mathbf B})$ converges to one in squared mean.
Convergence of \eqref{th1a} to zero implies 
that 
$\left[ \|\E[Z\|\mathbf B'Z = x]\|^2 - \|x\|^2 \right] h(x|\mathbf B)^2$ 
converges to zero in expectation.
Similarly, convergence of \eqref{th2a} to zero implies 
that $\trace{\Delta_{x|\mathbf B}^4} h(x|\mathbf B)^4$ 
converges to zero in expectation.
Because the involved random variables are all non-negative,
the first relation in the preceding display follows from Markov's inequality
and Slutzky's Lemma.
The second relation follows in a similar fashion
upon observing that the symmetry of $\Delta_{x|\mathbf B}$
entails that $\|\Delta_{x|\mathbf B}\|^4$ is bounded from above by 
$\trace \Delta_{x|\mathbf B}^4$.]

In this subsection, we have seen that to prove 
Theorem~\ref{thm:main}\ref{theoremA},  it suffices to show, 
under the assumptions maintained there, that
both \eqref{th1a} and \eqref{th1b} are bounded
by \eqref{bound} for constants as in \eqref{th1a}.
And, similarly, to prove Theorem~\ref{thm:main}\ref{theoremB},
it remains to show,
under the assumptions maintained there, that
\eqref{th1a}, \eqref{th2a} and \eqref{th1b} are all bounded
by \eqref{bound} for constants as in \eqref{th2a}.

\subsection{Changing the reference measure}
\label{s3.2}

Throughout the following, set
\begin{align}\label{Wj}
W_j \quad=\quad  \mathbf B x + (I_d- \mathbf B \mathbf B')V_j,
\end{align}
for $j=1,\dots, k$,
where ${\mathbf B}$, $V_1, \dots, V_k$ are independent such that
$\mathbf B$ is a random $d\times p$ matrix with distribution $\nu_{d,p}$
and such that each of 
the $V_i$ is distributed as $N(0,I_d)$. 
We call $W_1,\dots, W_k$ the `rotational clones', in analogy 
to the name `rotational twins' that  the authors of \citep{Hal93a} use for  
the pair $W_1, W_2$ in case $p=1$.
With this, we may re-write the integral in \eqref{th1b} as
\begin{align}
&\int 
	\left(
		\E\left[ \frac{f(W^{x|B})}{\phi(W^{x|B})}\right] 
	\right)^2 
	- 2 \E\left[ \frac{f(W^{x|B})}{\phi(W^{x|B})}\right] + 1 
\,~d\nu_{d,p}(B) \notag\\
& \quad =\quad
\int 
	\E\left[ \frac{f(Bx + (I_d-BB')V_1)}{\phi(Bx + (I_d-BB')V_1)}\right] 
	\E\left[ \frac{f(Bx + (I_d-BB')V_2)}{\phi(Bx + (I_d-BB')V_2)}\right] 
~d\nu_{d,p}(B) \notag\\
& \quad \quad\quad
- 2 \int 
	\E\left[ \frac{f(Bx + (I_d-BB')V_1)}{\phi(Bx + (I_d-BB')V_1)}\right] 
~d\nu_{d,p}(B) + 1 \notag\\
&\quad =\quad
\E\left[ \frac{f(W_1)}{\phi(W_1)} \frac{f(W_2)}{\phi(W_2)} - 1 \right] - 
	2 \E\left[ \frac{f(W_1)}{\phi(W_1)} - 1 \right],  \label{th1bb}
\end{align}
provided that the expected values in \eqref{th1bb} are all finite.
And, clearly, if both expected values in \eqref{th1bb} are bounded by
\eqref{bound} in absolute value, then \eqref{th1bb} is bounded
by three times the expression in \eqref{bound}.
To establish the desired bounds on \eqref{th1a} and \eqref{th2a} it will 
be convenient to also express the integrals in \eqref{th1a} and \eqref{th2a} 
in terms of the $W_i$. 
This can be accomplished by virtue of the following proposition.

\begin{proposition}\label{p1}
Fix $d\ge p \geq 1$, and
consider a random $d$-vector $Z$ with Lebesgue density $f$.
Let $V \sim N(0, I_d)$, and write $\phi(\cdot)$ 
for a Lebesgue density of 
$V$. Moreover, for a fixed $d\times p$-matrix $B \in \mathcal V_{d,p}$ 
and for any $x\in \R^p$, set 
$W^{x|B} = Bx  + (I_d-BB')V$.
Then the function $h(\cdot|B) : \R^p \to \bar{\R}$ defined by 
$$
h(x|B) \quad=\quad \E\left[\frac{f(  W^{x|B})}{\phi(  W^{x|B})}\right]
$$
for $x\in \R^p$
is a density of $B'Z$ with respect to the $p$-variate standard Gaussian
measure (i.e., $h(x|B) \phi_p(x)$ is a Lebesgue density of $B'Z$
if $\phi_p$ denotes a $ N_p(0,I_p)$-density).
Moreover, if $\Psi:\R^d \to\R$ is such that $\Psi(Z)$ is
integrable, then a conditional expectation $\E[\Psi(Z)\|B'Z=x]$
of $\Psi(Z)$ given $B'Z=x$ satisfies
$$
\E\Bigg[ \Psi(Z)\Bigg\|B'Z =x\Bigg] \;\; h(x|B)
\quad=\quad
	\E\Bigg[ \Psi(W^{x|B}) \frac{f(W^{x|B})}{\phi(W^{x|B})} 
	\Bigg] 
$$
whenever $x\in \R^p$ is such that $h(x|B) < \infty$.
\end{proposition}

Note that this proposition applies under the assumptions
of both  parts of Theorem~\ref{thm:main}.
Assume therefore that Proposition~\ref{p1} is applicable throughout the
rest of this subsection.
The integral in \eqref{th1a} can then be re-written as
\begin{align}
&\int 
	\left[ 
		\|\mu_{x|B}\|^2 - \|x\|^2 
	\right] h(x|B)^2 
~d\nu_{d,p}(B) \notag\\
& \quad =\quad
\int 
	\|\mu_{x|B}h(x|B)\|^2 
~d\nu_{d,p}(B) -
\|x\|^2
\int 
	h(x|B)^2 
~d\nu_{d,p}(B) \notag\\
&\quad =\quad
\int 
	\E\left[ W^{x|B}  \frac{f(W^{x|B})}{\phi(W^{x|B})} \right]'
	\E\left[ W^{x|B}  \frac{f(W^{x|B})}{\phi(W^{x|B})} \right]
~d\nu_{d,p}(B) \notag\\
& \quad \quad \quad
- \|x\|^2 \E\left[ \frac{f(W_1)}{\phi(W_1)} \frac{f(W_2)}{\phi(W_2)}\right]  \notag\\
& \quad =\quad
\E \left[ 
	\left(W_1'W_2 - \|x\|^2 \right) 
	\frac{f(W_1)}{\phi(W_1)} \frac{f(W_2)}{\phi(W_2)} 
\right], \label{th1aa}
\end{align}
provided that the expected values 
in \eqref{th1bb} and \eqref{th1aa} are all finite.
[Indeed, finiteness of the expected values in \eqref{th1bb}
entails that $\int [ h(x|B) - 1]^2 d \nu_{d,p}(B)$ is finite, so that
$\nu_{d,p}\{ B : h(x|B) = \infty \} = 0$,
whence Proposition~\ref{p1} can be used to obtain the second equality
in the preceding display. The first and the third equality follow from
finiteness of the expected values
in \eqref{th1bb} and \eqref{th1aa}.]

To express the integral in \eqref{th2a} in a similar way, define 
$\Delta_{x|B}(z) : \R^p \times \mathcal V_{d,p} \times 
\R^d \to \R^{d\times d}$ by $\Delta_{x|B}(z)=zz'-(I_d+B(xx'-I_p)B')$,
and use Proposition~\ref{p1}  component-wise with 
$\Psi_{i,j}(Z) = \left[\Delta_{x|B}(Z)\right]_{i,j}$ for all $i,j=1,\dots,d$ 
to obtain
\begin{align}
&\int \trace{\Delta_{x|B}^k} h(x|B)^k ~ d\nu_{d,p}(B) \notag \\
& \quad =\quad
\int 
		\trace \left\{ 
			 \left( \E[\Delta_{x|B}(Z)\|B'Z=x] h(x|B) \right)^k
		\right\}
~ d\nu_{d,p}(B) \notag \\
& \quad =\quad
\int 
		\trace \left\{ 
			 \left(\E \left[ \Delta_{x|B}(W^{x|B}) 
			 \frac{f(W^{x|B})}{\phi(W^{x|B})} \right]\right)^k
		\right\}
~ d\nu_{d,p}(B) \notag\\
& \quad =\quad
\E \left[ 
	\trace \Delta_{x|\mathbf B} (W_1) \cdots \Delta_{x|\mathbf B} (W_k) 
	\frac{f(W_1)}{\phi(W_1)} \cdots \frac{f(W_k)}{\phi(W_k)}
 \right], \label{th2aa}
\end{align}
provided that the expected values in 
\eqref{th1bb} and \eqref{th2aa} are all finite.
Lemma~\ref{trace} describes how the expression in \eqref{th2aa} can be 
written as a weighted sum of expressions that, similarly to \eqref{th1aa}, 
involve only inner products of the $W_i$ and a product of density ratios. 
In particular, we find that \eqref{th2aa} can be written as a 
weighted sum of terms of the form
\begin{equation}\label{th2aaa}
\begin{split}
&\sum_{j=1}^k (-1)^{k-j}\binom{k}{j} \E \left[  
	\left(W_1'W_2\cdots W_jW_j' W_1 - d + p- \|x\|^{2j} 
	\right) 
	\frac{f(W_1)}{\phi(W_1)} \cdots \frac{f(W_k)}{\phi(W_k)}\right] ,
	\text{ and}
\\
& \E \left[  \left(\prod_{i=1}^m W_{j_{i-1}+1}'W_{j_{i-1}+2}W_{j_{i-1}+2}'
	\cdots W_{j_i-1}' W_{j_i} -\|x\|^{2(j_m-m)} \right) 
	\frac{f(W_1)}{\phi(W_1)} \cdots \frac{f(W_k)}{\phi(W_k)}\right] 
\end{split}
\end{equation}
for $m\ge 1$ and indices $j_0,\ldots,j_m$ satisfying $j_0=0$, $j_m<k$, and 
$j_{i-1}+1<j_i$ whenever $1\le i\le m$,
provided that the expected values in \eqref{th2aaa}
are all finite.
[Note that these requirements entail 
that $m \le k/2$, and that there are no more than ${k \choose m}$
choices for the indices $j_0,\dots, j_m$ 
in the second expected value in \eqref{th2aaa}.]
Lemma~\ref{trace} also shows that the weights in this expansion of \eqref{th2aa}
only depend on $k$
and on $x$, and are polynomials in $\|x\|^2$. 
Note that hence all the weights are bounded,
in absolute value and uniformly in $x\in \mathcal S_{M,p}$, 
by $e^{c(k) M^2}$ for some constant $c(k)$ that depends only on $k$.
In particular, if the expected values in \eqref{th2aaa} are all bounded
by \eqref{bound} in absolute value, then the same is true for \eqref{th2aa}
or, equivalently, \eqref{th2a} upon replacing the constants 
$g$ and $\kappa$
in \eqref{bound} 
by, say,  $g+c(k)$ and $(1+(k/2) {k \choose k/2})\kappa$, respectively.

In this subsection, we have seen how the integrals in \eqref{th1a}, 
\eqref{th2a} and \eqref{th1b} can be re-written 
as weighted sums of expected
values involving the rotational clones, provided these
expected values are all finite.
To prove Theorem~\ref{thm:main}\ref{theoremA}, it thus remains
to show, under the maintained assumptions,
that the expected values in \eqref{th1bb} and \eqref{th1aa}
are all bounded by \eqref{bound} in absolute value,
uniformly in $x\in \mathcal S_{M,p}$, and for constants as in
\eqref{th1a}. 
And Theorem~\ref{thm:main}\ref{theoremB} follows if, under
the assumptions of that theorem, the expected values
in \eqref{th1bb}  and 
\eqref{th1aa} as well as
all the expressions in 
\eqref{th2aaa} are bounded by \eqref{bound}
in absolute value, uniformly in $x\in \mathcal S_{M,p}$, and for
constants as in \eqref{th2a}.

\subsection{The joint density of the `rotational clones'} 
\label{s3.3}

\begin{proposition}\label{p2}
For integers $1 \le p < d$ and $1\le k \le d-p$, let $x\in\R^p$, and 
let $W_1,\dots,W_k$ be as in \eqref{Wj}. Then $W_1,\dots,W_k$ have a
joint density $\varphi_x(w_1,\dots,w_k)$ with respect to Lebesgue 
measure which satisfies
\begin{align*}
&\frac{\varphi_x(w_1,\dots,w_k)}{\phi(w_1)\cdots\phi(w_k)} \quad =\\
&
\left(\frac{d}{2}\right)^{-\frac{p k }{2}}
\prod_{i=1}^k 
\frac{ \Gamma \left( (d-i+1)/2\right)}
	{\Gamma \left((d-p-i+1)/2\right)}
\det(S_k)^{-\frac{p}{2}}
\left( 1-\frac{\|x\|^2}{d}\iota'S_k^{-1}\iota\right)^{\frac{d-p-k-1}{2}}
e^{\frac{k}2 \|x\|^2}
\end{align*}
if $S_k$ is invertible with $\|x\|^2\iota'S_k^{-1}\iota<d$, and 
$\varphi_x(w_1,\dots,w_k) = 0$ otherwise, where
$S_k=\left(w_i'w_j/d\right)_{i,j=1}^k$ denotes the $k\times k$ 
matrix of scaled inner products of the $w_i$, and $\iota=(1,\dots,1)'$ 
denotes an appropriate vector of ones.

If, in addition,  $k < d-p-1$, then the normalizing constant  in 
the preceding display,
i.e., the quantity 
$\eta(d,p,k) = (d/2)^{-kp/2}
\prod_{i=1}^k \frac{ \Gamma ( (d-i+1)/2)} {\Gamma ((d-p-i+1)/2)}$
satisfies
\begin{align*}
0\quad < \quad\eta(d,p,k)
\quad\le\quad
\exp\left[ \frac{p^2}{d} \left( 1-\frac{p+k-1}{d}\right)^{-1} 
\frac{k^2}{2}\right].
\end{align*}
\end{proposition}

Note that Proposition~\ref{p2} applies whenever $p$, $d$, and $k$
are as in \eqref{th1a} or \eqref{th2a}.
For $p$, $d$, and $k$ as in Proposition~\ref{p2}, 
we can re-express the expression in \eqref{th1bb} as
\begin{align}
&\E\left[ \frac{f(W_1)}{\phi(W_1)} \frac{f(W_2)}{\phi(W_2)} - 1 \right] 
	- 2 \E\left[ \frac{f(W_1)}{\phi(W_1)} - 1 \right] \notag\\
&\quad =\quad
\int_{\R^d\times \R^d} \left( \frac{f(w_1)}{\phi(w_1)} 
	\frac{f(w_2)}{\phi(w_2)} - 1 \right) 
	\varphi_x(w_1,w_2)\, dw_1 \, dw_2 \notag\\
& \quad \quad \quad - 2 \int_{\R^d} \left( \frac{f(w_1)}{\phi(w_1)} 
	- 1 \right) \varphi_x(w_1) \,dw_1  \notag\\
& \quad =\quad
\E\left[ \frac{\varphi_x(Z_1, Z_2)}{\phi(Z_1)\phi(Z_2)}- 1 \right] 
	- 2 \E\left[ \frac{\varphi_x(Z_1)}{\phi(Z_1)} -1 \right], \label{EZ1}
\end{align}
we can re-write \eqref{th1aa} as
\begin{align}
\E \left[ 
	\left(Z_1'Z_2 - \|x\|^2 \right) 
	\frac{\varphi_x(Z_1, Z_2)}{\phi(Z_1)\phi(Z_2)} 
\right], \label{EZ2}
\end{align}
and the expressions in \eqref{th2aaa} can be written as
\begin{equation}\label{EZ3}
\begin{split}
&\sum_{j=1}^k (-1)^{k-j}\binom{k}{j} 
	\E \left[  \left(Z_1'Z_2\cdots Z_jZ_j' Z_1 - d + p- \|x\|^{2j} \right) 
		\frac{\varphi_x(Z_1,\dots, Z_k)}{\phi(Z_1)\cdots\phi(Z_k)}
		\right],  \\
& 	\E \left[  \left(\prod_{i=1}^m Z_{j_{i-1}+1}'Z_{j_{i-1}+2}\cdots 
		Z_{j_i-1}' Z_{j_i} -\|x\|^{2(j_m-m)} \right) 
		\frac{\varphi_x(Z_1, \dots, Z_k)}{\phi(Z_1) \cdots \phi(Z_k)}
		\right], 
\end{split}
\end{equation}
for $m\ge 1$ and indices $j_0,\ldots,j_m$ satisfying $j_0=0$, $j_m<k$, 
and $j_{i-1}+1<j_i$ whenever $1\le i\le m$.

In this subsection,
we have seen how the integrals in \eqref{th1a}, 
\eqref{th2a} and \eqref{th1b} can be re-written 
as weighted sums of expected
values involving i.i.d. copies of $Z$ and the
density of the rotational clones,
provided these expected values are all finite.
Theorem~\ref{thm:main}\ref{theoremA}  now follows if we can show
that the expected values in \eqref{EZ1} and \eqref{EZ2}
are all bounded by \eqref{bound} in absolute value,
uniformly in $x\in \mathcal S_{M,p}$, and for constants as in
\eqref{th1a}, under the assumptions of that theorem. 
Similarly, 
Theorem~\ref{thm:main}\ref{theoremB}  follows, provided
the expected values in \eqref{EZ1} and \eqref{EZ2} as well as
all the expressions in \eqref{EZ3}
are bounded by \eqref{bound} in absolute value,
uniformly in $x\in \mathcal S_{M,p}$, and for constants as in
\eqref{th2a}, under the assumptions of that theorem.

\subsection{Two sufficient conditions}
\label{s3.4}

For an even integer $k$, consider the quantities
\begin{align}
& \E \left[  \left(\prod_{i=1}^m Z_{j_{i-1}+1}'Z_{j_{i-1}+2}\cdots 
	Z_{j_i-1}' Z_{j_i} \right) 
	\frac{\varphi_x(Z_1, \dots, Z_l)}{\phi(Z_1) \cdots \phi(Z_l)}\right]  
	- \|x\|^{2(j_m-m)} \label{condition2} 
\end{align}
for $l=1,\dots, k$, for each $m\ge 0$, and 
for each set of indices $j_0\dots, j_m$ that satisfies $j_0=0$, 
$j_m \le l$ and $j_{i-1} +1 < j_i$ whenever $1\le i \le m$.
And, again for even $k$, consider
\begin{align}
&\sum_{j=1}^k (-1)^{k-j}\binom{k}{j} \E \Bigg[  
\left(Z_1'Z_2\cdots Z_jZ_j' Z_1 -d +p- 1 \right) \notag\\
& \hspace{4cm} \times \frac{\varphi_x(Z_1,\dots Z_k)}{\phi(Z_1)
	\cdots\phi(Z_k)}\Bigg] - (1-\|x\|^2)^k. \label{condition1}
\end{align}

If the expressions of the form  \eqref{condition2} 
are all bounded by \eqref{bound}, in absolute value
and with constants as in \eqref{th1a},
then both the expected values in \eqref{EZ1} and \eqref{EZ2} 
are also bounded by \eqref{bound}, again in absolute value and
for constants as in \eqref{th1a}.
Indeed, the two expected values in \eqref{EZ1} are special cases
of  \eqref{condition2}, namely with $m=0$ and $l=1$, and
with $m=0$ and $l=2$, respectively.
Similarly, one sees that \eqref{EZ2} equals 
\begin{align*}
\E \left[ 
	\left(Z_1'Z_2 \right) 
	\frac{\varphi_x(Z_1, Z_2)}{\phi(Z_1)\phi(Z_2)} - \|x\|^2 
\right] 
- \|x\|^2\E\left[\frac{\varphi_x(Z_1, Z_2)}{\phi(Z_1)\phi(Z_2)}-1\right].
\end{align*}
Note that the two expected values in the preceding display
are special cases of \eqref{condition2}, namely with $m=1$, $l=2$ and with
$m=0$, $l=2$. If these special cases of \eqref{condition2}
are both bounded by \eqref{bound} in absolute value, uniformly in 
$x \in \mathcal S_{M,p}$, and for constants as in \eqref{th1a}, then
the expression in the preceding display
is similarly bounded by the product of \eqref{bound} and
$1+M^2$. It is now easy to see that the resulting upper bound
on the expression in the preceding display, and hence also on \eqref{EZ2},
is itself  upper bounded by an expression of the form \eqref{bound}
for constants as in \eqref{th1a}.

Similarly,
if the expressions of the form \eqref{condition2} and also 
\eqref{condition1} are bounded by \eqref{bound}, in absolute value
and with constants as in \eqref{th2a}, then the expected values in
\eqref{EZ1} and \eqref{EZ2} as well as all 
the expressions in \eqref{EZ3} are also bounded
by \eqref{bound}, again in absolute value and for constants as in 
\eqref{th2a}.
Indeed, \eqref{EZ1} can be bounded as claimed by arguing as in the
preceding paragraph.
For \eqref{EZ3}, we re-write
the first expression in that display as
\begin{align*}
&\sum_{j=1}^k (-1)^{k-j}\binom{k}{j} \E \left[  
\left(Z_1'Z_2\cdots Z_jZ_j' Z_1 - d + p- \|x\|^{2j} \right) 
\frac{\varphi_x(Z_1,\dots, Z_k)}{\phi(Z_1)\cdots\phi(Z_k)}\right]  \\
&\quad =\quad
\sum_{j=1}^k (-1)^{k-j}\binom{k}{j} \E \left[  
\left(Z_1'Z_2\cdots Z_jZ_j' Z_1 - d + p- 1 \right) 
\frac{\varphi_x(Z_1,\dots, Z_k)}{\phi(Z_1)\cdots\phi(Z_k)}\right] \\
& \quad \quad \quad + 
\E\left[\frac{\varphi_x(Z_1,\dots, Z_k)}{\phi(Z_1)\cdots\phi(Z_k)}\right]
\sum_{j=1}^k (-1)^{k-j}\binom{k}{j}  \left( 1-\|x\|^{2j}\right) \\
&\quad =\quad 
\eqref{condition1} + \left(1-\|x\|^2 \right)^k - 
\E\left[\frac{\varphi_x(Z_1,\dots, Z_k)}{\phi(Z_1)\cdots\phi(Z_k)}\right] 
\left( 1-\|x\|^2 \right)^k \\
&\quad =\quad 
\eqref{condition1} - 
\left(1-\|x\|^2 \right)^k  
\E\left[\frac{\varphi_x(Z_1,\dots, Z_k)}{\phi(Z_1)\cdots\phi(Z_k)}-1\right] ,
\end{align*}
where the second equality is obtained from the binomial formula
upon recalling that $k$ is even.
From this, the first expression in \eqref{EZ3} can be bounded by
an expression of the form \eqref{bound}, namely by first
bounding both \eqref{condition2} and
\eqref{condition1} with $m=0$ and $l=k$ by \eqref{bound},
by using the fact that $(1-\|x\|^2)^k \leq 2^k M^{2 k}$
for $x\in \mathcal S_{M,p}$, and by adjusting
the constants $\kappa$ and $g$ in \eqref{bound} accordingly.
The second expression in \eqref{EZ3} can be bounded in a similar fashion
upon using appropriate bounds on \eqref{condition2}.

In this subsection,
we have seen how bounds on the expressions of the form
\eqref{condition2} and on \eqref{condition1} can be used to
prove both parts of Theorem~\ref{thm:main}.
In particular, Theorem~\ref{thm:main}\ref{theoremA} follows if,
under the assumptions of that theorem,
the expressions of the form \eqref{condition2} 
are all bounded by \eqref{bound}, in absolute value,
uniformly in $x \in \mathcal S_{M,p}$, and for constants as
in \eqref{th1a}.
And Theorem~\ref{thm:main}\ref{theoremB} follows if
the expressions of the form \eqref{condition2} and also \eqref{condition1}
are all bounded by \eqref{bound}, in absolute value, uniformly in 
$x\in \mathcal S_{M,p}$, and for constants as in \eqref{th2a},
under the assumptions of that theorem.

\subsection{Approximating the density ratio}
\label{s3.5}

\begin{proposition} \label{p3} 
Fix $M>1$, positive integers $k$, $d$, and $p$,
such that $d>p^2$ and $d>4(k+p+1)M^4$ and let $x\in \mathcal S_{M,p}$.
For a collection of $d$-vectors $w_1,\dots, w_k$, define the $k\times
k$-matrix $S_k = (w_i' w_j/d)_{i,j=1}^k$.  Then the density ratio
$\frac{\varphi_x(w_1,\dots,w_k)}{\phi(w_1)\cdots\phi(w_k)}$ can be expanded as
\begin{align} 
	&\frac{\varphi_x(w_1,\dots,w_k)}{\phi(w_1)\cdots\phi(w_k)} 
	\quad=\quad
	\psi_x(S_k-I_k) \;+\; \Delta,
	\nonumber
\end{align} 
where the quantities on
the right hand side have the following properties:

$\psi_x$ is a polynomial of degree $k$ in the elements of $S_k-I_k$ whose
coefficients are bounded by $p^kM^{2(k+2)} C_\psi(k)$, 
where $C_\psi(k)$ depends only
on $k$. In particular, we may write 
\begin{align*} 
	\psi_x (S_k-I_k) \quad=\quad
	\sum_{H\in\mathcal M_k} \mathcal C(H) \; H(S_k-I_k), 
\end{align*} 
where
$\mathcal M_k$ is the set of all monomials in the entries of a symmetric
$k\times k$-matrix (i.e., in $k(k+1)/2$ variables) up to degree $k$ and
$\mathcal C(H) \in \R$ is the coefficient in $\psi_x$ corresponding to the
monomial $H$, which satisfies $|\mathcal C(H)| \le p^kM^{2(k+2)} C_\psi(k)$.  In
addition, the coefficients $\mathcal C(H)$ are invariant under permutations in
the following sense: Define the function $g$ by $g(w_1,\dots,w_k) = S_k - I_k$.
If $H,G \in \mathcal M_k$ are such that $H\circ g (w_1, \dots, w_k) = G \circ g
(w_{\pi(1)}, \dots, w_{\pi(k)})$, for some permutation $\pi$ of $k$ elements
and every choice of $w_1,\dots,w_k \in\R^d$, then $\mathcal C(H) = \mathcal
C(G)$.

Moreover, there exists a constant $\xi(k)>2k$ that depends only on the value of
$k$, such that whenever $\|S_k - I_k \| < 1/(p \xi(k))$, the remainder term
$\Delta$ satisfies $|\Delta| \le
p^{k+1}M^{2(k+2)}e^{\frac{k}{2}M^2}\|S_k-I_k\|^{k+1} C_\Delta(k)$, where
$C_\Delta(k)$ is a constant that depends only on $k$.  \end{proposition}


Note that Proposition~\ref{p3} applies whenever $M$, $k$, $d$,
and $p$ are either as in \eqref{th1a} or as in \eqref{th2a}.
The proposition suggests to replace the density ratio
$\frac{\varphi_x(w_1,\dots, w_k)}{\phi(w_1)\cdots\phi(w_k)}$ by the polynomial
$\psi_x(S_k-I_k)$.  For a fixed even integer $k$, we
therefore consider, as approximations to the expressions in 
\eqref{condition2}, the quantities 
\begin{align} 
&
\E \left[  \left(\prod_{i=1}^m Z_{j_{i-1}+1}'Z_{j_{i-1}+2}\cdots Z_{j_i-1}'
Z_{j_i} \right) \psi_x(S_l-I_l)\right]  - \|x\|^{2(j_m-m)} \label{approx2}
\end{align}
for $l=1,\dots, k$, for each $m\ge 0$, and 
for each set of indices $j_0\dots, j_m$ that satisfies $j_0=0$, 
$j_m \le l$ and $j_{i-1} +1 < j_i$ whenever $1\le i \le m$. 
And as approximation to \eqref{condition1}, we consider
\begin{align}
&\sum_{j=1}^k
(-1)^{k-j}\binom{k}{j} \E \Bigg[  \left(Z_1'Z_2\cdots Z_jZ_j' Z_1 -d +p- 1
\right) \psi_x(S_k-I_k)\Bigg] - (1-\|x\|^2)^k. \label{approx1} 
\end{align} 
In order for these approximations to be useful we
have to make sure that the difference between \eqref{condition1} and
\eqref{approx1} as well as the difference between \eqref{condition2} and
\eqref{approx2} can be controlled.
The following proposition provides us with the appropriate tool.

\begin{proposition}\label{p4} 
Fix positive integers $d$ and $k$. Moreover, let $M>1$ and $p\in \N$ 
be such that 
$d>\max\{4(k+p+1)M^4, 2k+p(2k+2)2^{k+3}, p^2\}$. 
Let $Z$ be a random $d$-vector such that $\E Z = 0$ and $\E Z Z' = I_d$,
and such that bounds \ref{c.Sk}.(\ref{c.sup}) 
and \ref{c.density} obtain with $k$ as chosen here. Write
$Z_1,\dots, Z_k$ for i.i.d. copies of $Z$.
Finally, fix $l\in\{1,\dots,k\}$, let $S_l = (Z_i'Z_j/d)_{i,j=1}^l$, and
let $H(S_l-I_l)$ be a (fixed) monomial in the elements of $S_l-I_l$ whose
degree, denoted by $\deg(H)$, satisfies $0\leq \deg(H)  \leq l$. Then 
$$
\sup_{x \in \mathcal S_{M,p}}
\E\left [ d^\frac{l+\deg(H)}{2} \;| H(S_l-I_l)|\;  \left|
\frac{\varphi_x(Z_1,\dots, Z_l)}{\phi(Z_1)\cdots\phi(Z_l)} - \psi_x(S_l -
I_l)\right|\right] 
$$ 
is bounded by 
\begin{align*} p^{2k+1+\eps} M^{2(k+2)}
e^{\frac{k}{2}M^2} \left(2D\sqrt{\pi e}\right)^{pk} \alpha\, d^{-\eps/2-1/4}
\end{align*} 
times a constant that depends only on $k$. Here $\eps\in[0,1/2]$
and $\alpha\ge 1$ are the constants from~\ref{c.Sk}.(\ref{c.sup}) and $D$ is
the constant from~\ref{c.density}.  
\end{proposition}

This proposition applies 
under the assumptions of
Theorem~\ref{thm:main}\ref{theoremA} and for constants as in 
\eqref{th1a}, and also 
under the assumptions of
Theorem~\ref{thm:main}\ref{theoremB} and for constants as in 
\eqref{th2a}.
Under the assumptions of Proposition~\ref{p4},
consider first the difference of \eqref{condition1}
and \eqref{approx1}. This difference is a sum of $k$ terms (with $k$ 
an even integer), where the modulus of the $j$-th term is bounded by
\begin{align*} 
& \binom{k}{j} \E \left[ \left|Z_1'Z_2\cdots Z_jZ_j'Z_1 - d+p-1
\right| \left| \frac{\varphi_x(Z_1,\dots, Z_k)}{\phi(Z_1)\cdots\phi(Z_k)} -
\psi_x(S_k - I_k) \right| \right] 
\\ 
&\quad \le \binom{k}{j} \E \left[
d^{\frac{k+j}{2}} \left| H_j(S_k-I_k) \right| \left| \frac{\varphi_x(Z_1,\dots,
Z_k)}{\phi(Z_1)\cdots\phi(Z_k)} - \psi_x(S_k - I_k) \right| \right] 
\\ 
& \quad\quad + \binom{k}{j} \E \left[ d^{\frac{k+0}{2}}  \left| H_0(S_k-I_k)
\right| \left| \frac{\varphi_x(Z_1,\dots, Z_k)}{\phi(Z_1)\cdots\phi(Z_k)} -
\psi_x(S_k - I_k) \right| \right], 
\end{align*} 
with $H_0(S_k-I_k) = 1$,
$H_1(S_k-I_k)=Z_1'Z_1/d - 1$ and, for $j\ge2$, $H_j(S_k-I_k) =
d^{-j}Z_1'Z_2\cdots Z_jZ_j'Z_1$, with $\deg(H_j)=j$ for $j=0,1,\dots k$, and
where we have used $|p-1|\le |-d+p-1|\le d^{k/2}$. Thus Proposition~\ref{p4}
yields an upper bound for the supremum over $\mathcal S_{M,p}$ of the absolute
value of this difference. 
Taken together, the difference of \eqref{condition1} and \eqref{approx1}
can be bounded, in absolute value and uniformly over 
$x \in\mathcal S_{M,p}$, by the upper bound from Proposition~\ref{p4}
multiplied by $2 \sum_{j=1}^k {k \choose j} = 2(2^k-1)$ (which is
a constant that depends only on $k$).
In the same way one sees that also the absolute difference
between \eqref{condition2} and \eqref{approx2} can be bounded, uniformly
over $x \in \mathcal S_{M,p}$, by the quantity
given by Proposition~\ref{p4}.
Note that the bound in Proposition~\ref{p4} is itself upper bounded by 
\eqref{bound}, e.g.,  by choosing $\kappa=\kappa(k,\alpha)$ 
and $g = k/2 + 2(k+2)$.

The arguments in this subsection entail that it suffices to bound
the approximating quantities \eqref{approx2} and \eqref{approx1}:
Theorem~\ref{thm:main}\ref{theoremA} follows if, under the assumptions
of that theorem, the expressions of the form \eqref{approx2}
are all bounded by \eqref{bound}, in absolute value, uniformly over
$x \in \mathcal S_{M,p}$, and for constants as in \eqref{th1a}.
Similarly, 
Theorem~\ref{thm:main}\ref{theoremB} follows if, under the assumptions
of that theorem, the expressions of the form \eqref{approx2}
and also \eqref{approx1}
are all bounded by \eqref{bound}, in absolute value, uniformly over
$x \in \mathcal S_{M,p}$, and for constants as in \eqref{th2a}.

\subsection{Comparing the approximating quantities with Gaussian expressions}
\label{s3.6}

We now compare \eqref{approx2} and \eqref{approx1} with the analogous
expressions where the $Z_1,\dots, Z_k$ are replaced by i.i.d. standard normal
$d$-vectors $V_1,\dots, V_k$ (and the Gram matrices 
$S_l = (Z_i'Z_j/d)_{i,j=1}^l$
for $l=1,\dots, k$ are replaced by the corresponding Gram matrices of the
$V_i$, i.e., by $S_l^\star = (V_i'V_j/d)_{i,j=1}^l$). 
In particular, we show that \eqref{approx2}, with the $Z_i$ replaced
by the $V_i$, can be controlled as desired, and
we then show that the difference of \eqref{approx2} and of
\eqref{approx2}, with the $Z_i$ replaced by the $V_i$,
can also be controlled. This will lead to the desired bound on
\eqref{approx2}. A similar strategy is also employed to bound
\eqref{approx1}.

The first step is to bound \eqref{approx2} and \eqref{approx1} 
with the $Z_i$ replaced by the $V_i$ in both displays.
More explicitly, for an even integer $k$, 
we want to bound the quantities
\begin{align}
& \E \left[  \left(\prod_{i=1}^m V_{j_{i-1}+1}'V_{j_{i-1}+2}\cdots 
	V_{j_i-1}' V_{j_i} \right) 
\psi_x(S_l^\star-I_l)\right]  - \|x\|^{2(j_m-m)}, \label{approx2gauss}
\end{align}
for $l=1,\dots, k$, for each $m\ge 0$, and 
for each set of indices $j_0\dots, j_m$ that satisfies $j_0=0$, 
$j_m \le l$ and $j_{i-1} +1 < j_i$ whenever $1\le i \le m$. 
And we also want to bound
\begin{align}
&\sum_{j=1}^k (-1)^{k-j}\binom{k}{j} \E \Bigg[  \left(V_1'V_2\cdots 
	V_jV_j' V_1 -d +p- 1 \right)
\psi_x(S_k^\star-I_k)\Bigg] - (1-\|x\|^2)^k. \label{approx1gauss}
\end{align}
Note that these expressions can be viewed as
approximations, in the sense of Proposition~\ref{p3}, 
to the expressions in the two preceding displays with 
$\psi_x(S_l^\star-I_l)$ 
and 
$\psi_x(S_k^\star-I_k)$ 
replaced by
$\varphi_x(V_1,\dots, V_l) / (\phi(V_1)\cdots \phi(V_l))$
and
$\varphi_x(V_1,\dots, V_k) / (\phi(V_1)\cdots \phi(V_k))$ ,
respectively.
But with that replacement, the resulting expressions are equal to
zero, as is easily verified (cf. Lemma~\ref{normalzero} for details). 
It is now elementary to verify, for a standard Gaussian $d$-vector $V$, 
that the bound~\ref{c.Sk} with $V$ replacing $Z$
is satisfied with $\eps=1/2$
(and therefore for all $\eps\in[0,1/2]$),
with $\xi=1/2$ (and thus for all $\xi \in (0,1/2]$), 
and with $\alpha$ and $\beta$ replaced by constants
$\alpha^\star$ and $\beta^\star$ that depend
only on $k$ (see also the proof of Example A.1(i) in \citep{Lee12b}).
Clearly also the bound~\ref{c.density} holds with $V$ replacing $Z$
(e.g. with $D=1$). 
Therefore, all the arguments so far concerning
the $d$-vectors $Z_1,\dots, Z_k$
also apply to the Gaussian $d$-vectors $V_1,\dots, V_k$.
In particular, we see that \eqref{approx2gauss} and \eqref{approx1gauss} are
bounded in absolute value by the quantity given in Proposition~\ref{p4},
and hence also
by \eqref{bound}, uniformly in $x \in \mathcal S_{M,p}$ and with constants
that depend only on $k$.

It remains to bound the difference of \eqref{approx2} and
\eqref{approx2gauss} as well as the difference of \eqref{approx1} and
\eqref{approx1gauss}. In particular, to complete the proof of 
Theorem~\ref{thm:main}\ref{theoremA}, we need to show, under the
assumptions of that theorem, that the absolute difference of
each expression of the form \eqref{approx2} and
the corresponding expression in \eqref{approx2gauss}
is bounded by \eqref{bound}, uniformly in $x\in \mathcal S_{M,p}$
and for constants as in \eqref{th1a}.
And to finish proving Theorem~\ref{thm:main}\ref{theoremB},
we have to show, under the assumptions of that theorem,
that the absolute difference of each expression of the form
\eqref{approx2} and the corresponding expression in \eqref{approx2gauss}
and also the absolute difference of \eqref{approx1} and \eqref{approx1gauss}
are all bounded by \eqref{bound},
uniformly in $x\in\mathcal S_{M,p}$ and for constants as
in \eqref{th2a}.

\begin{proposition}\label{p5} 
Suppose the random $d$-vector $Z$ satisfies the bounds \ref{c.Sk}.(\ref{c.sup})
and \ref{c.Sk}.(\ref{c.monom1}) for a fixed positive integer $k \le 4$. For
each $d\ge1$, let $Z_1,\dots, Z_k$ be i.i.d. copies of $Z$, set $S_k = (Z_i'
Z_j/d)_{i,j=1}^k$, and let $G$ and $H$ be two (fixed) monomials in the elements
of $S_k - I_k$ of degree $g$ and $h$, respectively, where $max\{g,h\}\le k$.  
Finally,
define $G^\star$ and $H^\star$ as $G$  and $H$, but with the $Z_1,\dots, Z_k$
replaced by i.i.d. standard Gaussian $d$-vectors, and consider
\begin{equation}\label{p5.1} 
\E\left[ d^g \Big( G - \E[G]\Big) H\right] \quad -
\quad \E\left[ d^g \Big( G^\star - \E[G^\star]\Big) H^\star\right].
\end{equation} 
\begin{list}{\thercnt}{ \usecounter{rcnt}
\setlength\itemindent{5pt} \setlength\leftmargin{0pt} \setlength\topsep{0pt}
\setlength\partopsep{0pt}
	} 
\item \label{p5.i} 
We then have $|\E[H] - \E[H^\star]| \le d^{-h/2} (\alpha
+ \alpha^\star)$.

\item \label{p5.ii} 
Assume that $G$ is given by the monomial
\begin{align}\label{openc} 
	\prod_{i=1}^m (S_k-I_k)_{j_{i-1}+1, j_{i-1}+2}
	\,(S_k-I_k)_{j_{i-1}+2, j_{i-1}+3} \cdots (S_k-I_k)_{j_i-1,j_i}, 
\end{align}
for some $m\ge 0$ and for indices $j_0\dots, j_m$ that satisfy
$j_0=0$, $j_m \le k$ and $j_{i-1} +1 < j_i$ whenever $1\le i \le m$.  Then
the expression in \eqref{p5.1} is bounded in absolute value by
$d^{-\min\{\xi,1/2\}} \max\{\alpha+\alpha^\star, \beta+\beta^\star\}$.

\item \label{p5.iii}
Assume that $G$ is given by the monomial 
\begin{align}\label{closedc}
(S_k-I_k)_{1,2}\,(S_k-I_k)_{2,3}\cdots(S_k-I_k)_{j-1,j}(S_k-I_k)_{j,1}
\end{align}
and note that the degree of $G$ is $j$ and satisfies
$1\leq j \leq k$.
Moreover, assume that also the bound \ref{c.Sk}.(\ref{c.monom2}) is 
satisfied with $k$ as 
chosen here. Then the expression in 
\eqref{p5.1} is bounded in absolute value by 
$d^{-\min\{\xi,1/2\}}(\beta(\alpha+1)+
\beta^\star(\alpha^\star+1)+\beta^2+{\beta^\star}^2)$, unless
either 
(a) $H=(S_k-I_k)_{a,a}$ for some $a$ satisfying $1\leq a \leq j$,
(b) $H=(S_k - I_k)_{a,b}$ with $1\leq a < b \leq j$,
or 
(c) $H=\left((S_k - I_k)_{a,b}\right)^2$ with $1\leq a < b \leq j$.  
In case (a), the expression in \eqref{p5.1} is
equal to $\Var[Z_1'Z_1]/d - 2$; 
in case (b), it is equal
to $\E[ (Z_1'Z_2)^3]/d$; and in case 
(c), it equals $\Var[(Z_1'Z_2)^2]/d^2 - 2(1+3/d)$.
\end{list}
Here, $\alpha$, $\beta$ and $\xi$ are the constants from~\ref{c.Sk}
and $\alpha^\star$ and $\beta^\star$ are the analogous quantities to $\alpha$
and $\beta$ if $Z$ is replaced by a standard Gaussian $d$-vector. Note
that
$\alpha^\star$ and $\beta^\star$ can be chosen such that they depend only on
the value of $k$.  
\end{proposition}

For later use, we note that Proposition~\ref{p5}\ref{p5.i}--\ref{p5.ii}
applies under the assumptions of Theorem~\ref{thm:main}\ref{theoremA}, 
and that
Proposition~\ref{p5}\ref{p5.i}--\ref{p5.iii} applies under the assumptions
of Theorem~\ref{thm:main}\ref{theoremB}.

Assume for the moment that Proposition~\ref{p3} and 
Proposition~\ref{p5}\ref{p5.i}--\ref{p5.ii} applies, 
and consider the difference between \eqref{approx2} and 
\eqref{approx2gauss}, for some integer $k\leq 4$,
and for indices $l$, $m$, $j_0, \dots, j_m$ such that 
$1\le l\le k$, $m\ge 0$, 
$j_0=0$, $j_m \le l$ and $j_{i-1} + 1 < j_i$ whenever $1\le i\le m$. 
If $m=0$, the difference of interest is simply
\begin{align} \label{simplediff}
&\E \left[ \psi_x(S_l-I_l)\right] 
 - \E \left[  \psi_x(S_l^\star-I_l)\right]. 
\end{align}
For $x\in \mathcal S_{M,p}$, recall that
$\psi_x(S_l-I_l)$ is a weighted sum of monomials in $S_k-I_k$ of degree up to
$k$, where the weights are all bounded in absolute value by $p^k M^{2(k+2)}
C_\psi(k)$, where $C_\psi(k)$ is a constant that depends only on $k$
(use Proposition~\ref{p3} with $l$ replacing $k$, and note that $l\leq k$).
In that
sum, the number of summands depends only on $k$. Therefore, \eqref{simplediff}
is a weighted sum of terms of the form $\E[H] - \E[H^\star]$ as in
Proposition~\ref{p5}\ref{p5.i} and thus we see that \eqref{simplediff} is
bounded by $p^k M^{2(k+2)}d^{-1/2} \kappa$ in absolute value and uniformly in 
$x\in \mathcal S_{M,p}$, where $\kappa$ depends only on 
$\alpha$ and $k$.
[Note that for $\deg(H) = 0$ the
difference in question is equal to zero.]
This, in turn, is obviously 
bounded by the quantity in
\eqref{bound} uniformly in $x\in \mathcal S_{M,p}$,
for an appropriate choice of $g$, e.g., $g=2(k+2)$, that depends only on $k$
and for $\kappa = \kappa(k,\alpha)$ depending only on $k$ and $\alpha$.
If $m>0$, the
difference of interest reads 
\begin{align*} 
&\E \left[  \left(\prod_{i=1}^m
Z_{j_{i-1}+1}'Z_{j_{i-1}+2}\cdots Z_{j_i-1}' Z_{j_i} \right)
\psi_x(S_l-I_l)\right] 
\\ 
&\quad - \E \left[  \left(\prod_{i=1}^m
V_{j_{i-1}+1}'V_{j_{i-1}+2}\cdots V_{j_i-1}' V_{j_i} \right)
\psi_x(S_l^\star-I_l)\right] \\ &= \E \left[ d^{\deg(G)}(G-\E [G])
\psi_x(S_l-I_l) \right]  - \E \left[ d^{\deg(G^\star)}(G^\star-\E [G^\star])
\psi_x(S_l^\star-I_l) \right], 
\end{align*} 
where $G$ and $G^\star$ are as in
Proposition~\ref{p5}\ref{p5.ii} and thus have mean zero. Again, if 
$x \in \mathcal S_{M,p}$, then
we are dealing with a weighted sum of expressions of the form \eqref{p5.1}
whose weights are all bounded in absolute value by $p^k M^{2(k+2)}C_\psi(k)$
(by Proposition~\ref{p3}).
Arguing as in the case where $m=0$ and now using 
Proposition~\ref{p5}\ref{p5.ii}, 
we see that the expression in the preceding display is bounded
by \eqref{bound} in absolute value and uniformly in $x\in \mathcal S_{M,p}$
for a constant $g = g(k)$ depending only on $k$ and for
$\kappa=\kappa(k,\alpha,\beta)$ depending only on $k$, $\alpha$, and $\beta$.

Now suppose that the assumptions of Theorem~\ref{thm:main}\ref{theoremA}
are satisfied. The argument in the preceding paragraph with $k=2$
shows that
the absolute difference between \eqref{approx2} and \eqref{approx2gauss} 
is bounded by \eqref{bound} uniformly in $x\in \mathcal S_{M,p}$ and
for constants as in \eqref{th1a}  (because $k$ is fixed
and equals $2$ here). This completes the proof of
Theorem~\ref{thm:main}\ref{theoremA}.

Finally, assume that the conditions of Theorem~\ref{thm:main}\ref{theoremB}
are met (and note that hence $k=4$).
Arguing as in the second-to-last paragraph, we see that
the absolute difference between \eqref{approx2} and \eqref{approx2gauss} 
is bounded by \eqref{bound} uniformly in $x\in \mathcal S_{M,p}$ and
for constants as in \eqref{th2a}  (again because $k$ is fixed here).
It remains to deal with the difference of \eqref{approx1} and
\eqref{approx1gauss}. 
It is not difficult to re-write this difference as
\begin{align}
 &\sum_{j=1}^k (-1)^{k-j}\binom{k}{j} 
	\Bigg\{ 
		\E \Big[  \left(Z_1'Z_2\cdots Z_jZ_j' Z_1 -d +p- 1 \right)
		\psi_x(S_k-I_k)\Big] 			\notag\\
 &\quad\quad	- \E \Big[  \left(V_1'V_2\cdots V_jV_j' V_1 -d +p- 1 \right)
		\psi_x(S_k^\star-I_k)\Big]
	\Bigg\} 					\notag\\
   \begin{split}
 &\quad =
		\sum_{j=1}^k (-1)^{k-j}\binom{k}{j} 
	\Bigg\{ 
		\E \Big[  d^{\deg(G_j)} \left(G_j - \E [G_j]  \right)
		\psi_x(S_k-I_k)\Big] 			\label{lastsum}\\
 &\quad\quad\quad	- \E \Big[  d^{\deg(G_j^\star)}\left(G_j^\star 
 	- \E[G_j^\star]  \right)
		\psi_x(S_k^\star-I_k)\Big]
	\Bigg\} \\				
   \end{split}\\ \nonumber
  &\quad\quad + 
\sum_{j=1}^k (-1)^{k-j}\binom{k}{j} (p-1)
	\Bigg\{ 
		\E \Big[ \psi_x(S_k-I_k) \Big] 
		-  \E \Big[ \psi_x(S_k^\star-I_k) \Big] 
	\Bigg\},
\end{align}
where $G_j$ is shorthand for 
$Z_1'Z_1/d - 1$ if $j=1$ and for $Z_1'Z_2\cdots Z_jZ_j' Z_1/d^j$ if $j>1$, and
where 
$G_j^\star$ denotes the corresponding quantity computed from the $V_i$. 
Clearly, $\E[G_1] =
0$, while, for $j>1$, $\E[G_j] = d^{-j+1}$, which verifies the equality in
the previous display.  Now the expected value in the last line of the preceding
display can be treated similarly to \eqref{simplediff} with $l=k$. Noting that
$(p-1) \sum_{j=1}^k (-1)^{k-j} {k \choose j} = (p-1)$, we see that
the expression in the last line of the preceding display
is bounded, in absolute value and uniformly in $x \in \mathcal S_{M,p}$,
by \eqref{bound} for constants as in \eqref{th2a}.
We are left with the sum in \eqref{lastsum}. 
Recalling that
$\psi_x(S_k-I_k)$ and $\psi_x(S^\star_k - I_k)$ are both weighted sums
of monomials,  we can re-write \eqref{lastsum} as 
\begin{align}
\begin{split}
\sum_{j=1}^k 
(-1)^{k-j}\binom{k}{j} 
\sum_{H\in \mathcal M_k} 
\mathcal C(H)
	\Bigg\{ 
		&\E \Big[  d^{\deg(G_j)} \left(G_j - \E [G_j]  \right)
		H)\Big] 			\label{doublesum}\\
 &\quad\quad	- \E \Big[  d^{\deg(G_j^\star)}\left(G_j^\star - 
 	\E[G_j^\star]  \right)
		H^\star)\Big]
	\Bigg\}, 
\end{split}
\end{align}
where $\mathcal M_k$ and $\mathcal C(H)$, $H \in \mathcal M_k$, are
as in Proposition~\ref{p3}, and where $H^\star$ is defined as $H$
but with the $Z_i$ replaced by the $V_i$.
Note that \eqref{doublesum} is a weighted sum of expressions of the
form \eqref{p5.1}, where $G_j$ and $G_j^\star$ are given by \eqref{closedc}
defined with the $Z_i$ and the $V_i$ respectively. 
The coefficients $\mathcal C(H)$ of the terms in the polynomial 
$\varphi_x(S_k-I_k)$, as defined in Proposition~\ref{p3},
can be bounded, in absolute value
and uniformly in $x \in \mathcal S_{M,p}$, by $p^k M^{2(k+2)} C_\psi(k)$.
And using Proposition~\ref{p5}\ref{p5.iii}, 
the difference of expectations in \eqref{doublesum}
can be bounded by $d^{-\min\{\xi, 1/2\}} 
(\beta(\alpha+1) + \beta^\star(\alpha^\star+1) + \beta^2 + \beta^{\star^2})$,
unless the case (a), (b), or (c) in Proposition~\ref{p5}\ref{p5.iii} occurs.
Taken together, it is now elementary to verify that those
terms in \eqref{doublesum} that do not correspond to 
the cases (a), (b), or (c) in the proposition are bounded,
in absolute value and uniformly in $x\in \mathcal S_{M,p}$,
by an expression of the form \eqref{bound} with constants as in 
\eqref{th2a}.

To deal with the remaining terms in \eqref{doublesum}, consider first those 
terms where $G_j$ and $H$ are as in
case (a) of Proposition~\ref{p5}\ref{p5.iii}. 
For each of these terms, the monomial $H$ is of the form 
$(S_k - I_k)_{a,a}$ for some $a$ with $1\leq a \leq j$.
Because the coefficients $\mathcal C(H)$ are invariant under
permutations in the sense of Proposition~\ref{p3}, we see that
the coefficient $\mathcal C(H)$ is the same number whenever case (a)
occurs, and we denote this number by $\mathcal C_{(a)}$ in the following.
Moreover, whenever $G_j$ and $H$ are such that the case (a) occurs,
then the difference in expectations in \eqref{doublesum}
is equal to the same number,   namely  $\Var[Z_1'Z_1]/d - 2$,
by Proposition~\ref{p5}\ref{p5.iii}.
Finally, for fixed $j$ (and hence for fixed $G_j$), 
we note that the number of monomials $H$,
where $G_j$ and $H$ are as in case (a), is equal to $j$.
Putting the pieces together, we see that the combined contribution
of those terms in \eqref{doublesum}, where $G_j$ and $H$ are as in
case (a) of Proposition~\ref{p5}\ref{p5.iii}, is equal to
$$
\mathcal C_{(a)} \left(\Var[Z_1' Z_1]/d - 2\right) \sum_{j=1}^k
{k \choose j}   (-1)^{k-j} j.
$$
But the sum in the preceding display can be written as 
$k \sum_{j=0}^{k-1} {k-1 \choose j} (-1)^{k-1-j} =  k (1-1)^{k-1} = 0$
(recall that we have chosen $k=4$ here).
Hence, the combined contribution of these terms is zero.

Consider now those 
terms in \eqref{doublesum} where $G_j$ and $H$ are as in
case (b) of Proposition~\ref{p5}\ref{p5.iii}. 
For each of these terms, the monomial $H$ is of the form 
$(S_k - I_k)_{a,b}$ for some $a$ and $b$ with $1\leq a < b \leq j$.
As in the preceding paragraph, we see that
the coefficient $\mathcal C(H)$ is the same number whenever case (b)
occurs, and we denote this number by $\mathcal C_{(b)}$;
cf. Proposition~\ref{p3}.
And whenever $G_j$ and $H$ are such that the case (b) occurs,
then the difference in expectations in \eqref{doublesum}
is equal to the same number,   namely  $\E[(Z_1'Z_1)^3]/d$;
cf. Proposition~\ref{p5}\ref{p5.iii}.
Finally, for fixed $j$ (and hence for fixed $G_j$), 
there are ${j \choose 2}$ monomials $H$
so that $G_j$ and $H$ are as in case (b).
The contribution of those terms in \eqref{doublesum},
where $G_j$ and $H$ are as in case (b) of Proposition~\ref{p5}\ref{p5.iii}, 
is therefore equal to
$$
\mathcal C_{(b)} \E[(Z_1' Z_1)^3/d] \sum_{j=1}^k {k \choose j} (-1)^{k-j}
{ j \choose 2}.
$$
The sum in the preceding display can be written as
${k \choose 2} \sum_{j=0}^{k-2} {k-2 \choose j} (-1)^{k-2-j} = 0$.
A similar argument shows that the combined contribution of
the terms where case (c) occurs is also zero.

In summary, we see, under the assumptions of 
Theorem~\ref{thm:main}\ref{theoremB}, that the difference of
\eqref{approx2} and \eqref{approx2gauss}
as well as the difference of \eqref{approx1} and \eqref{approx1gauss}
can be bounded by \eqref{bound} in absolute value, uniformly
in $x\in \mathcal S_{M,p}$, and for constants as in \eqref{th2a}.
With this, also the proof of Theorem~\ref{thm:main}\ref{theoremB}
is complete.

\section*{Acknowledgements}

The Austrian Science Fund (FWF) supports
first author with project P 28233-N32 and
the second author with projects P 28233-N32 and P 26354-N26.

\begin{supplement}[id=supp]
\stitle{Proofs for Sections~\ref{sec:ExEx} and \ref{proof:main}}
\slink[url]{appendix.pdf}
\sdescription{The supplement contains the proofs of Examples~\ref{Ex:2} and \ref{Ex:3}
	as well as several more technical arguments
	that are used in Section~\ref{proof:main} including, in particular,
	the proofs of Propositions~\ref{p1} through~\ref{p5}.
	}
\end{supplement}


\rem{
\section*{Notes on the proofs}

\begin{itemize}
\item 
	Lemma~\ref{measurability} is a simple extension of
	Corollary B.2 of \citep{Lee12b} with little change to the proof.
	[$0.5$ page.]
\item 
	Proposition~\ref{p1} is a simple extension of Proposition~2.2 of
	\citep{Lee12a}, and the proof is completely analogous.
	[$1$ pages.]
\item
	The second part of  the proof of Lemma~\ref{trace} 
	is completely analogous to the proof of Lemma~B.3
	in \citep{Lee12b}.
	[$1.5$ pages.]

\item
	The proof of Lemma~\ref{Wj1} is a simple adaptation of
	the argument surrounding relation (C1) in \citep{Lee12b}.
	[$0.5$ pages.]

\item
	Lemma~\ref{Wj2} is an extension of Lemma~C.1 and Lemma~C.2
	of \citep{Lee12b}, and can be proved by arguing along similar
	lines. [$3$ pages.]

\item	
	Lemma~\ref{Wj3} is an extension of part (i) and part (iii) of
	Lemma~C.3 in \citep{Lee12b}, and the proof is very very similar.
	[$1$ pages.]

\item
	Corollary~\ref{Wj4} is an extension of Lemma~C.4 in \citep{Lee12b}, 
	and the proof is again very very similar.
	[$1$ pages.]

\item
	Lemma~\ref{Uk} is kind of trivial and is already implicitly
	used in \citep{Lee12b}.
	[$0.5$ pages]

\item
	Also, the proof of Proposition~\ref{p2} is rather similar
	to the proof of Proposition~2.3 in \citep{Lee12b}.
	[$4$ pages.]

\end{itemize}

}


{\small

}

\clearpage
\newpage

\begin{appendix}


\section{Proofs of Section~\ref{sec:ExEx}}
\label{sec:AppExEx}

\begin{proof}[\bf Proof of Example~\ref{Ex:2}]
Part~\ref{Ex:2A} is immediate, since $S_k = S_k(Z_1,\dots, Z_k)$ is invariant under rotation of $Z_1,\dots, Z_k$ and Condition~\ref{c.density} is also rotation-invariant. 

For part~\ref{Ex:2B}, let $r_1,\dots, r_k$ and $Z_1,\dots, Z_k$ be i.i.d. copies of $r$ and $Z$, respectively, and let $Z_i^* = r_i Z_i$. Write $S_k = (Z_i'Z_j/d)_{i,j=1}^k$, $S_k^* = ({Z_i^*}'Z_j^*/d)_{i,j=1}^k = (r_ir_j Z_i'Z_j/d)_{i,j=1}^k$ and $D = \diag(r_1,\dots, r_k)$, and note that $S_k^* = D S_k D$ and $D^2=I_k$. Therefore, $\|\sqrt{d}(S_k^*-I_k)\|^{2k+1+\eps} = \|D\sqrt{d}(S_k-I_k)D\|^{2k+1+\eps} = \|\sqrt{d}(S_k-I_k)\|^{2k+1+\eps}$, and we see that \ref{c.Sk}.(\ref{c.sup}) also holds for the $Z_i^*$. To see that also the bounds in \ref{c.Sk}.(\ref{c.monom1},\ref{c.monom2}) still hold if the $Z_i$ are replaced by the $Z_i^*$, we use the notation $G^* = G(S_k^*-I_k)$ to denote an arbitrary monomial $G$ of degree $g$ evaluated at $S_k^*-I_k$. Such a monomial can be represented as an undirected graph on $k$ vertices with $g$ edges, where we denote the edges by $e_1,\dots,e_g \in \{1,\dots, k\}^2$. Alternatively, the monomial $G^*$ can be represented algebraically as 
\begin{align*}
G^* &= d^{-g}\left(\prod_{\substack{i=1\\e_i(1)\ne e_i(2)}}^g r_{e_i(1)} r_{e_i(2)} Z_{e_i(1)}'Z_{e_i(2)}\right)
\left(
\prod_{\substack{i=1\\e_i(1)= e_i(2)}}^g (r_{e_i(1)}^2 Z_{e_i(1)}'Z_{e_i(1)} - d)
\right)\\
&=
\left(\prod_{\substack{i=1\\e_i(1)\ne e_i(2)}}^g r_{e_i(1)} r_{e_i(2)} \right)
d^{-g}\left(\prod_{\substack{i=1\\e_i(1)\ne e_i(2)}}^g Z_{e_i(1)}'Z_{e_i(2)}\right)
\left(
\prod_{\substack{i=1\\e_i(1)= e_i(2)}}^g ( Z_{e_i(1)}'Z_{e_i(1)} - d)
\right)\\
&=
\left(\prod_{\substack{i=1\\e_i(1)\ne e_i(2)}}^g r_{e_i(1)} r_{e_i(2)}\right) G = \rho \;G. 
\end{align*}
By assumption, the product $\rho$ of binary random variables on the last line of the previous display is stochastically independent of $G = G(S_k-I_k)$, and its expectation is bounded in absolute value by one. Moreover, if $G^*$ consists only of quadratic factors above the diagonal, then $\rho = 1$, with probability one. Therefore, the bounds in \ref{c.Sk}.(\ref{c.monom1},\ref{c.monom2}) hold also for $Z^*$. 
Finally, for a fixed orthogonal $d\times d$ matrix $R$, the density of $RZ^*$ is a convex combination of the densities of $RZ$ and $-RZ$, and the same holds true for arbitrary marginal densities thereof. Hence, $Z^*$ satisfies \ref{c.density}.
\end{proof}

\begin{proof}[\bf Proof of Example~\ref{Ex:3}]
The first claim is easy to verify upon using the well known fact that the distribution of a spherically symmetric random $d$-vector $Z$ can be represented as $\|Z\| b$, where $b$ is uniformly distributed on the unit sphere in $\R^d$ and is independent of $\|Z\|$. Now fix $k\in\{2,4\}$. For Condition~\ref{c.Sk}, our strategy is to compare the quantities depending on i.i.d. copies of $Z$ through $S_k= (Z_i'Z_j/d)_{i,j=1}^k$, to analogous quantities calculated from i.i.d. copies of a standard Gaussian vector $V$. From Example~\ref{Ex:1} we know that \ref{c.Sk} is satisfied for $V$, with $\eps=\xi=1/2$ and certain constants $\alpha^*$ and $\beta^*$, depending only on $k$. Define $S_k^* = (V_i'V_j/d)_{i,j=1}^k$, and let $G = G(S_k-I_k)$ and $G^*=G(S_k^*-I_k)$ denote the same monomial of degree $g$ but evaluated at the entries of $S_k-I_k$ and $S_k^*-I_k$, respectively. As in the proof of the previous example, we denote the edges in the graph corresponding to $G$ (or $G^*$) by $e_1,\dots, e_g$. For each vertex in that graph, we write $p_j$ for the number of edges incident to vertex $j$ that do not constitute a loop of length one (i.e., such an edge satisfies $e_i(1)\ne e_i(2)$) and we write $q_j$ for the number of loops of length one incident to vertex $j$. Hence, we have $g=\sum_{j=1}^k (q_j + p_j/2)$, where $\sum_{j=1}^k p_j$ is an even number. Now introduce i.i.d. random vectors $b_j$, $j=1,\dots, k$, that are uniformly distributed on the unit sphere in $\R^d$, to obtain the representation $Z_j \thicksim \|Z_j\| b_j$, and decompose $\E G$ as
\begin{align*}
\E G &=
\E \left(\prod_{\substack{i=1\\e_i(1)\ne e_i(2)}}^g Z_{e_i(1)}'Z_{e_i(2)}/d\right)
\left(
\prod_{\substack{i=1\\e_i(1)= e_i(2)}}^g (Z_{e_i(1)}'Z_{e_i(1)}/d - 1)
\right)\\
&=
\E \left(\prod_{\substack{i=1\\e_i(1)\ne e_i(2)}}^g b_{e_i(1)}'b_{e_i(2)}\right)
\prod_{j=1}^k \E (\|Z_j\|/\sqrt{d})^{p_j} \left( \|Z_j\|^2/d - 1 \right)^{q_j}.
\end{align*}
Clearly, $\E G^*$ can be decomposed in exactly the same way upon replacing $Z_j$ by $V_j$. The moments of $\|Z\|$ are easily calculated (from the distribution function of the radial part of the uniform distribution $\P(\|Z\|\le x) = (x/\sqrt{d+2})^d$, for $x\in[0,\sqrt{d+2}]$) to be $\E[\|Z\|^m] = (d+2)^{m/2} d /(d+m)$. The moments of the $\chi_d$-distribution are given by $\E[\|V\|^m] = 2^{m/2}\Gamma(d/2 + m/2)/\Gamma(d/2)$, which reduces to $\prod_{i=0}^{m/2-1}(d+2i)$ if $m$ is even. If $G$ consists only of quadratic factors above the diagonal, then all $q_j$ are equal to zero, all $p_j$ are even and the ratio $\E G / \E G^*$ is given by
$$
\frac{\E G}{\E G^*} = \prod_{j=1}^k \frac{\E \|Z_j\|^{p_j}}{\E \|V_j\|^{p_j}} = \prod_{j=1}^k\frac{d (d+2)^{p_j/2}}{(d+p_j)\prod_{i=0}^{p_j/2-1}(d+2i)},
$$
where an empty product is defined to be equal to one. Since the factors $d/(d+p_j)$ and $(d+2)/(d+2i)$ are all of order $1 + O(1/d)$, as $d\to\infty$, so is the whole product. Therefore, we have $|d^{g/2}\E G - 1| \le |d^{g/2}\E G^* | | \E G/\E G^* - 1| + |d^{g/2}\E G^* - 1| = O(d^{-1/2})$, provided that we can show that $|d^{g/2}\E G^*| = O(1)$. But this last requirement follows from Lemma~A.3 in \cite{Lee12b}. This establishes the first part of \ref{c.Sk}.(\ref{c.monom1}).

Now, in full generality, we obtain for $q, p \in \N_0$
\begin{align}
\E  (\|Z\|/\sqrt{d})^{p} &\left( \|Z\|^2/d - 1 \right)^{q}
=
\sum_{\ell=0}^{q} \binom{q}{\ell} (-1)^{q-\ell} \E (\|Z\|/\sqrt{d})^{p+2\ell}\notag\\
&=
\sum_{\ell=0}^{q} \binom{q}{\ell} (-1)^{q-\ell} \left(\frac{d+2}{d}\right)^{(p+2\ell)/2} \frac{d}{d+p+2\ell}\notag\\
&=
\left(\frac{d+2}{d}\right)^{p/2} \sum_{\ell=0}^{q} \binom{q}{\ell} (-1)^{q-\ell} \left(1+\frac{2}{d}\right)^{\ell} \frac{d}{d+p+2\ell}.\label{eq:Ex1Unif}
\end{align}
Similarly, in the Gaussian case, we find
\begin{align}
\E  (\|V\|/\sqrt{d})^{p} &\left( \|V\|^2/d - 1 \right)^{q}
=
\sum_{\ell=0}^{q} \binom{q}{\ell} (-1)^{q-\ell} 2^{p/2+\ell} \frac{\Gamma(d/2+p/2+\ell)}{d^{p/2+\ell}\Gamma(d/2)}\notag\\
&=
\frac{\Gamma(d/2+p/2)}{(d/2)^{p/2}\Gamma(d/2)} \sum_{\ell=0}^{q} \binom{q}{\ell} (-1)^{q-\ell} \prod_{i=0}^{\ell-1} 
\left(1 + \frac{p+2i}{d}\right),\label{eq:Ex1Gauss}
\end{align}
where an empty product has to be interpreted as equal to one.
Both factors $\left(\frac{d+2}{d}\right)^{p/2}$ and $\frac{\Gamma(d/2+p/2)}{(d/2)^{p/2}\Gamma(d/2)}$ converge to one as $d\to \infty$ (use Sterling's approximation to the Gamma function). We abbreviate the alternating binomial sums in \eqref{eq:Ex1Unif} and \eqref{eq:Ex1Gauss} by $S(q,p,d)$ and $T(q,p,d)$, respectively. The rate of $S$ and $T$ as $d\to\infty$ is hard to work out in general. However, by direct calculation we get $S(0,p,d)=d/(d+p)$, $T(0,p,d) = 1$, $S(1,p,d) = 2p/((d+p)(d+p+2))$, $T(1,p,d) = p/d$, and automated symbolic calculations with Wolfram Mathematica $8$ show that $d^{q} S(q,p,d) \to s(q,p)$ and $d^{\lceil q/2\rceil}T(q,p,d) \to t(q,p)$, for $(q,p)\in\{2,\dots, 10\}\times \N_0$, as $d\to\infty$, and where the limiting expressions $s(q,p)$ and $t(q,p)$ are finite and non-zero for $(q,p)\in \{2,\dots, 10\}\times \N_0$. 
Therefore, we see that the expectation of any monomial $G$ of degree $g\le 10$ in $S_k-I_k$ is of the same or smaller asymptotic order in $d$ than $\E G^*$, which is calculated from the i.i.d. Gaussian vectors $V_1,\dots, V_k$. In other words, there exists a constant $c_1>0$, such that $|\E G| \le c_1 |\E G^*|$, at least for all sufficiently large $d$. This establishes the second part of \ref{c.Sk}.(\ref{c.monom1}), as well as Condition~\ref{c.Sk}.(\ref{c.monom2}).

For Condition~\ref{c.Sk}.(\ref{c.sup}), we bound the spectral norm by the Frobenius norm to obtain
\begin{align*}
\E &\left( \sqrt{d} \|S_k-I_k\|\right)^{2k+2}  \le
\E \left( d \sum_{i,j=1}^k (S_k-I_k)_{i,j}^2\right)^{k+1}\\
&= \sum_{i_1,j_1=1}^k \dots \sum_{i_{k+1},j_{k+1}=1}^k d^{k+1} \E \prod_{l=1}^{k+1}(S_k-I_k)_{i_l,j_l}^2\\
&\le
c_1\sum_{i_1,j_1=1}^k \dots \sum_{i_{k+1},j_{k+1}=1}^k d^{k+1} \E \prod_{l=1}^{k+1}(S_k^*-I_k)_{i_l,j_l}^2,
\end{align*}
at least for sufficiently large $d$, where we have used the previous considerations together with the fact that the monomial $\prod_{l=1}^{k+1}(S_k-I_k)_{i_l,j_l}^2$ has degree $2(k+1)\le 10$. The quantity on the last line of the previous display is bounded by a constant that depends only on $k$, again, in view of Lemma~A.3 in \cite{Lee12b}. 

Finally, for \ref{c.density}, simply recall that the volume of the $d$-ball with radius $r>0$ is given by $\pi^{d/2}r^d/\Gamma(d/2+1)$, i.e., a Lebesgue density $f_Z$ of $Z$ on $\R^d$ is given by
$$
f_Z(z) = \frac{\Gamma(d/2+1)}{\pi^{d/2}(d+2)^{d/2}} \mathbf 1_{\{\|z\|\le \sqrt{d+2}\}}.
$$
Thus, a corresponding $(d-n)$-dimensional marginal density $f_n$ on $\R^{d-n}$ is given by
\begin{align*}
f_n(z_{n+1},\dots, z_d) &= \int_{\R^n} f_Z(z_1,\dots, z_d) dz_1\dots dz_n \\
&=
\frac{\Gamma(d/2+1)}{\pi^{d/2}(d+2)^{d/2}} \int\limits_{\sum_{i=1}^n z_i^2\le d+2- \sum_{i=n+1}^d z_i^2} dz_1\dots dz_n\\
&=
\frac{\Gamma(d/2+1)}{\pi^{d/2}(d+2)^{d/2}} \frac{\pi^{n/2}(d+2-\sum_{i=n+1}^d z_i^2)^{n/2}}{\Gamma(n/2 +1)} \mathbf 1_{\{\sum_{i=n+1}^d z_i^2 \le d+2\}}.
\end{align*}
This marginal density attains its maximum at the origin, so that
$$
\|f_n\|_\infty = \frac{\Gamma(d/2+1)}{\pi^{d/2}(d+2)^{d/2}} \frac{\pi^{n/2}(d+2)^{n/2}}{\Gamma(n/2 +1)}.
$$
If $d$ is even, then $\Gamma(d/2+1)/(d+2)^{d/2}= (d/2)!/(d+2)^{d/2} \le 1$, and if $d$ is odd $\Gamma(d/2+1)/(d+2)^{d/2} \le \lceil d/2 \rceil !\sqrt{\pi}/(d+2)^{d/2} \le \sqrt{\pi/(d+2)}$. For $k\in\{2,4\}$, we need to consider $n\in\{1,3\}$, so
$$
\|f_1\|_\infty \le \sqrt{\frac{\pi}{\pi^d(d+2)}} \frac{\sqrt{\pi(d+2)}}{\sqrt{\pi}/2} \le 2 \pi^{(1-d)/2} \le  \binom{d}{1}^{1/2},
$$
for $d\ge 2$, and
$$
\|f_3\|_\infty \le \sqrt{\frac{\pi}{\pi^d(d+2)}} \frac{\sqrt{\pi^3(d+2)^3}}{\sqrt{\pi}3/4} = \frac{4}{3} \pi^{(3-d)/2} (d+2)\le \binom{d}{3}^{1/2},
$$
for $d\ge 5$. So \ref{c.density} holds with $D=1$, at least if $d\ge 5$.
\end{proof}

\begin{proof}[\bf Proof of Proposition~\ref{prop:Sigma}]
For $\mathbb G\subseteq \mathcal V_{d,p}$ as in Theorem~\ref{thm:main} and for a $d\times d$ positive definite diagonal matrix $\Lambda$, define the set $\mathbb U(\Lambda) = \mathbb U(\mathbb G,\Lambda)$ by 
$$
\mathbb U(\mathbb G,\Lambda) = \{ U\in\mathcal O_d : \nu_{d,p}(\mathbb J^c(U\Lambda U')) \le \nu_{d,p}(\mathbb G^c)^{1/2}\}.
$$
Thus, the claimed bound on $\sup_{\Sigma\in\mathbb S} \nu_{d,p}(\mathbb J^c(\Sigma))$ holds by definition.
To establish the same bound on $\sup_{\Lambda:\Lambda=\diag(\lambda_i)>0}\nu_{d,d}(\mathbb U^c(\Lambda))$, note that by Theorem~\ref{thm:main}\ref{theoremC} we have $\mathbb G^c \cdot V = \mathbb G^c$, for every $V\in\mathcal O_p$. By rotation invariance of the uniform distribution $\nu_{d,p}$, we get for every $d\times d$ positive definite diagonal matrix $\Lambda$, every $U\in\mathcal O_d$ and every $V\in\mathcal O_p$,
\begin{align*}
\nu_{d,p}(\mathbb J^c(U\Lambda U'))
&=
\nu_{d,p}(\{A\in\mathcal V_{d,p}: U\Lambda^{1/2} U'A (A'U\Lambda U'A)^{-1/2} \in \mathbb G^c\})\\
&=
\nu_{d,p}(\{C\in\mathcal V_{d,p}: U\Lambda^{1/2} A (A'\Lambda A)^{-1/2} V \in \mathbb G^c\})\\
&=
\int_{\mathcal V_{d,p}} \mathbf{1}_{\mathbb G^c}(U\Lambda^{1/2} A (A'\Lambda A)^{-1/2} V)\;\nu_{d,p}(dA).
\end{align*}
Integrating this expression, as a function in $(U,V) \in \mathcal O_d\times\mathcal O_p$, with respect to the product measure $\nu_{d,d}\otimes \nu_{p,p}$ and using the fact that $UMV \thicksim \nu_{d,p}$ under $\nu_{d,d}\otimes \nu_{p,p}$, for any fixed $M\in \mathcal V_{d,p}$, we obtain
\begin{align*}
&\int_{\mathcal O_d\times \mathcal O_p}\nu_{d,p}(\mathbb J^c(U\Lambda U'))\;(\nu_{d,d}\otimes \nu_{p,p}) (dU,dV)\\
&=
\int_{\mathcal V_{d,p}} \int_{\mathcal O_d\times \mathcal O_p}
\mathbf{1}_{\mathbb G^c}(U\Lambda^{1/2} A (A'\Lambda A)^{-1/2} V)\; (\nu_{d,d}\otimes \nu_{p,p}) (dU,dV)\;\nu_{d,p}(dA)\\
&=
\int_{\mathcal V_{d,p}} \nu_{d,p}(\mathbb G^c)\;\nu_{d,p}(dA) = \nu_{d,p}(\mathbb G^c),
\end{align*}
where we have used Tonelli's theorem. This identity together with Markov's inequality and another application of Tonelli's theorem now yields
\begin{align*}
&\nu_{d,d} \left( \mathbb U^c(\Lambda)\right) 
\quad\le\quad
[\nu_{d,p}(\mathbb G^c)]^{-1/2} \int_{\mathcal O_d}\nu_{d,p}(\mathbb J^c(U\Lambda U'))\;\nu_{d,d} (dU)\\
&\quad=\quad
[\nu_{d,p}(\mathbb G^c)]^{-1/2} \int_{\mathcal O_d\times \mathcal O_p}\nu_{d,p}(\mathbb J^c(U\Lambda U'))\;(\nu_{d,d}\otimes \nu_{p,p}) (dU,dV)\\
&\quad=\quad [\nu_{d,p}(\mathbb G^c)]^{1/2},
\end{align*}
which finishes the proof.
\end{proof}


\section{Proofs for Section~\ref{s3.1}}
\label{appendixA}

\begin{lemma} \label{measurability}
For $Z$, $d$ and $p$ as in Section~\ref{s3.1},
there are measurable
functions $\mu : \R^p \times \mathcal V_{d,p} \to \R^d$, 
$\Delta : \R^p \times
\mathcal V_{d,p} \to \R^{d\times d}$ and $h : \R^p\times \mathcal V_{d,p} \to
\bar{\R}$ such that $\mu(x,B) = \E [Z \| B'Z = x]$, $\Delta(x,B) =
\E[ZZ'\|B'Z=x] - (I_d + B(xx'-I_p)B')$ and $h(x,B) = \E\left[
\frac{f(W^{x|B})}{\phi(W^{x|B})}\right]$, respectively, where $W^{x|B} = Bx +
(I_d-BB')V$ and $\phi$ denotes a Lebesgue density of $V\thicksim
N(0,I_d)$.
\end{lemma}

\begin{proof} 
Let $\mathbf B$ be uniformly distributed on $\mathcal V_{d,p}$
and let $z_j$ denote the $j$-th component of $Z$,
$j=1,\dots, d$.  Clearly,
$\E [ z_j \| \mathbf B, \mathbf B'Z] = m_j(\mathbf B, \mathbf B'Z)$ for some
measurable function $m_j: \mathcal V_{d,p} \times \R^p \to \R$. Now Lemma B.1
in \citep{Lee12b} entails that for $B\in \mathcal V_{d,p}$, the function $m_j$
satisfies $m_j(B, B'Z) = \E[z_j\| B'Z]$ (a.s.). In other words, for every $B\in
\mathcal V_{d,p}$, $m_j(B,x)$ satisfies the definition of a conditional
expectation of $z_j$ given $B'Z = x$. Setting $\mu_{x|B} = (m_1(B,x),\dots,
m_d(B,x))'$ finishes the first part.

For existence of $\Delta$ it suffices to show existence of a measurable
function $D :  \R^p \times \mathcal V_{d,p} \to \R^{d\times d}$ such that
$D(x,B) = \E[ZZ'\|B'Z=x]$, since $(I_d + B(xx'-I_p)B')$ is continuous from
$\R^p \times \mathcal V_{d,p}$ to $\R^{d\times d}$. To do this, argue as in the
preceding paragraph but with $z_j$ replaced by $z_iz_j$ for $(i,j) \in
\{1,\dots, d\}^2$.

The quantity $\E[ f(W^{x|B})/\phi(W^{x|B})]$, as a function of $x$ and $B$,
is uniquely defined and measurable in view of Tonelli's theorem since
$g(x,B,v)=Bx + (I_d-BB')v$ is continuous on $\R^p\times \mathcal V_{d,p} \times
\R^d$, and hence, measurable.  
\end{proof}


\begin{lemma}\label{final}
\begin{list}{\thercnt}{
	\usecounter{rcnt}
	\setlength\itemindent{5pt}
	\setlength\leftmargin{0pt}
	\setlength\itemsep{\medskipamount}
}
\item \label{final.i}
Suppose \eqref{th1a} is bounded by \eqref{bound}, for $k=2$ and 
for every $M>1$ and all $p\in \N$ such that 
$d > \max\{4(k+p+1)M^4, 2k +p(2k+2)2^{k+3},p^2\}$,
where $\kappa=\kappa_1\geq 1$ depends only on $\alpha$ and $\beta$, and
where $g = g_1$.
Then the conclusion of Theorem~\ref{thm:main}\ref{theoremA}
holds with $\gamma_1 = \max\{g_1, 6 + 2 \log(2 D\sqrt{\pi e})\}$.

\item \label{final.ii}
Suppose \eqref{th1a} and \eqref{th2a} are both 
bounded by \eqref{bound}, for $k=4$
and for every $M>1$ and all $p\in \N$ 
such that $d > \max\{4(k+p+1)M^4, 2k +p(2k+2)2^{k+3},p^2\}$,
where $\kappa=\kappa_2\geq 1$ depends only on $\alpha$ and $\beta$,
and where $g = g_2$.
Then the conclusion of Theorem~\ref{thm:main}\ref{theoremB} holds
with $\gamma_2 = \max\{g_2, 10  + 4 \log(2 D\sqrt{\pi e})\}$.

\item \label{final.iii}
In both cases \ref{final.i} and \ref{final.ii} the set $\mathbb G$ can be chosen 
in such a way that the conclusion of Theorem~\ref{thm:main}\ref{theoremC}
holds true.
\end{list}
\end{lemma}

\begin{proof}
For part~\ref{final.i}, 
let $k=2$, $p < d$, $\tau \in (0,1)$, and 
$\xi_1 = \min\{\xi, \eps/2+1/4,1/2\}/3$, as in 
Theorem~\ref{thm:main}\ref{theoremA}. 
We first note that \eqref{bound} can be re-written and further bounded as
\begin{align}
&p^{2k+1+\eps} e^{g M^2} \left(2D\sqrt{\pi e}\right)^{pk}
	d^{-\min\{\xi, \eps/2+1/4,1/2\}} 
	\kappa 
	\notag\\
&\quad = \quad
	p^{(5+\eps)} e^{g_1 M^2} \left(2D\sqrt{\pi e}\right)^{2 p}
	d^{-3 \xi_1} 
	\kappa_1
	\notag\\
&\quad \le \quad e^{g_1 M^2 + p[6+2\log(2D\sqrt{\pi e})]}
	d^{-3 \xi_1}\kappa_1
\quad \le \quad e^{\gamma_1 M^2+ \gamma_1 p}d^{-3 \xi_1}\kappa_1.
	\label{bound1}
\end{align}
We introduce the parameters $\tau_1>0$ and
$\tau_2 > 0$ to fine-tune our bounds. 
Let $M_d = \sqrt{\tau_1 \log{d}/\gamma_1}$ and
$\delta_d = d^{-\tau_2}$. Define the set $\mathbb G$
as $\mathbb G = \mathcal V_{d,p}$ in case $M_d \leq 1$, and as
\begin{align*}
\mathbb G \quad= \quad  
\Big\{ 
	B \in \mathcal V_{d,p} : 
	\int\limits_{\|x\|\le M_d} 
		\|\E[Z\|B'Z=x] - Bx\|^2
	h(x|B)^2 \phi_p(x) dx 
	\le \delta_d
\Big\}
\end{align*}
in case $M_d > 1$,
where $dx$ denotes integration with respect to
Lebesgue measure on $\R^{p}$ and $\phi_p$
denotes a Lebesgue density of the standard normal distribution on $\R^p$.
Note that $\mathbb G$ depends on the distribution of $Z$
(through $h(x|B)$ and through the conditional expectation), and also on
$\tau_1$ and $\tau_2$, in case $M_d > 1$.
In either case, $\mathbb G$ is a Borel subset of $\mathcal V_{d,p}$
(in view of Lemma~\ref{measurability} and Tonelli's theorem). 
For any $t>0$ and $B \in \mathbb G$, we can bound the probability in  \eqref{mean}
as follows:
In case $M_d > 1$, we have
\begin{align}
&\P\left( \|\E[Z\|B'Z] - BB'Z\| > t \right)\notag\\
&\quad \le \quad
\P\left(\|\E[Z\|B'Z] - BB'Z\| > t , \|B'Z\| \le M_d\right) 
\,\, + \,\, \P\left( \|B'Z\| > M_d \right) \notag\\
&\quad \le\quad
\frac{1}{t} 
\int\limits_{\|x\| \le M_d} 
	\|\E[Z\|B'Z=x] - Bx\|
\,d\P_{B'Z}(x) \,\, +\,\,
\P\left( \|B'Z\|^2 > M_d^2 \right)\notag \\
& \quad\le\quad
\frac{1}{t} 
\int\limits_{\|x\| \le M_d} 
	\|\E[Z\|B'Z = x] - Bx\|
	h(x|B)\phi_{p}(x)
\,dx 
 + \frac{1}{M_d^2} \E \left[ \trace BB'ZZ' \right] \notag\\
&\quad \le\quad
\frac{1}{t} 
\left(
\int\limits_{\|x\| \le M_d} 
	\|\E[Z\|B'Z = x] - Bx\|^2
	h(x|B)^2\phi_{p}(x)
\,dx \right)^{1/2} \,\,
+ \,\,\frac{p}{M_d^2}\notag \\
&\quad \le\quad
\frac{\delta_d^{1/2}}{t}
\,\,+ \,\,\frac{p}{M_d^2} 
\quad=\quad
\frac{1}{d^{\tau_2/2} t}
\,\,+ \,\,\frac{\gamma_1}{\tau_1}\frac{p}{\log d},
	\label{bound3}
\end{align}
where the second-to-last inequality follows, for example, from Jensen's
inequality for concave functions, and we have used that $x \mapsto
h(x|B)\phi_p(x)$ is a Lebesgue density of $B'Z$ (cf. Proposition~\ref{p1}). 
And in the case where $M_d \leq 1$, the probability in \eqref{mean} is 
also bounded by \eqref{bound3}, because that upper bound is larger than
one in this case.
Finally, the probability of the complement of $\mathbb G$ can be re-written and 
bounded as
\begin{align}
\nu_{d,p}&\left( \mathbb G^c \right) \quad=  \quad
\nu_{d,p} 
\left( 
	\int\limits_{\|x\|\le M_d} 
	\|\E[Z\|B'Z=x] - Bx\|^2
	h(x|B)^2 \phi_{p}(x) dx 
	> \delta_d
\right)\notag \\
&\le\quad
\frac{1}{\delta_d} 
\int 
\int\limits_{\|x\|\le M_d} 
	\left[ 
		\|\mu_{x|B}\|^2 - \|x\|^2
	\right] 
	h(x|B)^2 \phi_{p}(x) 
	\,dx\, d\nu_{d,p}(B) \notag\\
&=\quad
\frac{1}{\delta_d} 
\int\limits_{\|x\|\le M_d} 
\int 
	\left[ 
		\|\mu_{x|B}\|^2 - \|x\|^2
	\right] 
	h(x|B)^2  
	\,d\nu_{d,p}(B)\, \phi_{p}(x) dx\notag \\
&\le\quad
\frac{1}{\delta_d} 
\sup_{\|x\| \le M_d} 
	\int \left[ 
			\|\mu_{x|B}\|^2 - \|x\|^2
			\right] 
			h(x|B)^2  
	\,d\nu_{d,p}(B),\label{sup}
\end{align}
where we have used Markov's inequality and Tonelli's theorem.

Consider first the case where $M_d > 1$ and where
$d > \max\{4(k+p+1) M_d^4, 2 k + p(2 k+2)2^{k+3}, p^2\}$.
In that case, \eqref{th1a} is bounded by \eqref{bound}
with $M_d$ replacing $M$ by assumption,
and hence also
the supremum in \eqref{sup} is bounded by \eqref{bound1} with $M_d$
replacing $M$.
This entails that
\begin{align}
\nu_{d,p}&\left( \mathbb G^c \right) \quad\leq\quad
\frac{1}{\delta_d} e^{\gamma_1 M_d^2+ p \gamma_1}d^{-3 \xi_1}\,\kappa_1 
\quad=\quad
\kappa_1 d^{-3 \xi_1 + \tau_1+\tau_2 + \gamma_1 p /\log d}.
\label{bound2} 
\end{align}
Now
we see that there is a trade-off between the bound in \eqref{bound2} and the
one in \eqref{bound3} in the sense that one can be made smaller at the expense
of the other. We here choose the tuning parameters $\tau_1$ and $\tau_2$ such
that both bounds are of the same leading order in $d$, i.e., we choose
$\tau_1$ so that $\tau_1+\tau_2 - 3\xi_1 = -\tau_2/2$. 
And with the constant $\tau$ chosen earlier, we choose $\tau_2$ so that
$\tau_2/2 = \tau \xi_1$.
With this choice of $\tau_1$ and $\tau_2$, it is now elementary to
verify that the bounds \eqref{bound2} and \eqref{bound3} 
are of the form claimed by Theorem~\ref{thm:main}\ref{theoremA}.

Consider next the case where $M_d \leq 1$.
In that case, we have $\mathbb G = \mathcal V_{d,p}$ by construction, so that
$\nu_{d,p}(\mathbb G^c)$ equals zero and hence is trivially bounded  as claimed.
Moreover, we have also seen that the probability in \eqref{mean}
is here trivially bounded by the expression in \eqref{bound3}
and hence also by the bound given in Theorem~\ref{thm:main}\ref{theoremA}
upon choosing $\tau_1$ and $\tau_2$ as in the preceding paragraph.

Finally, consider the case where $M_d > 1$ and where 
$d \leq  \max\{4(k+p+1) M_d^4, 2 k + p(2 k+2)2^{k+3}, p^2\}$.
With the constants $\tau_1$ and $\tau_2$ as chosen before,
the upper bound in \eqref{bound3}
is again as claimed in Theorem~\ref{thm:main}\ref{theoremA}.
It remains to show that the upper bound in \eqref{bound2} is trivial,
i.e., not less than one, here.
The expression in \eqref{bound2} can be written as
$\kappa_1 d^{-\tau \xi_1 + \gamma_1 p/\log d}$
upon noting that $\tau_1 + \tau_2 - 3 \xi_1 = -\tau \xi_1$.
It is now easy to see that \eqref{bound2} is not less than one if
\begin{equation} \label{suff}
\log d \quad \leq \quad 6 \gamma_1 p,
\end{equation}
upon noting that $\kappa_1 \geq 1$, that
$\tau < 1$, and that $\xi_1 \leq 1/6$.
If $d \leq p^2$, we have $\log d \leq 2 p$, and the relation in \eqref{suff}
holds because $2 p \leq 36 p\leq 6 \gamma_1 p$ by our choice of $\gamma_1$.
If $d \leq 2 k + p (2 k + 2) 2^{k+3}$, or, equivalently, if
$d \leq 4+192 p$ (recall that $k=2$), then
we have $\log d \leq \log(192 p) \leq 36 p$ (where the last inequality
is elementary to verify). In particular, we see that \eqref{suff} holds
also in this case.
Lastly, consider the case where $d \leq 4(k+p+1) M_d^4$.
Using the definitions of $M_d$ and $\tau_1$,
together with the facts that $\tau \in (0,1)$, that $\xi_1 \leq 1/6$,
and that $\gamma_1 \geq 6$, we see that
$d / (\log d)^2 \leq (p+3)/36 \leq p/9$ or, equivalently,
that $324 d / (\log d)^2 \leq 36 p$.
It is now elementary to verify that $\log d \leq 324 d / (\log d)^2$,
such that, again, the relation \eqref{suff} is satisfied.
This completes the proof of part \ref{final.i}.

The derivation of part~\ref{final.ii} is similar:
Set $k=4$, $ p < d$, $\tau \in (0,1)$, and 
$\xi_2 = \min\{\xi, \eps/2+1/4,1/2\}/5$.
Note that \eqref{bound} can here be written and bounded as
\begin{equation}\label{bound1x}
p^{2k+1+\eps} e^{g M^2} \left(2D\sqrt{\pi e}\right)^{pk}
	d^{-\min\{\xi, \eps/2+1/4,1/2\}} 
	\kappa
\,\, \le \,\, e^{\gamma_2 M^2+ \gamma_2 p}d^{-5 \xi_2}\kappa_2
\end{equation}
by arguing as in \eqref{bound1}.
For tuning-parameters $\tau_1>0$ and $\tau_2>0$,
set $M_d = \sqrt{ \tau_1 \log d / \gamma_2}$ and
$\delta_d = d^{-\tau_2}$. 

For bounding \eqref{mean}, we argue as in the proof of part~\ref{final.i} 
to obtain
a Borel set $\mathbb G_1$ so that \eqref{mean} with $\mathbb G_1$ replacing $\mathbb G$
is bounded as claimed in Theorem~\ref{thm:main}\ref{theoremB},
and so that $\nu_{d,p}(\mathbb G_1^c)$ is bounded by half the
upper bound given in Theorem~\ref{thm:main}\ref{theoremB}
(i.e., with $\kappa_2$ replacing $2 \kappa_2$).

If we can also find a Borel set $\mathbb G_2$ so that \eqref{variance} with
$\mathbb G_2$ replacing $\mathbb G$ is bounded as claimed, and so that
$\nu_{d,p}(\mathbb G_2^c)$ is bounded by half the bound given in
the theorem, then the proof is completed by setting $\mathbb G = \mathbb G_1 \cap \mathbb G_2$.
To this end, set $\mathbb G_2 = \mathcal V_{d,p}$ if $M_d \leq 1$, and set
$$
\mathbb G_2 \quad=\quad
\Big\{ 
	B \in \mathcal V_{d,p} : 
	\int\limits_{\|x\|\le M_d} 
	\| \Delta_{x|B}\|^4
	h(x|B)^4 \phi_{p}(x) dx 
	\le \delta_d
\Big\}
$$
otherwise.
For fixed $g>0$ and $B \in \mathbb G_2$, the probability in \eqref{variance}
can be bounded as follows: In case $M_d>1$, we have
\begin{equation}\label{bound3x}
\P \left( \|\Delta_{B'Z|B}\| > t \right)
\quad\leq\quad \frac{\delta_d^{1/4}}{t} + \frac{p}{M_d^2}
\quad=\quad \frac{1}{d^{\tau_2/4} t} + \frac{\gamma_2}{\tau_1}\frac{p}{\log d}.
\end{equation}
by arguing as in \eqref{bound3}. And in case $M_d \leq 1$,
the probability in \eqref{variance} is trivially bounded as 
in \eqref{bound3x}, because the upper bound is larger than one.
Finally, for $\nu_{d,p}(\mathbb G_2^c)$, we have
\begin{equation}\label{supx}
\nu_{d,p}(\mathbb G_2^c)\quad \leq\quad  
	\frac{1}{\delta_d} \sup_{\|x\|\leq M_d}
	\int \trace \Delta_{x|B}^4 h(x|B)^4 \, d \nu_{d,p}(B),
\end{equation}
by arguing  as in \eqref{sup} and upon noting that
$\|\Delta_{x|B}\|^4 \leq \trace \Delta_{x|B}^4$.

In the case where $M_d > 1$ and where 
$d > \max\{ 4 (k+p+1) M_d^4, 2 k + p(2 k + 2)2^{k+3}, p^2\}$,
we see that
\begin{equation}\label{bound2x}
\nu_{d,p}(\mathbb G_2^c) \quad\leq\quad
\frac{1}{\delta_d} e^{ \gamma_2 M_d^2 + \gamma_2 p} d^{-5 \xi_2} \kappa_2
\quad = \quad
\kappa_2 d^{-5 \xi_2 + \tau_1 + \tau_2 + \gamma_2 p / \log d}
\end{equation}
upon using \eqref{supx}, the assumption of the lemma, and \eqref{bound1x}.
We balance the upper bounds in \eqref{bound3x} and \eqref{bound2x}
by choosing $\tau_1$ and $\tau_2$ so that
$\tau_1 + \tau_2 - 5 \xi_2 = -\tau_2/4 = -\tau \xi_2$, resulting
in upper bounds as desired.
The remaining cases, i.e., the case where $M_d \leq 1$, and
the case where $M_d > 1$ and 
$d \leq \max\{4(k+p+1)M_d^4, 2 k + p(2 k + 2)2^{k+3}, p^2\}$,
are treated as in the proof of part~\ref{final.i}, mutatis mutandis.

The additional properties of the sets $\mathbb G$ claimed in part~\ref{final.iii}
can be established directly from the definition of these sets in parts~\ref{final.i}
and \ref{final.ii}, respectively. However, it is even easier to consider
somewhat larger sets instead. Indeed, consider the collection of matrices
$B$ on the Stiefel manifold, such that the probability in \eqref{mean} 
is bounded by \eqref{eq:boundMean} and the set of matrices $B$ such that both 
probabilities in \eqref{mean} and \eqref{variance} are bounded by 
\eqref{eq:boundVariance}. Clearly, these sets also satisfy the conclusions of 
Theorem~\ref{thm:main}\ref{theoremA} and Theorem~\ref{thm:main}\ref{theoremB}, 
respectively. Now the right-rotation invariance is immediate because 
conditioning on $B'Z$ and on $QB'Z$, for some orthogonal $p\times p$ 
matrix $Q$ is equivalent, and therefore the probabilities in \eqref{mean} and 
\eqref{variance} are invariant under right-rotation of $B$. For the left-rotation 
equivariance, first recall that the upper bounds \eqref{eq:boundMean} and 
\eqref{eq:boundVariance} depend on the distribution of $Z$ only through the 
quantities $\eps$, $\xi$ and $D$ of conditions~\ref{c.Sk} and \ref{c.density} and 
these conditions are invariant under rotation of $Z$. In particular, we see that 
the upper bounds \eqref{eq:boundMean} and \eqref{eq:boundVariance} are 
invariant under rotation of $Z$. On the other hand, it is easy to see that 
transformation of the distribution of $Z$ by a $d\times d$ orthogonal matrix 
$R$ has the same effect on the probabilities in \eqref{mean} and \eqref{variance} 
as left-multiplication of $B$ by $R'$. Thus, the left-rotation equivariance of 
the defined sets follows.
\end{proof}

\section{Proofs for Section~\ref{s3.2}}
\label{appendixB}


\begin{proof}[\bf Proof of Proposition~\ref{p1}]
Define $T:\R^d \to \R^p$ by $T(v) = B'v$, 
write $\lambda_d$ for the Lebesgue measure on $\R^d$, and
write $\P_Z$, $\P_{T(Z)}$, etc. 
and $\E_Z$, $\E_{T(Z)}$, etc. for the 
push-forward measures and the corresponding expectation operators,
respectively,
induced by the indicated random variables.
With this, $\P_{T(V)}$ is the $N(0,I_p)$-distribution,
and $\P_{W^{x|B}}$ is the $N(B x, I_d-B'B)$-distribution and also
a regular
conditional distribution of $V$ given $B'V=x$, i.e., $\P_{W^{x|B}} =
\P_{V\|B'V=x}$ (cf. \citep[Relation 10.7.6]{Hof94b}). 

For the first statement, we have
$\P_Z \ll \lambda_d$ with $\frac{d \P_Z}{d \lambda_d} = f$, 
and hence also 
$\P_Z \ll \P_V$ with $\frac{d \P_Z}{d\P_V} = \frac{f}{\phi}$. 
It follows that 
$\P_{T(Z)} \ll \P_{T(V)}$,  and that
\begin{align} \label{p1.1}
\frac{d\P_{T(Z)}}{d \P_{T(V)}}(x) 
\;\;=\;\; \E_V\left[ \frac{f}{\phi} \Big \| T = x\right] 
\;\;=\;\; \E\left[ \frac{f}{\phi}(V) \Big \| B'V = x\right]  \hspace{.5cm} 
	\lambda_p\text{-a.e.},
\end{align}
where the first equality is obtained from, say, 
\citep[Relation 10.11.1]{Hof94b},  and
where the second equality holds because one easily verifies that the
left-hand-side satisfies the definition of the conditional expectation on the
right-hand-side.
On the far left-hand-side of \eqref{p1.1},
we already have a density
$h(x|B)$ of $B'Z$ with respect to the
$p$-variate standard Gaussian measure. 
The expression on the far right-hand side of \eqref{p1.1} can be written as
$\int f(v)/\phi(v) d \P_{V\|B'V=x}(v) = 
\int f(v)/\phi(v) d \P_{W^{x|B}}(v) = \E[ f(W^{x|B}) / \phi(W^{x|B})]$
$\lambda_p$-a.e., as desired
(cf., say, \citep[Relation 10.3.5]{Hof94b}).

For the second statement
fix a version of $h(x|B)$ and set $D = \{x \in \R^p : 0 <
h(x|B) < \infty \}$. For a Borel set $F\subseteq \R^d$, set
$$
\P_{Z\|B'Z=x}(F) \quad=\quad
\int_F \frac{f}{\phi}(v)d\P_{V\|B'V=x}(v) / h(x|B)
$$
if $x\in D$, and set $\P_{Z\|B'Z=x}(F)=\P_Z(F)$ otherwise.
Note that $\P_{Z\|B'Z=x}$ is a regular conditional distribution of $Z$ 
given $B'Z=x$; cf.  \citep[Relation 10.11.3]{Hof94b}.
Recalling that $\P_{V\|B'V=x} = \P_{W^{x|B}}$, this entails that
$\E[ \Psi(Z) \| B'Z=x]$ can be computed as
$$
\int \Psi(v) \frac{f}{\phi}(v) d \P_{W^{x|B}}(v) / h(x|B)
\,\,=\,\, 
\E \left[ \Psi(W^{x|B}) \frac{f(W^{x|B})}{\phi(W^{x|B})} \right] 
	/h(x|B)
$$
whenever $x \in D$. This proves the second statement unless $h(x|B)=0$.
But if $h(x|B)=0$, then $f/\phi = 0$ $\P_{W^{x|B}}$-a.e. 
in view of \eqref{p1.1}, and hence 
also $\E [ \Psi(W^{x|B}) \frac{f(W^{x|B})}{\phi(W^{x|B})}] = 0$.
\end{proof}

\begin{lemma} \label{trace}
The expression in
\eqref{th2aa}
can be decomposed as the weighted sum of the expressions in \eqref{th2aaa},
provided all expected values that occur are finite.
The number of summands in this decomposition
depends only on $k$, and the weight of each summand only depends
on k and on $\|x\|^2$, and is a polynomial in $\|x\|^2$.
\end{lemma}

\begin{proof}
Write $\Upsilon$ as an abbreviation for the product of density ratios in
\eqref{th2aa} and set $X_j = W_j W_j'-I_d$ and 
$Y=\mathbf B(xx'-I_p)\mathbf B'$. 
Then the quantity of interest can be written
as $\E[ \trace \prod_{j=1}^k (X_j - Y)\Upsilon]$ or, equivalently, as
\begin{align}\label{ExpTrace1}
\begin{split}
&\sum_{l=1}^k (-1)^{k-l}\sum_{1\le j_1<\ldots<j_l\le k} \E\left[  \trace
\,\,X_{j_1} Y^{j_2-j_1-1}X_{j_2} \cdots X_{j_l} Y^{k-j_l+j_1-1} \; \Upsilon
\right]
\\ 
&\quad\quad +  (-1)^k \E[ \trace Y^k\;\Upsilon],
\end{split}
\end{align}
where we adopt the convention that $Y^0=I_d$ and where we have moved a possibly
leading $Y$-term to the end of the product. 
In this sum, $l$
indicates the number of occurrences of $X_j$ in the product. 

We first show
that \eqref{ExpTrace1}
can be written as the sum of
\begin{align}
\label{reducedLeq0}
& \trace\left\{ (-1)^k(xx'-I_p)^k\E[\Upsilon] \right\}, 
\\
\label{reducedLeqK}
&\trace\left\{ \E \left[X_1X_2\cdots X_k  \; \Upsilon\right] \right\},
	\text{ and} 
\\
\begin{split}
\label{reducedLlessK} 
&\trace  \Bigg\{ \sum_{l=1}^{k-1} (-1)^{k-l} (xx'-I_p)^{k-l}    
\\
&  \quad\quad\quad\times \,\, \sum_{1\le j_1<\ldots<j_l\le k} \E
\Big[\big( 
	\prod_{i=1}^m \mathbf B'X_{f_i}X_{f_i+1}\cdots X_{f_{i+1}-1}\mathbf B  
	\big)
	\;
\Upsilon \Big]  \Bigg\},
\end{split}
\end{align}
where the indices $m$ and $f_1,\dots, f_{m+1}$ in \eqref{reducedLlessK}
depend  on the indices $j_1,\dots, j_{l}$ in \eqref{reducedLlessK}, and they
satisfy $m\ge 1$ and $1=f_1<f_2<\ldots<f_{m+1}=l+1$.
Obviously, 
\eqref{reducedLeq0} equals the last term in \eqref{ExpTrace1},
and 
\eqref{reducedLeqK} is equal to the summand in \eqref{ExpTrace1}
with $l=k$,.
Now, for $0<l<k$ consider the corresponding
expected value in \eqref{ExpTrace1}, i.e.,  
\begin{align*}
\E\left[  \trace \,\,X_{j_1} Y^{j_2-j_1-1}X_{j_2} \cdots X_{j_l}
Y^{k-j_l+j_1-1} \; \Upsilon \right], 
\end{align*}
for a fixed set of indices $1\le j_1< \ldots < j_l \le k$.
At least one of the exponents of the $Y$-terms is non-zero (because $l<k$) and
there is at least one $X_i$-factor in the product (because $l>0$).
Hence, the matrix product inside the expectation can be seen as a product of
blocks, each consisting of one or more consecutive $X_j$ followed
by a non-zero power of $Y$. Since the $W_j$ are conditionally i.i.d.
given $B$,
we may assume that 
$X_{j_i} = X_i$ for $i=1,\dots, l$.
Furthermore,
$Y^m = [\mathbf B(xx'-I_p)\mathbf B']^m = \mathbf B(xx'-I_p)^m\mathbf B'$ for
any integer $m\ge 1$. Therefore, the expression in the preceding display can be
written as
\begin{align*}
\E \left[\trace \prod_{i=1}^m \mathbf B'X_{f_i}X_{f_i+1}\cdots 
	X_{f_{i+1}-1}\mathbf B(xx'-I_p)^{\upsilon_i} \; \Upsilon\right],
\end{align*}
where $m$ indicates the number of blocks as just described 
and satisfies $1\le m \le \lfloor k/2 \rfloor$, 
where the $f_i$ obey
$1=f_1<f_2<\ldots<f_{m+1}=l+1$,  
and where $f_{i+1}-f_i \geq 1$ and $\upsilon_i\ge 1$ 
is the number of $X_j$ 
and the power of $Y$, respectively,  in the $i$-th block.
Set $n_i = f_{i+1}-f_i$. We note that
$\sum_{i=1}^m n_i = l$ and that $\sum_{i=1}^m \upsilon_i = k-l$. 
If the matrices $\mathbf B'X_{f_i}X_{f_i+1}\cdots X_{f_{i+1}-1}\mathbf B$ and
$(xx'-I_p)^{\upsilon_i}$ commute, then we can write the matrix product 
in the preceding display as
\begin{align*}
&(xx'-I_p)^{k-l} \; \prod_{i=1}^m \mathbf B'X_{f_i}X_{f_i+1}\cdots
X_{f_{i+1}-1}\mathbf B. 
\end{align*}
To show that these matrices do commute,
we note that $\mathbf B'W_j = x$. We can thus write
$\mathbf B'X_{f_i}X_{f_i+1}\cdots X_{f_{i+1}-1}\mathbf B$ as
\begin{align*}
&=\quad \mathbf B'  \left[ (-1)^{n_i}I_d\,+\,
\sum_{g=1}^{n_i} (-1)^{n_i-g} \sum_{f_i\le u_1 <
\ldots < u_g\le f_{i+1}-1} W_{u_1}W_{u_1}' \cdots W_{u_g}W_{u_g}' 
\right] \mathbf B \\
&=\quad
(-1)^{n_i}I_p  \,+\,
(-1)^{n_i-1} \,n_i\, xx'
\\
&\quad\quad\quad\,+\,
xx'  \sum_{g=2}^{n_i} (-1)^{n_i-g} \sum_{f_i\le u_1 < \ldots < u_g\le
f_{i+1}-1} W_{u_1}' W_{u_2} W_{u_2}' \cdots  W_{u_{g-1}} W_{u_{g-1}}' 
W_{u_g}.
\end{align*}
The expression in the preceding display can be written as
$\sum_{g=0}^{n_i} c_{i,g}$ with
$c_{i,0} = (-1)^{n_i}I_p$, $c_{i,1} = (-1)^{n_i-1}n_i xx'$, 
and with $c_{i,g} =
(-1)^{n_i-g} xx'\sum W_{u_1}' \cdots W_{u_g}$ if $g\ge2$. 
Note that each of the $c_{i,g}$ is a scalar multiple of
$I_p$ or of the matrix $x x'$. And note that $(x x' - I_p)$ commutes
with $x x'$.  Therefore
$(xx'-I_p)^{\upsilon_i}$ commutes with all the $c_{i,g}$ and hence with
$\mathbf B'X_{f_i}X_{f_i+1}\cdots X_{f_{i+1}-1}\mathbf B$, as desired.
This entails that the expression in \eqref{ExpTrace1} is indeed given
by the sum of 
\eqref{reducedLeq0}, 
\eqref{reducedLeqK}, 
and  
\eqref{reducedLlessK}.

For the next step, we note that the sum of 
\eqref{reducedLeq0}, \eqref{reducedLeqK}, and  
\eqref{reducedLlessK}
can also be written as the sum of
\begin{align}
\label{circle}
&\trace\left\{ \E \left[ \left( X_1X_2\cdots X_k - [\mathbf B(xx'-I_p)\mathbf
B']^k \right) \; \Upsilon\right] \right\} \text{ and}
\\
\begin{split} 
\label{strips}
&\trace  \Bigg\{ \sum_{l=1}^{k-1} (-1)^{k-l} (xx'-I_p)^{k-l}  
\sum_{1\le i_1<\ldots<i_l\le k}  \E \Bigg[
\\
& \quad \quad\quad\quad
\left(\prod_{i=1}^m \mathbf B'X_{f_i}X_{f_i+1}\cdots X_{f_{i+1}-1}\mathbf B  -
(xx'-I_p)^l \right) \; \Upsilon \Bigg]  \Bigg\}, 
\end{split}
\end{align}
which is elementary to verify
upon writing $(-1)^k (x x' - I_p)^k \Upsilon$ as
$-(x x' - I_p)^k \Upsilon \sum_{l=1}^k {k \choose l}(-1)^{k-l}$.

We now expand \eqref{circle} as
\begin{align*}
&\trace\,\,\E \left[ \left( (W_1W_1' - I_d)(W_2W_2'-I_d)\cdots(W_kW_k' - I_d)
- [\mathbf B(xx'-I_p)\mathbf B']^k \right) \; \Upsilon\right] \\
&\quad=\quad \E \left[ \left( (-1)^kd + \sum_{j=1}^k (-1)^{k-j} \binom{k}{j}
W_1'W_2\cdots W_jW_j' W_1\right. \right. \\
&\qquad\qquad\qquad
\left. \left. - (-1)^kp - \sum_{j=1}^k(-1)^{k-j}\binom{k}{j}
\|x\|^{2j}\right) \; \Upsilon\right]  \\
& \quad=\quad \E \left[  \left((-1)^k(d-p) + \sum_{j=1}^k (-1)^{k-j}\binom{k}{j}
(W_1'W_2\cdots W_jW_j' W_1 - \|x\|^{2j} )\right) \; \Upsilon\right]  \\
&\quad=\quad \E \left[  \left(\sum_{j=1}^k (-1)^{k-j}\binom{k}{j} (W_1'W_2\cdots
W_jW_j' W_1 - d + p- \|x\|^{2j} )\right) \; \Upsilon\right],
\end{align*}
using the exchangeability of the $W_i$ in the first equality,
and the relation $(-1)^k + \sum_{j=1}^k \binom{k}{j} (-1)^{k-j} = 0$ in
the last one.
In particular, we see that \eqref{circle} is equal to the first expression
in \eqref{th2aaa}.

For \eqref{strips}, we first consider each of the
expected values in \eqref{strips}. Expanding $\mathbf
B'X_{f_i}X_{f_i+1}\cdots X_{f_{i+1}-1}\mathbf B$ and $(xx'-I_p)^l =
\prod_{i=1}^m (xx'-I_p)^{n_i}$ similarly as above, each expected
value in \eqref{strips} can be written as
\begin{align*}
&\E \left[ \left(\prod_{i=1}^m \mathbf B'X_{f_i}X_{f_i+1}\cdots
X_{f_{i+1}-1}\mathbf B  - (xx'-I_p)^l \right) \; \Upsilon \right]  \\
&\quad= \quad\E \left[ \left( \prod_{i=1}^m \left( (-1)^{n_i}I_p + xx'
\sum_{g=1}^{n_i} (-1)^{n_i-g} \binom{n_i}{g}W_{f_i}'W_{f_i+1}W_{f_i+1}' \cdots
W_{f_i+g-1} \right) \right. \right.\\
&\quad\quad \quad\left. \left. - \prod_{i=1}^m \left(  (-1)^{n_i}I_p  + xx'
\sum_{g=1}^{n_i} (-1)^{n_i-g} \binom{n_i}{g} \|x\|^{2(g-1)} \right) \right) \;
\Upsilon \right],
\end{align*}
where we adopt the convention that 
$W_{f_i}'W_{f_i+1}W_{f_i+1}' \cdots W_{f_i+g-1}$ is to be
interpreted as 1 in case $g=1$. 
Now we write the expression in the preceding
display more compactly as
\begin{align*}
\E  \left[ \left( \prod_{i=1}^m  \sum_{g=0}^{n_i} a_{i,g}\gamma_{i,g}  -
\prod_{i=1}^m  \sum_{g=0}^{n_i} a_{i,g} \delta_{i,g}    \right) \; \Upsilon
\right],
\end{align*}
where $a_{i,0} = (-1)^{n_i}I_p$, $\gamma_{i,0}=\delta_{i,0}=1$, and for
$g=1,\ldots,n_i$, $a_{i,g}= xx'\binom{n_i}{g}(-1)^{n_i-g}$, 
$\gamma_{i,g}= W_{f_i}'W_{f_i+1}W_{f_i+1}' \cdots W_{f_i+g-1}$,
$\delta_{i,g}=\|x\|^{2(g-1)}$. 
Note that we  have $\gamma_{i,1} = 1$ by the convention 
adopted earlier.
Note that the $\gamma_{i,g}$ and 
the $\delta_{i,g}$ are
scalars and hence commute with the $a_{i,g}$. 
The expression in the preceding  display or, equivalently,
each expected value in \eqref{strips},
can now be re-written as
\begin{align*}
\sum_{g_1=0}^{n_1} \ldots \sum_{g_m=0}^{n_m} \left(\prod_{i=1}^m
a_{i,g_i}\ \right) \E  \left[ \left(\prod_{i=1}^m \gamma_{i,g_i}  -
\prod_{i=1}^m   \delta_{i,g_i}   \right)\; \Upsilon \right].
\end{align*}
To complete the proof, we show that the (scalar) expected value
in the preceding display is of the same form as the second expression
in \eqref{th2aaa}, and we show that the corresponding weight, i.e.,
the trace of $(x x'-I_p)^{k-l}\prod_{i=1}^m a_{i,g_i}$,
is as claimed.

To deal with the weight, we observe that 
we can factor out a matrix term $(x x')^m$ from the product of the
$a_{i,g_i}$, so that
$(x x'-I_p)^{k-l}\prod_{i=1}^m a_{i,g_i}$
equals $(xx'-I_p)^{k-l}(xx')^m$ times a scalar
that depends only on the $g_i$ and on the $n_i$,
and the collection of all such $g_i$ and $n_i$ depends
only on $k$.
Obviously, the trace of that matrix is a polynomial in $\|x\|^2$,
and the collection of all possible such polynomials depends only on $k$.

To deal with the expected
value in the preceding display, fix a collection $(g_1, \ldots,g_m)' \in
\{0,1,\ldots,n_1\}\times\cdots\times \{0,1,\ldots,n_m\}$. In case  $g_i\le 1$
for all $i=1,\ldots,m$, the expected value above is equal to zero, because of
$\gamma_{i,0}=\delta_{i,0}=\gamma_{i,1}=\delta_{i,1}=1$. In view of this
and of the exchangeability of the $W_j$,
we may assume that $m\geq 1$ and that
$g_i>1$ for each $i=1,\dots,m$, without loss of generality.
Set $j_0=0$, $j_i=j_{i-1}+g_i$ for
$i=1,\ldots,m$ and note that $\sum_{i=1}^{m}(g_i-1) = j_{m} - m$ and
$j_{m} \le \sum_{i=1}^{m} n_i = l <k$. We therefore see that
$\E [(\prod_{i=1}^m \gamma_{i,g_i} - \prod_{i=1}^m \delta_{i,g_i} )\Upsilon]$ 
equals
\begin{align*}
\E  \left[ \left( \prod_{i=1}^{m}  W_{j_{i-1}+1}'W_{j_{i-1}+2}\cdots
W_{j_i-1}'W_{j_i} - \|x\|^{2(j_{m}-m)} \right)\; \Upsilon \right],
\end{align*}
which coincides with the second expression in \eqref{th2aaa} with
indices $m$ and $j_0,\dots, j_m$ as claimed.
\end{proof}

\section{Proofs for Section~\ref{s3.3}}
\label{appendixC}

\begin{lemma}\label{normConst} 
For $d$, $p$, $k$, and $\eta(d,p,k)$ as in Proposition~\ref{p2}
with $k<d-p-1$, we have
\begin{align*}
\eta(d,p,k)\quad\leq \quad
\exp\left[ \frac{p^2}{d} \left( 1-\frac{p+k-1}{d}\right)^{-1}
\frac{k^2}{2}\right].
\end{align*}
\end{lemma}

\begin{proof}
Use the inequality
\begin{align*}
\frac{\Gamma(x)}{\Gamma(y)} < \frac{x^{x-1/2}}{y^{y-1/2}}e^{y-x},
\end{align*}
for $x\ge y>1$, proved in \citep[Theorem 1]{Kec71a}, to bound $\eta(d,p,k)$ by
\begin{align*}
\prod_{i=1}^k 
(d/2)^{-p/2}
\frac{
	\left(\frac{d-i+1}{2}\right)^{(d-i)/2} 
}
{
	\left(\frac{d-p-i+1}{2}\right)^{(d-p-i)/2}  
}e^{-p/2}.
\end{align*}
The $i$-th factor in the above product equals
\begin{align*}
\left(
\frac{d-i+1}{d-p-i+1}
\right)^{(d-i)/2}
\left(
\frac{d-p-i+1}{d}
\right)^{p/2}
e^{-p/2}.
\end{align*}
In this display,
write the first expression in parentheses as $1+a$ where
$a=p/(d-p-i+1)$, and note that $a>0$. 
The second factor in parentheses is bounded by one, so that
the expression in the preceding display is bounded by
$(1+a)^{d/2} e^{-p/2} = [\exp( d \log(1+a)-p)]^{1/2}$.
The argument of the exponential, i.e., $d \log(1+a)-p$,
can be further bounded  by 
$$
d a - p = p\frac{p+i-1}{d-p-i+1} \,\,\leq \,\,
p\frac{p+k-1}{d-p-k+1} \,\,\leq\,\, \frac{p^2}{d}\frac{k}{1-(p+k-1)/d},
$$
where the first inequality follows from the fact that the quantity
on the left is  increasing in $i$, and the second inequality holds
because $p+k-1 \leq p k$.
Together with the bound for $\eta(d,p,k)$ derived earlier,
this gives the desired upper bound.
\end{proof}

\begin{lemma}\label{Wj1}
For $1\le k \le d-p$, let $w_1,\ldots ,w_k$ be linearly independent 
vectors in $\R^d$, and write $N$ for the $d\times k$ matrix
$N=[w_1,\dots, w_k]$. Moreover, let $x\in \mathbb R^p$, and set 
$\iota = (1,\dots,1)'\in\R^k$. Then there exits a matrix 
$B\in \mathcal V_{d,p}$
such that $[N,B]$ has full rank $k+p$ and such that
\begin{align}
w_j = Bx + (I_d - BB')w_j \hspace{1cm} 
\label{wj}
\end{align}
holds for each $j=1,\dots, k$, 
if and only if $\|x\|^2\iota'(N'N)^{-1}\iota < 1$.
If either one of these equivalent conditions hold, then
the matrix
$A =(I_d-N(N'N)^{-1}N')B$ satisfies $\det(A'A)=1-\|x\|^2\iota'(N'N)^{-1}\iota$.
\label{equivalence}
\end{lemma}

\begin{proof}
Set $\eta = \iota'(N'N)^{-1}\iota$, and note that
one eigenvalue of $I_p - \eta xx'$ is $1-\eta \|x\|^2$ and
the others are $1$. In particular, that matrix is positive definite
if and only if $\|x\|^2 \eta < 1$.

Assume first that there exits a matrix $B$ with the given properties.
The relation \eqref{wj} then entails that $B' N = x\iota'$. For
$A$ as in the lemma, we get
\begin{align*}
A'A \quad= \quad I_p - x\iota'(N'N)^{-1}\iota x' \quad= \quad
I_p - \eta xx',
\end{align*}
and
it remains to show that $A'A$ is positive definite. 
To this end, fix a non-zero vector $v\in \R^p$, and note that
$Bv\notin \s\{N\}$ since $\s\{N\} \cap \s\{B\} = \{0\}$ and $B$ has 
full rank, by assumption. Therefore $Av \ne 0$ and thus $v'A'Av = \|Av\|^2 >0$. 

Conversely, assume that $\|x\|^2 \eta < 1$.
Choose $C \in \mathcal V_{d,p}$
such that $C'N=0$.
By assumption,
we see that the matrix $I_p - \eta xx'$ is positive definite and
thus has a non-vanishing square root. Setting
$B = N(N'N)^{-1}\iota x' + C (I_p- \eta xx')^{1/2}$,
it is now easy to see that
$B'B=I_p$. The
relation ($\ref{wj}$) follows by noting, for $j=1,\ldots,k$, that
\begin{align*}
B'w_j \,\,=\,\, x\iota'(N'N)^{-1}N'w_j + (I_p- \eta xx')^{1/2}C'w_j \,\,=
\,\,x\iota'(N'N)^{-1}N'Ne_j \,\,= \,\,x,
\end{align*}
where $e_j$  denotes the $j$-th standard basis vector in $\R^k$.
To see that $[N,B]$ has full rank $k+p$, simply
reverse the argument in the preceding paragraph.
\end{proof}


For $1\le p < d$, consider a $d\times p$-matrix
$B=[\beta_{1-p},\beta_{2-p},\ldots,\beta_0]$, a $d\times (d-p)$-matrix
$N=[w_1,w_2,\ldots,w_{d-p}]$ and $x\in\R^p$,
such that $B'B=I_p$, such that $M =[B,N]$ has
$\rank(M) = d$, and such that
$w_j = Bx + (I_d-BB')w_j$ for all $j=1,\ldots, d-p$.
Define the vectors
\begin{align}
&\beta_j \quad=\quad
\frac{(I_d-P_{B,\beta_1,\ldots,\beta_{j-1}})w_j}
	{\|(I_d-P_{B,\beta_1,\ldots,\beta_{j-1}})w_j\|} 
		\hspace{1cm}  &\text{for } j=1,\ldots,d-p,  \notag\\
& c_j \quad=\quad \frac{(I_d-P_{w_1,\ldots,w_{j-1}})w_j}
			{\|(I_d-P_{w_1,\ldots,w_{j-1}})w_j\|} 
				\hspace{1cm} &\text{for } j=1,\ldots,d-p, 
\label{ci}				\\
& c_j \quad=\quad \frac{(I_d-P_{c_{j+1},\ldots,c_0,c_1,\ldots,c_{d-p}})\beta_j}
	{\|(I_d-P_{c_{j+1},\ldots,c_0,c_1,\ldots,c_{d-p}})\beta_j\|} 
	\hspace{1cm} &\text{for } j=1-p,\ldots,0. \notag
\end{align}
Here $P_{\ldots}$ indicates the orthogonal projection 
on the span of the vectors
in the subscript, and we adopt the conventions
$\{y_1,\ldots,y_{1-1}\}=\{y_{0+1},\ldots,y_0\} = \emptyset$ and
$P_{\emptyset}=0$. 
With this, set $\tilde{B} = [\beta_1,\dots, \beta_{d-p}]$,
set $C = [c_{1-p}, \dots, c_0]$, and set
$\tilde{C} = [c_1,\dots, c_{d-p}]$.
By construction, the matrices $\mathcal B = [B, \tilde{B}]$
and $\mathcal C = [C,\tilde{C}]$ are both orthonormal.
We will consider rotated versions of $M = [B,N]$ that are given by
$\mathcal B' M$ and $\mathcal C' M$ and that we denote by $S$ and $T$,
respectively, with rows and
columns numbered from $1-p$ to $d-p$. 
In other words, we have
\begin{align*}
&S \quad=\quad (s_{i,j})_{i=1-p}^{d-p} \,\mbox{}_{j=1-p}^{d-p} \mbox{}
\quad=\quad \mathcal B'M 
\quad=\quad \begin{bmatrix} B'B &
B'N \\ \tilde{B}'B & \tilde{B}' N\end{bmatrix} 
,
\\
&T \quad=\quad
(t_{i,j})_{i=1-p}^{d-p} \,\mbox{}_{j=1-p}^{d-p} \mbox{}
\quad=\quad
\mathcal C'M \quad=\quad
\begin{bmatrix} C'B &
C'N \\ \tilde{C}'B & \tilde{C}' N\end{bmatrix}.
\end{align*}

\begin{lemma}\label{Wj2}
Let $B$, $N$, and $x$ as in the preceding paragraph.
For $M$, $\mathcal B$, $\mathcal C$, $T$, and $S$, which 
are functions of $B$, $N$, and $x$, the following holds true.
\begin{list}{\thercnt}{
	\usecounter{rcnt}
	\setlength\itemindent{5pt}
	\setlength\leftmargin{0pt}
	\setlength\itemsep{\medskipamount}
}
\item \label{Wj2i} We have $S=\begin{bmatrix} I_p & x \cdots x \\ 0
& \tilde{B}' N\end{bmatrix}$. Moreover, 
$\tilde{B}'N$ is
upper-triangular and its $k$-th column depends only on $w_1,\ldots,w_k$ and
$x$. In particular, S does not depend on B. 

\item \label{Wj2ii} We have $C'N=0$, and $\tilde{C}'N$ is upper-triangular. 
The $k$-th column of the matrix
$\mathcal B'\tilde{C}= \begin{bmatrix} B'\tilde{C} \\ \tilde{B}'\tilde{C}
\end{bmatrix}$ depends only on $w_1,\ldots,w_k$ and on $x$. Moreover,
the matrix $\tilde{B}'\tilde{C}$ is upper-triangular.

\item \label{Wj2iii} For $1\le k \le d-p$, set $\Lambda_k =
(c_i'\beta_j)_{i,j=1}^k$. Then $\Lambda_k$ is lower-triangular, depends only on
$w_1,\dots,w_k$ and on $x$, and satisfies  
\begin{align*}
\det(\Lambda_k\Lambda_k') = \prod_{i=1}^k (c_i'\beta_i)^2 = 1- \|x\|^2
\iota'\left[(w_i'w_j)_{i,j=1}^k\right]^{-1}\iota,	
\end{align*}
where $\iota = (1,\ldots,1)'$ denotes an appropriate vector of ones.
Moreover, if $k>1$ the vector $(t_{1,k},\ldots,t_{k-1,k})'$ can be written as
\begin{align*}
(t_{1,k},\ldots,t_{k-1,k})' = \zeta + \Lambda_{k-1} (s_{1,k},\ldots,s_{k-1,k})'
\end{align*}
for a $(k-1)$-vector $\zeta$ which is a function of the $w_1,\dots, w_{k-1}$
and $x$ only.

\item \label{Wj2iv} For $1\le k \le d-p$, the quantity $t_{k,k}$ can be written
as
\begin{align*}
t_{k,k} = \left(\kappa_k^2 + s_{k,k}^2\right)^{1/2}.
\end{align*}
In case $k>1$,  we have
$\kappa_k^2 =
\|P_{(I_d-P_{w_1,\dots,w_{k-1}})B}w_k\|^2$, which is a function of $w_1,\ldots,
w_{k-1}$, $t_{1,k},\ldots, t_{k-1,k}$, and $x$. In case k=1, we
have  $\kappa_1^2 = \|x\|^2$.
\end{list}
\end{lemma}

\begin{proof}[Proof of Lemma~\ref{Wj2}]
The first statement in part~\ref{Wj2i} is 
clear, since $B'B=I_p$, $B'N=x\iota'$ and
$\tilde{B}'B=0$, by construction. 
To show that $\tilde{B}'N = (s_{i,j})_{i,j=1}^{d-p}$ 
is upper-triangular, fix $i$ and $j$ so that
$1\le j<i\le d-p$, and note that
$s_{i,j} = \beta_i' w_j$ with
\begin{align*}
\beta_i'w_j \,\,=\,\,
\frac{w_i'(I_d-P_{B,\beta_1,\ldots,\beta_{i-1}})w_j}{\|(I_d-P_{B,\beta_1,
	\ldots,\beta_{i-1}})w_i\|}  
\,\,=\,\, \frac{w_i'(I_d-P_{B,w_1,\ldots,w_{i-1}})w_j}{\|(I_d-P_{B,\beta_1,
	\ldots,\beta_{i-1}})w_i\|} \,\,= \,\,0,
\end{align*}
because
$\s\{B,\beta_1,\ldots,\beta_{i-1}\}=\s\{B,w_1,\ldots,w_{i-1}\}$.
It remains to show that the $s_{i,k}$ with $1\leq i \leq k$
depend only on $w_1,\dots, w_k$ and $x$.
For $i=1$ and $j=1,\dots, d-p$, we have
\begin{align*}
s_{1,j} = \beta_1'w_j = \frac{w_1'(I_d-P_B)w_j}{\|(I_d-P_B)w_1\|} =
\frac{w_1'w_j - w_1'BB'w_j}{\left(w_1'w_1 - w_1'BB'w_1\right)^{1/2}}  =
\frac{w_1'w_j - \|x\|^2}{\left(w_1'w_1 - \|x\|^2\right)^{1/2}}.
\end{align*}
This shows that $s_{1,k}$ has the desired property, and it also shows
that the first $k$ elements of the first row of $\tilde{B}'N$ depend only on
$w_1,\ldots,w_k$ and $x$. Now, for general $i$ satisfying $1\le i \le k$ note
that the $\beta_{1-p},\ldots, \beta_{i-1}$ are mutually orthogonal and thus the
projection onto their span equals the sum of the projections onto every single
one of them. Therefore
\begin{align*}
s_{i,k} &= \beta_i'w_k = 
\frac{w_i'(I_d-P_{B,\beta_1,\ldots,\beta_{i-1}})w_k}
	{\|(I_d-P_{B,\beta_1,\ldots,\beta_{i-1}})w_i\|} 
= \frac{w_i'w_k - \sum_{j=1-p}^{i-1}w_i'\beta_j\beta_j'w_k}
	{\left(w_i'w_i - \sum_{j=1-p}^{i-1}(\beta_j'w_i)^2\right)^{1/2}}  =\\
&= \frac{w_i'w_k - \sum_{j=1-p}^{i-1}s_{j,i}s_{j,k}}
	{\left(w_i'w_i - \sum_{j=1-p}^{i-1}(s_{j,i})^2\right)^{1/2}}. 
\end{align*}
Consequently, we see that $s_{i,k}$ only depends on $w_i$, on $w_k$, 
on $x$, and on elements of $\tilde{B}'N$ in the upper left sub-matrix of the
first $i-1$ rows and $k$ columns. Hence, the claim follows inductively. 

For part~\ref{Wj2ii}, we have $C'N = 0$, because 
$C'N = (c_i' w_j)_{i=1-p}^{0}\mbox{ }_{j=1}^{d-p}$, and because
$\s\{w_1,\ldots,w_{d-p}\} = \s\{c_1,\ldots,c_{d-p}\}=
\s\{c_{1-p},\ldots,c_0\}^\perp$.
And for $\tilde{C}'N = (c_i' w_j)_{i,j=1}^{d-p}$, we have
$c_i' w_j = 0$ whenever $i>j$ by construction of the $c_i$.
To show that $\tilde{B}' \tilde{C} = (\beta_i' c_j)_{i,j=1}^{d-p}$
is upper-triangular, take $i,j$ with $1\leq j < i \leq d-p$,
and note that $c_j \in \s\{w_1,\dots,w_j\}$ while
$\beta_i$ is orthogonal to the span of $B, \beta_1,\dots, \beta_{i-1}$
or, equivalently, to the span of $B, w_1,\dots, w_{i-1}$.
Now consider the $k$-th column of $\mathcal B'\tilde{C}$,
$1\leq k \leq d-p$. For the non-zero elements of this column,
take $i$ so that $1-p \leq i \leq k$, and consider
$\beta_i' c_k$. Because $w_1,\dots, w_{k-1}$
span the same space as $c_1,\dots, c_{k-1}$,  we get
\begin{align*}
\beta_i'c_k \quad&=\quad
\frac{\beta_i'(I_d-
	P_{w_1,\ldots,w_{k-1}})w_k}{\|(I_d-P_{w_1,\ldots,w_{k-1}})w_k\|}
\quad= \quad
 \frac{\beta_i'(I_d-
 	P_{c_1,\ldots,c_{k-1}})w_k}{\|(I_d-P_{c_1,\ldots,c_{k-1}})w_k\|}
\\
&=\quad
\frac{\beta_i'w_k-\sum_{j=1}^{k-1}
	(\beta_i'c_j)(c_j'w_k)}{\left(w_k'w_k-
	\sum_{j=1}^{k-1}(c_j'w_k)^2\right)^{1/2}}.
\end{align*}
In this display, the $c_j$ depend only on $w_1,\dots, w_j$,
so that the terms $c_j' w_k$ depend only on $w_1,\dots, w_k$.
And by part~\ref{Wj2i}, the terms $\beta_i' w_k = s_{i,k}$ depend
only on $w_1,\dots, w_k$ and $x$.
With this, the proof is completed inductively by arguing as in
part~\ref{Wj2i}, mutatis mutandis.

For part \ref{Wj2iii}, we first note that $\tilde{B}'\tilde{C}$ is 
upper-triangular by part~\ref{Wj2ii},
and hence $\Lambda_k$ is lower-triangular, which yields
$\det(\Lambda_k\Lambda_k') = \det(\Lambda_k)^2 = \prod_{i=1}^k
(c_i'\beta_i)^2$. 
For $i\ge 1$, set $N_{(i)}  = (w_1,\ldots,w_i)$ and
$A_{(i)}  = (I_d - P_{N_{(i)}})B$, and set $N_{(0)}=0$. 
Now note that $\s\{B,\beta_1,\dots, \beta_{i-1} \} = \s \{B,N_{(i-1)}\}$, and
calculate $(c_i'\beta_i)^2$ as
\begin{align*}
\frac{\left(
	w_i'(I_d-P_{N_{(i-1)}})(I_d-P_{B,N_{(i-1)}})w_i
	\right)^2}{\|(I_d-P_{N_{(i-1)}})w_i\|^2\|(I_d-P_{B,N_{(i-1)}})w_i\|^2}
\quad=\quad \frac{w_i'(I_d-P_{B,N_{(i-1)}})w_i}{w_i'(I_d-P_{N_{(i-1)}})w_i}. 
\end{align*}
The numerator and the denominator on the right can be written as
\begin{align*}
 \det(w_i'(I_d-P_{B,N_{(i-1)}})w_i) &\quad=\quad
 \frac{\det([B,N_{(i)}]'[B,N_{(i)}])}{\det([B,N_{(i-1)}]'[B,N_{(i-1)}])}
 \text{ and}
\\
 \det(w_i'(I_d-P_{N_{(i-1)}})w_i) &\quad=\quad
 \frac{\det(N_{(i)}'N_{(i)})}{\det(N_{(i-1)}'N_{(i-1)})},
\end{align*}
respectively, upon noting that
$\det( [X,Y]' [X,Y]) = \det(X'X) \det(Y' (I-P_X)Y)$
for appropriate matrices $X$ and $Y$; cf. \citep[Problem 2.4]{Rao73a}).
Using that relation again,
the numerator, i.e., the first expression in the preceding display,
can be further re-expressed as
\begin{align*}
\frac{\det(N_{(i)}'N_{(i)})
 \det(B'(I_d-P_{N_{(i)}})B)}{\det(N_{(i-1)}'N_{(i-1)})\det(B'(I_d-P_{N_{(i-1)}})B)}  
 \,\,=\,\, \frac{\det(N_{(i)}'N_{(i)})
 \det(A_{(i)}'A_{(i)})}{\det(N_{(i-1)}'N_{(i-1)})\det(A_{(i-1)}'A_{(i-1)})} .
\end{align*}
Note that $\det(A_{(i)}'A_{(i)}) = 
1 - \|x\|^2\iota'(N_{(i)}'N_{(i)})^{-1}\iota
> 0$ for $i\geq 1$ by Lemma~\ref{equivalence}, 
and that $\det(A_{(0)}'A_{(0)}) = 1$. 
Taking the pieces together, we see that
$\prod_{i=1}^k(c_i'\beta_i)^2= \det(A_{(k)}' A_{(k)}) = 
1-\|x\|^2\iota'(N_{(k)}'N_{(k)})^{-1}\iota$,
as claimed.
For the last statement, we first note that 
we have $M = \mathcal B S$, because $\mathcal B$ is orthogonal.
For the $k$-th column of $M$, we thus have
$w_k = \sum_{i=1-p}^{d-p} \beta_is_{i,k} = Bx + \sum_{i=1}^k \beta_is_{i,k}$,
where the second equality follows from part~\ref{Wj2i}.
Take now $k>1$,  and consider the vector
$(t_{i,k})_{i=1}^{k-1} = (c_i' w_k)_{i=1}^{k-1}$. Using the last formula for
$w_k$, we see that
\begin{align*}
\begin{pmatrix} t_{1,k} \\ \vdots \\ t_{k-1,k} \end{pmatrix} 
\quad=\quad
\begin{pmatrix} c_1' \\ \vdots \\ c_{k-1}' \end{pmatrix} Bx + \begin{pmatrix}
c_1' \\ \vdots \\ c_{k-1}' \end{pmatrix}  
\begin{pmatrix} \beta_1 & \cdots & \beta_{k-1} \end{pmatrix} \begin{pmatrix}
s_{1,k} \\ \vdots \\ s_{k-1,k} \end{pmatrix} 
\end{align*}
because $c_{i}'\beta_k=0$ for $i=1,\dots, k-1$ in view of part \ref{Wj2ii}. 
We thus see that $(t_{i,k})_{i=1}^{k-1}$ can be decomposed 
as claimed. 
Finally, the statements that $\Lambda_{k}$ 
depends only on $w_1,\dots, w_k$ and $x$, and
that $\zeta$ and $(t_{i,k})_{i=1}^{k-1}$ depend only on 
$w_1,\dots, w_{k-1}$ and $x$ (in case $k>1$),
follow from part~\ref{Wj2i} and~\ref{Wj2ii}.

For part \ref{Wj2iv}, we 
immediately obtain that $t_{1,1}= c_1'w_1 = \|w_1\| = (\|Bx\|^2 +
\|\beta_1s_{1,1}\|^2)^{1/2} = (\|x\|^2 + s_{1,1}^2)^{1/2}$
because $w_1 = B x + (I_d-P_B) w_1$.
In the remaining cases, we have
\begin{align*}
t_{k,k} &\quad=\quad c_k'w_k \quad=\quad \|(I_d-P_{N_{(k-1)}})w_k\| 
	\quad=\quad \|(P_{B,N_{(k)}} -
	P_{N_{(k-1)}})w_k\| \\
&\quad=\quad 
	\|(P_{B,N_{(k-1)}} + P_{(I_d-P_{B,N_{(k-1)}})w_k} -   
	P_{N_{(k-1)}})w_k\| 
\\
&\quad=\quad
	\|(P_{(I_d-P_{N_{(k-1)}})B}  + P_{(I_d-P_{B,N_{(k-1)}})w_k})w_k\|  
\\
&\quad=\quad
	\left(\|P_{(I_d-P_{N_{(k-1)}})B}w_k\|^2  + 
	\|P_{\beta_k}w_k\|^2\right)^{1/2} 
	\quad=\quad 
	\left(\kappa_k^2  + s_{k,k}^2\right)^{1/2},
\end{align*}
where the last two equalities are obtained by noting that
$(I_d-P_{N_{(k-1)}})B \perp \beta_k$ and that $\|P_{\beta_k}w_k\|^2 = w_k'
\beta_k \beta_k' w_k = s_{k,k}^2$. To see that $\kappa_k^2$ is a function of
$w_1,\ldots,w_{k-1},t_{1,k},\ldots,t_{k-1,k}$ and $x$ only, set $\eta_k =
\iota'(N_{(k)}'N_{(k)})^{-1}\iota$, set $\tilde{C}_{(k)}$ equal to the matrix
consisting of the first $k$ columns of $\tilde{C}$, and observe that
\begin{align*}
\kappa_k^2 &\quad=\quad w_k'(I_d-P_{N_{(k-1)}})B\left[ B'(I_d -
P_{N_{(k-1)}})B\right]^{-1} B'(I_d - P_{N_{(k-1)}})w_k \\
&\quad=\quad w_k'(I_d- \tilde{C}_{(k-1)}\tilde{C}_{(k-1)}')B\left[ I_p -
xx'\eta_{(k-1)}\right]^{-1} B'(I_d - \tilde{C}_{(k-1)}\tilde{C}_{(k-1)}')w_k,
\end{align*}
because $B' w_i = x$, and because $N_{(k-1)}$ and $\tilde{C}_{(k-1)}$
span the same space.
We further analyze the expressions to the left and right of the inverse, i.e.,
\begin{align*}
B'(I_d - \tilde{C}_{(k-1)}\tilde{C}_{(k-1)}')w_k 
&\quad=\quad x - B'\tilde{C}_{(k-1)}\tilde{C}_{(k-1)}'w_k \\
&\quad=\quad x - B'\tilde{C}_{(k-1)}(t_{1,k},\ldots,t_{k-1,k})'.
\end{align*}
Now, in view of part \ref{Wj2ii}, $B'\tilde{C}_{(k-1)}$ depends only on
$w_1,\ldots,w_{k-1}$ and $x$. Together with part~\ref{Wj2iii}, we see that
$\kappa_k^2$ does indeed depend only on the claimed quantities.
\end{proof}


We now replace the deterministic matrices $B = [\beta_{1-p},\ldots,\beta_{0}]$
and the vectors $w_1,\dots, w_{d-p}$ by their random counterparts 
$\mathbf B = [b_{1-p},\dots, b_0]$ and $W_1,\dots, W_{d-p}$,
where $\mathbf B$  and the $W_j$ are as in \eqref{Wj}
with $d-p$ replacing $k$ (and where $x\in \R^p$ is fixed throughout).
Because the event where the column-vectors of $\mathbf B$ and the
vectors $W_1,\dots, W_{d-p}$ are linearly independent
has probability one, we can assume, without loss of generality,
that these vectors are linearly independent.
In the following, re-define the matrices
$N$, $\tilde{B}$, $C$, and $\tilde{C}$
as in the discussion leading up to Lemma~\ref{Wj2} but with
$\mathbf B$ and the $W_j$ replacing $B$ and the $w_j$, respectively.

\begin{lemma}\label{Wj3}
Let $\mathbf B$, $W_1,\dots, W_{d-p}$, and $x$ as in the preceding paragraph.
For $\mathcal B$, $\mathcal C$, $S$, and $T$, which 
are functions of $\mathbf B$, $W_1,\dots, W_{d-p}$, and $x$, the following 
holds true.
\begin{list}{\thercnt}{
        \usecounter{rcnt}
        \setlength\itemindent{5pt}
        \setlength\leftmargin{0pt}
        \setlength\partopsep{0pt} 
        }
\item \label{Wj3ii}
The matrices $\mathcal B$ and $S$ are independent with $\mathcal B$ 
uniformly distributed on $\mathcal V_{d,d}$.
Furthermore, the random 
variables $s_{i,j}$, $1-p\le i \le d-p$, $1 \le j \le d-p$, are independent, 
and distributed as follows:
For $j$ satisfying $1\le j \le d-p$, we have 
$(s_{1-p,j},\dots, s_{d-p,j})' = (x_1,\dots,x_p,s_{1,j},\dots, s_{j,j}, 0,\dots,0)'$
with $s_{i,j} \sim N(0,1)$
for $i=1,\dots, j-1$ and $s_{j,j} \sim (\chi^2_{d-p-j+1})^{1/2}$.

\item \label{Wj3iii}
The random matrices $\mathcal C$ and $T$ are independent
with $\mathcal C$ uniformly distributed on $\mathcal V_{d,d}$.
\end{list}
\end{lemma}

\begin{proof}
Note that both $S=S(M)$ and $\mathcal B=\mathcal B(M)$ are functions of $M$,
and that the law of $M$ is left invariant under transformations from 
$\mathcal V_{d,d}$, in the sense that $A M$ is distributed as $M$
for any $A \in \mathcal V_{d,d}$. [Indeed,
$M=M(\mathbf B,V_1,\dots,V_{d-p})$ is a function of 
orthogonally invariant random vectors and satisfies $AM =
M(A\mathbf B,A V_1,\dots,A V_{d-p})$ 
for any $A\in \mathcal V_{d,d}$.]
Take $\mathbf A$ uniformly distributed on $\mathcal V_{d,d}$
and independent of $M$.
Then the law of $(\mathcal B(M), S(M))$ coincides with the
law of $(\mathcal B(\mathbf A M), S(\mathbf A M))$, and it is  easy to see that
$(\mathcal B(\mathbf A M), S(\mathbf A M)) = (\mathbf A \mathcal B(M), S(M))$. 
Hence,
$$
\left(\begin{array}{c}
\mathcal B(M)\\ S(M)
\end{array}\right)
\quad\sim\quad
\left(\begin{array}{c}
\mathbf A \mathcal B(M)\\ S(M)
\end{array}\right).
$$
On the right-hand side of this display, $\mathbf A \mathcal B(M)$ is 
uniformly distributed on $\mathcal V_{d,d}$
conditional on $M$, and thus also unconditionally. This also entails
that 
$\mathbf A \mathcal B(M)$ is 
independent of $M$ and, hence, of $S(M)$. From the above
display 
we therefore obtain that $\mathcal B$
is uniformly distributed on $\mathcal V_{d,d}$ and it is independent of $S$.
A similar argument shows that 
$\mathcal C$ also is uniformly distributed on $\mathcal V_{d,d}$ and is
independent of $T$.

It remains to consider the joint distribution of the $s_{i,j}$, with $1-p
\leq i \le d-p$ and $1 \le j \le d-p$.
Some $s_{i,j}$ are degenerate, as described by Lemma~\ref{Wj2}\ref{Wj2i}.
For the non-degenerate $s_{i,j}$, i.e., for those 
where we have $1\leq i \leq j \le  d-p$, recall that 
the vectors $b_{1-p},\dots,b_0,b_1,\dots, b_{d-p}$ are mutually orthogonal, 
so that
$$
s_{i,j} \quad = \quad b_i' W_j \quad=\quad
b_i' (\mathbf Bx  + (I_d - \mathbf B \mathbf B') V_j) \quad=\quad b_i' V_j
$$
for $i=1,\dots, j-1$ and
$$
s_{j,j} \quad = \quad b_j' W_j \quad=\quad 
\left( V_j' (I_d-P_{b_{1-p},\dots,b_0,\dots, b_{j-1}})V_j \right)^{1/2}
$$
for each fixed $j$ with $1\leq j \le d-p$.
Note that $b_{1-p},\dots, b_{j-1}$ are functions of
$B, W_1,\dots, W_{j-1}$. Hence, conditional on
$B, W_1,\dots, W_{j-1}$, we see that the random variables
$s_{1,j}, \dots, s_{j,j}$ are independent and distributed as claimed
(conditionally).
Because their conditional distribution does not depend on the conditioning
variables, they are independent of the conditioning random variables, and
their conditional distribution coincides with 
their unconditional distribution.
In other words, $s_{1,j}, \dots, s_{j,j}$ are jointly distributed as
stated in the lemma; moreover, they are independent of
$B, W_1,\dots, W_{j-1}$ and hence also independent of
the $s_{i,k}$ with $k < j$
(in view of Lemma~\ref{Wj2}\ref{Wj2i}).  As this holds 
for each $j=1,\dots, d-p$, it is easy to see that the
$s_{i,j}$ are all independent. [Use the independence relations
established so far to argue that the, say, joint characteristic
function of all the $s_{i,j}$ can be factorized into the product
of the joint characteristic function of the $s_{i,k}$ with $k<d-p$
and the marginal characteristic functions of the $s_{i,d-p}$.
Now repeat this argument with $d-p$ replaced by
$d-p-1, d-p-2, \dots, 2$.]  
\end{proof}

\begin{corollary}\label{Wj4}
Fix $k$ satisfying $1\leq k \le d-p$,
and let $\zeta$, $\Lambda_k$, and $\kappa_k^2$ be as in Lemma~\ref{Wj2} but
with $W_j$ replacing $w_j$.
\begin{list}{\thercnt}{
	\usecounter{rcnt}
	\setlength\itemindent{5pt}
	\setlength\leftmargin{0pt}
	\setlength\itemsep{\medskipamount}
}
\item \label{Wj4i}
In case $k>1$ and conditional on $W_1,\dots, W_{k-1}$,
the vector $(t_{1,k},\dots, t_{k-1,k})'$  is distributed
as $N(\zeta, \Lambda_{k-1} \Lambda_{k-1}')$.

\item \label{Wj4ii}
In case $k>1$ and
conditional on $W_1,\dots, W_{k-1}, t_{1,k},\dots, t_{k-1,k}$,
the quantities $c_k$ and $t_{k,k}$ are independent and  distributed as follows:
We have
$t_{k,k} \sim (\kappa_k^2 + \chi^2_{d-p-k+1})^{1/2}$ (conditionally), i.e., 
$t_{k,k}$ is distributed
as the square root of the sum of $\kappa_k^2$ and a chi-square distributed
random
variable with $d-p-k+1$ degrees of freedom (conditionally).
Moreover, the vector $c_k$ is distributed
as $\U({\mathcal V}_{d,1} \cap [W_1,\dots, W_{k-1}]^\perp)$
(conditionally), where $\U(\dots)$ denotes the uniform distribution on
the  indicated set, and where  $\mathcal V_{d,1}$ denotes the 
unit sphere in $\R^d$ by our conventions.
In case $k=1$, these statements apply unconditionally, mutatis mutandis.
\end{list}
\end{corollary}

\begin{proof}
For part~\ref{Wj4i},
it suffices to show that the vector $v =  (s_{1,k},\dots, s_{k-1,k})'$ is 
distributed as $N(0,I_{k-1})$ and is
independent of $W_1,\dots, W_{k-1}$, in view of Lemma~\ref{Wj2}\ref{Wj2iii}.
From Lemma~\ref{Wj3}\ref{Wj3ii}, we see that $v$
does indeed have the $N(0,I_{k-1})$ distribution, unconditionally, 
and Lemma~\ref{Wj3}\ref{Wj3ii} also shows that $v$
is independent of $\mathcal B$ and of the $s_{i,j}$ with $1\leq j < k$.
Noting that $W_1,\dots, W_{k-1}$ are functions of
$\mathcal B$ and of the $s_{i,j}$ with $1\leq j<k$ completes the argument.

For part~\ref{Wj4ii}, consider first the case where $k=1$.
Here, we have $c_1 \sim \U(\mathcal V_{d,1})$ independently of $t_{1,1}$
in view of Lemma~\ref{Wj3}\ref{Wj3iii}. 
And for the distribution of $t_{1,1}$, we note that
$t_{1,1} = (\|x\|^2 + s_{1,1}^2)^{1/2}$ by Lemma~\ref{Wj2}\ref{Wj2iv},
and that $s_{1,1}^2 \sim \chi^2_{d-p}$ by Lemma~\ref{Wj3}\ref{Wj3ii}.

For part~\ref{Wj4ii} with $k>1$, write $L_k$ as shorthand for the list
of conditioning variables under consideration, 
i.e., for $W_1,\dots, W_{k-1}, t_{1,k},\dots, t_{k-1,k}$, and write
$\tilde{L}_k$ for the list of random variables consisting of 
$c_1,\dots, c_{k-1}$, of the $t_{i,j}$ with $1\leq i \leq j < k$,
and of the $t_{i,k}$ with $1 \leq i < k$.
Note that conditioning on $L_k$ is equivalent to conditioning on 
$\tilde{L}_k$
(because there is a measurable bijection
between  the two groups of conditioning variables
in view of Lemma~\ref{Wj2}).
We now proceed similarly as in the case $k=1$, and note
first that $c_k$ and $t_{k,k}$ are conditionally independent given
$L_k$. [Indeed, $\mathcal C$ and $T$ are independent by
Lemma~\ref{Wj3}\ref{Wj3iii},
such that both 
the conditional distribution of (a) $c_k$ given $\tilde{L}_k$ and $t_{k,k}$,
and also the conditional distribution of (b) $c_k$ given $\tilde{L}_k$,
coincides with 
the conditional distribution of (c) $c_k$ given  $c_1,\dots, c_{k-1}$.
Because the conditional distributions (a) and (b) coincide,
it follows that $c_k$ and $t_{k,k}$ are conditionally independent given
$\tilde{L}_k$ or, equivalently, given $L_k$.\footnote{
	See Sections 6.12 and 6.15 in~\citep{Hof94a} 
	for some basic facts about conditional independence.
}]
Next, we note that $c_k$ given $L_k$ is distributed
as claimed. [It suffices to note that the conditional distribution
(b) agrees with the conditional distribution (c) mentioned earlier,
and that the latter is $\U(\mathcal V_{d,1} \cap [c_1,\dots, c_{k-1}]^\perp)$
or,
equivalently, $\U(\mathcal V_{d,1} \cap [W_1,\dots, W_{k-1}]^\perp)$.]
Finally, to obtain the conditional distribution of $t_{k,k}$ given $L_k$,
it remains to show that $s_{k,k}^2$ is independent of $L_k$
(because $t_{k,k} = (\kappa_k^2 + s_{k,k}^2)^{1/2}$ with $\kappa_k^2$
being a function of $L_k$ in view of Lemma~\ref{Wj2}\ref{Wj2iv},
and because $s_{k,k}^2 \sim \chi^2_{d-p-k+1}$ (unconditionally) by
Lemma~\ref{Wj3}\ref{Wj3ii}).
To this end, we again use Lemma~\ref{Wj3}\ref{Wj3ii} to see that
$s_{k,k}$ is independent of $\mathcal B$ and of the $s_{i,j}$ with 
$(i,j) \neq (k,k)$, and we note that
the random variables in $L_k$ are functions of 
$\mathcal B$ and of the $s_{i,j}$ with $(i,j)\neq (k,k)$.
[Indeed, the vectors $W_1,\dots, W_{k-1}$ are functions of $\mathcal B$ and
of the $s_{i,j}$ with $j < k$. Moreover, by 
Lemma~\ref{Wj2}\ref{Wj2iii}, 
the vector $(t_{1,k}, \dots, t_{k-1,k})'$ is a function
of the $W_1,\dots, W_{k-1}$ and of $(s_{1,k},\dots, s_{k-1,k})'$.]
\end{proof}

\begin{lemma}\label{Uk}
For positive integers $2\le k \le d$, let $q_1,\dots,q_k$ be independent random
variables
with $q_i \thicksim N(0,1)$ for $i=1,\dots, k-1$ and $q_k
\thicksim (\chi_{d-k+1}^2)^{1/2}$,
and let $c_1,\dots, c_{k-1}$ be fixed 
orthonormal $d$-vectors, while $c_k \thicksim
\mathcal U(\mathcal V_{d,1}\cap [c_1,\dots,c_{k-1}]^\perp)$ is independent of
all the $q_i$ with $i=1,\dots, k$. Finally, set
$U = \sum_{i=1}^{k-1} q_i c_i  +  q_k c_k$.
Then $U\thicksim N(0,I_d)$.
\end{lemma}

\begin{proof}
Let $C = [c_1,\dots,c_{k-1}]$,
and let $D$ denote a $d\times (d-k+1)$ matrix 
so that $M = [C,D]$ is orthogonal. 
Set $V= (q_1,\dots,q_{k-1})'$, and take $W \thicksim N(0,I_{d-k+1})$
independent of $V$, $q_k$ and $c_k$. Now observe that $\P(DW=0)=0$ and thus 
\begin{align*}
DW \quad=\quad \|DW\| \frac{DW}{\|DW\|} \quad=\quad \|W\| D \frac{W}{\|W\|}
\end{align*}
almost surely.
Note that 
$\|W\|$ and $W/ \|W\|$
are independent with $\|W\| \thicksim (\chi_{d-k+1}^2)^{1/2}$
and $W/\|W\| \thicksim \mathcal U(\mathcal V_{d,1})$,
because the law of $W$ is spherically symmetric 
(cf. also \citep[Theorem 1]{Cam81a}). 
Hence, $D\frac{W}{\|W\|}$ and $\|W\|$ are
independent and distributed as $c_k$ and $q_k$, respectively. Therefore we have
\begin{align*}
U \quad \thicksim \quad CV \;+\; DW \quad = \quad  M [V',W']'  \quad \thicksim
\quad N(0,I_d),
\end{align*}
due to rotation invariance of the standard normal distribution.
\end{proof}

\begin{proof}[\bf Proof of Proposition~\ref{p2}]
For later use, let $q_{i,k}$, $1\le i \le k \le d-p$, be independent
random variables such that $q_{k,k} \thicksim \left(\chi_{d-k+1}^2\right)^{1/2}$
and $q_{i,k} \thicksim N(0,1)$ for $i<k$, and such that the 
$q_{i,j}$ are jointly independent of all the other quantities considered so
far. 

Because the $W_j$ for $1\leq j \leq k$ are defined by \eqref{Wj}  and
are such that the columns of $\mathbf B = [b_{1-p},\dots,b_{0}]$ and
$W_1,\dots,W_k$ are linearly independent
with probability 1,
we can set $\varphi_x(w_1,\dots, w_k) = 0$ unless
the $w_j$, $1\leq j \leq k$, are linearly independent and 
such that \eqref{wj} is satisfied for some non-random $d\times p$-matrix
$B=(\beta_{1-p},\dots,\beta_0)$
that satisfies $B'B=I_p$ and
$\text{rank}([\beta_{1-p},\dots,\beta_0,w_1,\dots,w_k])=p+k$. Equivalently, we
may set
$\varphi_x(w_1,\dots, w_k) = 0$ whenever either 
$S_k=(w_i'w_j/d)_{i,j=1}^k$ is not invertible
or $\|x\|^2 \iota' S_k^{-1} \iota \geq  d$ (cf. Lemma~\ref{Wj1}). Assume, from
now on,
that $S_k$ is invertible and that $\|x\|^2 \iota' S_k^{-1} \iota < d$.

We derive the density ratio in 
Proposition~\ref{p2} 
by induction on $k$. First, for the case where $k=1$, we see that
$W_1 = c_1 t_{1,1}$, since $M=\mathcal CT$ and in view of the structure of $T$
described by Lemma~\ref{Wj2}\ref{Wj2ii}. Moreover, 
$c_1$ and $t_{1,1}$ are independent with 
$c_1 \thicksim \mathcal U(\mathcal V_{d,1})$ and 
$t_{1,1} \sim (\|x\|^2 + \chi^2_{d-p})^{1/2}$,
by Corollary~\ref{Wj4}\ref{Wj4ii}.
Now we use Lemma C.5 in \citep{Lee12b}  with
$k =1$, and with $c_1$, $q_{1,1}$ and $t_{1,1}$ replacing
the quantities $v_1$, $q_1$ and $t_1$, respectively, in that reference.
Noting that $q_{1,1} c_1 \sim N(0, I_{d})$ in view of Lemma~\ref{Uk},
we obtain that $\varphi_x(w_1)/\phi(w_1)$ is the ratio of the density of the
$(\|x\|^2 + \chi^2_{d-p})^{1/2}$-distribution and the density of the
$(\chi^2_d)^{1/2}$-distribution, both
evaluated at $||w_1||$. After simplification, we therefore have
\begin{align*}
\frac{\varphi_x(w_1)}{\phi(w_1)}& \quad=\quad
2^{p/2} \frac{\Gamma(d/2)}{\Gamma((d-p)/2)} \left(||w_1||^2\right)^{-p/2} 
(1-\|x\|^2/||w_1||^2)^{\frac{d-p-2}{2}}  e^{\frac{\|x\|^2}{2}}
\\ & \quad=\quad
\frac{(d/2)^{-p/2}\Gamma(d/2)}{\Gamma((d-p)/2)} \left(\det S_1\right)^{-p/2}
\left( 1-\frac{\|x\|^2}{d} \iota' S_1^{-1} \iota
\right)^{\frac{d-p-2}{2}} e^{\frac{\|x\|^2}{2}},
\end{align*}
as claimed.

Suppose now that $k>1$ and that the proposition holds 
with $k-1$ replacing $k$,
and consider the conditional distribution of
$W_k$ given $W_1=w_1,\dots, W_{k-1}=w_{k-1}$. 
Here, $W_k$ can be written as 
$W_k = \sum_{i=1}^{k-1} t_{i,k} c_i + t_{k,k} c_k$,
because $M = \mathcal C T$ and in view of
Lemma~\ref{Wj2}\ref{Wj2ii}.
Set $ U_k = \sum_{i=1}^{k-1} q_{i,k} c_i + q_{k,k} c_k$.
Conditional on $W_1=w_1,\dots, W_{k-1} = w_{k-1}$,
first
note that the orthonormal vectors $c_1,\dots, c_{k-1}$ are 
functions of $W_1,\dots, W_{k-1}$ by construction, while
$c_k \sim \mathcal U(\mathcal V_{d,1} \cap[c_1,\dots, c_{k-1}]^\perp)$
(in view of Corollary~\ref{Wj4}\ref{Wj4ii} and because $W_1,\dots, W_{k-1}$ and
$c_1,\dots, c_{k-1}$ span the same space).
Second,  note that $(q_{1,k},\dots, q_{k,k})$ 
and $(t_{1,k}, \dots, t_{k,k})$ are both conditionally 
independent of $c_k$ given $W_1=w_1,\dots, W_{k-1}=w_{k-1}$.
[This is obvious for the $q_{i,k}$, and this follows from
Corollary~\ref{Wj4}\ref{Wj4ii} for the $t_{i,k}$.]
And third, note that the conditional distribution of
$(t_{1,k},\dots, t_{k,k})$ given $W_1=w_1,\dots, W_{k-1}=w_{k-1}$
is described by Corollary~\ref{Wj4}.
In view of these three observations, it is now elementary to verify that
Lemma~C.5 in \citep{Lee12b} applies, with $k$ as chosen here, and
with
$c_1,\dots, c_{k-1}, c_k$,
$q_{1,k},\dots, q_{k,k}$, and
$t_{1,k},\dots, t_{k,k}$
replacing
$\upsilon_1,\dots,\upsilon_{k-1}, v_k$,
$q_1,\dots, q_k$, and
$t_1,\dots, t_k$, respectively.
In particular, we see that the conditional law of 
$W_k$ given $W_1=w_1,\dots, W_{k-1}=w_{k-1}$ has a density relative to the 
conditional law of $U_k$. 
Note that $U_k\sim N(0,I_d)$  
conditional on $c_1,\dots, c_{k-1}$ or, equivalently, 
conditional on $W_1,\dots, W_{k-1}$ by Lemma~\ref{Uk}.
Therefore, the conditional law of 
$W_k$ given $W_1=w_1,\dots, W_{k-1}=w_{k-1}$ 
has a Lebesgue density which we denote by
$\varphi_x(w_k|w_1,\dots, w_{k-1})$.
This entails that a Lebesgue density $\varphi_x(w_1,\dots, w_k)$ of
$W_1,\dots, W_k$ is well-defined by the relation
\begin{equation}\label{tmp2}
\frac{\varphi_x(w_1,\dots,w_k)}{ \phi(w_1)\cdots\phi(w_k) }
\quad=\quad
\frac{\varphi_x(w_1,\dots,w_{k-1})}{ \phi(w_1)\cdots\phi(w_{k-1}) }
\,
\frac{\varphi_x(w_k|w_1,\dots, w_{k-1})}{\phi(w_k)}.
\end{equation}
To complete the proof, it remains to explicitly compute the expression on the
right-hand side of \eqref{tmp2} and to show that 
it coincides with the right-hand side of the formula given in the proposition.
Recall that we already have a formula for
the first fraction on the right-hand side (by the induction assumption).

To compute the second fraction on the right side of \eqref{tmp2} we first use 
Lemma C.5 in \citep{Lee12b} to see that 
\begin{equation} \label{tmp3}
\frac{\varphi_x(w_k|w_1,\dots, w_{k-1})}{\phi(w_k)}
\quad=\quad
\frac{f_t}{f_q}(c_1' w_k,\dots, c_{k-1}' w_k, 
||(I_d-P_{c_1,\dots, c_{k-1}})w_k||),
\end{equation}
where $f_{t}$ and $f_q$ denote the conditional Lebesgue densities of
$(t_{1,k},\dots, t_{k,k})$ and $(q_{1,k}, \dots, q_{k,k})$, respectively,
given $W_1=w_1,\dots, W_{k-1}=w_{k-1}$.
Now, by definition of $q_{1,k},\dots,q_{k,k}$, we have
\begin{align*}
f_q(&c_1' w_k, \dots, c_{k-1}' w_k, 
||(I_d-P_{c_1,\dots, c_{k-1}})w_k||) 
\\
&\quad=\quad
\left(
\prod_{i=1}^{k-1} f_{N(0,1)}(c_i' w_k) \right)\; f_{(\chi^2_{d-k+1})^{1/2}} (
||(I_d-P_{c_1,\dots, c_{k-1}})w_k||) 
\\
&\quad=\quad
\frac{
	\left|\left|
	(I_d - P_{w_1,\dots, w_{k-1}})w_k
	\right|\right|^{d-k}
}{
\pi^{(k-1)/2} 2^{(d-2)/2} \Gamma((d-k+1)/2)
}
e^{-\frac{1}{2} ||w_k||^2},
\end{align*}
where, in the first equality, we use the symbols $f_{N(0,1)}$ and
$f_{(\chi^2_{d-k+1})^{1/2}}$ to denote
the Lebesgue densities of the distributions indicated in the subscript,
and where the second equality follows from elementary simplifications
and by noting that
$\sum_{i=1}^{k-1}(c_i' w_k)^2 = w_k' P_{c_1,\dots,c_{k-1}} w_k$.
Moreover, we have
\begin{align*}
f_t(&c_1' w_k, \dots, c_{k-1}' w_k, 
||(I_d-P_{c_1,\dots, c_{k-1}})w_k||) 
\\&\quad=\quad
f_{N(\zeta, \Lambda_{k-1} \Lambda_{k-1}')}( c_1' w_k,\dots, c_{k-1}' w_k)
\;
f_{(\kappa_k^2+\chi^2_{d-p-k+1})^{1/2}}( ||(I_d-P_{c_1,\dots, c_{k-1}})w_k||)
\\&\quad=\quad
\frac{2^{-(d-p-2)/2} \pi^{-(k-1)/2}}{\Gamma((d-p-k+1)/2)}
\left(
\det \Lambda_{k-1}\Lambda_{k-1}'
\right)^{-1/2}
\\ &\quad\qquad\times\;\;
||(I_d - P_{c_1,\dots, c_{k-1}})w_k||
\left(
||(I_d - P_{c_1,\dots, c_{k-1}})w_k||^2 - \kappa_k^2
\right)^{(d-p-k-1)/2}
\\ &\quad\qquad\times\;\;
\exp\left[-\frac{1}{2} 
  ||(I_d - P_{c_1,\dots, c_{k-1}})w_k||^2 
  +\frac{1}{2}\kappa_k^2
  -\frac{1}{2} || (s_{1,k},\dots, s_{k-1,k})'||^2\right],
\end{align*}
where the first equality follows from
Corollary~\ref{Wj4} upon noting that
$\Lambda_{k-1}$ is invertible. [Indeed, note that
$\det \Lambda_{k-1} \Lambda_{k-1}'
= 1-\|x\|^2 \iota'[ (w_i'w_j)_{i,j=1}^{k-1}]^{-1}\iota$
by Lemma~\ref{Wj2}\ref{Wj2iii}, and 
that $\|x\|^2 \iota'[ (w_i'w_j)_{i,j=1}^{k-1}]^{-1}\iota
< 1$
by Lemma~\ref{Wj1} and in view of our assumptions on the $w_i$.]
The second equality is obtained by elementary simplifications,
upon noting that $(c_1'w_k,\dots, c_{k-1}'w_k) = 
(t_{1,k},\dots, t_{k-1,k})$ by definition, 
and upon using Lemma~\ref{Wj2}\ref{Wj2iii}.

To further simplify the formula in the preceding display, recall the $d\times
p$-matrix $B$ introduced at the beginning of the proof, 
and note that $\kappa_k^2 = ||P_{(I_d - P_{w_1,\dots, w_{k-1}})B}w_k||^2$
by Lemma~\ref{Wj2}\ref{Wj2iv}. 
Also, note that
$\|(s_{1,k},\dots, s_{k-1,k})'\| 
=\|(\beta_1'w_k,\dots, \beta_{k-1}'w_k)'\|
=\|P_{\beta_1,\dots,\beta_{k-1}} w_k\|$,
and that
$\s\{B, c_1,\dots, c_{k-1}\} = \s\{B, w_1,\dots, w_{k-1}\}
= \s\{B, \beta_1,\dots, \beta_{k-1}\}$,
by our conventions (cf. the discussion leading up to Lemma~\ref{Wj2}).
With this, the exponent on the far right-hand side of the
preceding display can be written as $1/2$ times
\begin{align*}
&-||(I_d-P_{c_1,\dots, c_{k-1}}) w_k||^2
+||P_{(I_d-P_{w_1,\dots, w_{k-1}})B}  w_k||^2
- ||P_{\beta_1,\dots,\beta_{k-1}} w_k||^2
\\
&=\quad 
- ||w_k||^2 + ||P_{c_1,\dots, c_{k-1}} w_k||^2
+||P_{(I_d-P_{c_1,\dots, c_{k-1}})B}  w_k||^2
- ||P_{\beta_1,\dots,\beta_{k-1}} w_k||^2
\\
&=\quad 
- ||w_k||^2 +||P_{B,c_1,\dots, c_{k-1}} w_k||^2 
- ||P_{\beta_1,\dots,\beta_{k-1}} w_k||^2
\\
&=\quad 
- ||w_k||^2 +||P_{B,\beta_1,\dots, \beta_{k-1}} w_k||^2 
- ||P_{\beta_1,\dots,\beta_{k-1}} w_k||^2
\\
&=\quad - ||w_k||^2 + ||P_{B} w_k||^2 \quad=\quad
- ||w_k||^2 +\|x\|^2,
\end{align*}
where the last equality holds in view of \eqref{wj}.
Plugging the formulae for $f_q(\dots)$ and for $f_t(\dots)$ obtained
so far into \eqref{tmp3}, again using the formula for $\kappa_k^2$,
and simplifying, we can re-write the right-hand side of \eqref{tmp3} as
\begin{align*}
&\frac{2^{p/2} \Gamma((d-k+1)/2)}{\Gamma((d-p-k+1)/2)}
\left( \det \Lambda_{k-1}\Lambda_{k-1}'\right)^{-1/2}
\left(||(I_d - P_{w_1,\dots, w_{k-1}})w_k||^2\right)^{-p/2}
\\&\times \;\;
\left(
1 - \frac{
||P_{(I_d - P_{w_1,\dots, w_{k-1}})B} w_k||^2
}{
||(I_d - P_{w_1,\dots, w_{k-1}})w_k||^2
}
\right)^{(d-p-k-1)/2}
e^{\|x\|^2/2}.
\end{align*}
Using Lemma~\ref{Wj2}\ref{Wj2iii}, 
we see that $\det{\Lambda_{k-1}\Lambda_{k-1}'} =
\prod_{i=1}^{k-1} (c_i'\beta_i)^2$, while
the expression under the $(d-p-k-1)/2$ exponent can be written as
\begin{align*}
&\frac{\|(I_d-P_{w_1,\dots,w_{k-1}})w_k\|^2 -
\|P_{(I_d-P_{w_1,\dots,w_{k-1}})B}w_k\|^2}
	{\|(I_d-P_{w_1,\dots,w_{k-1}})w_k\|^2} \\
&=  \quad
\frac{w_k'w_k - 
		\left(
			\|P_{w_1,\dots,w_{k-1}}w_k\|^2 +
			\|P_{(I_d-P_{w_1,\dots,w_{k-1}})B}w_k\|^2
		\right)}
	{\|(I_d-P_{w_1,\dots,w_{k-1}})w_k\|^2} \\
&=  \quad
\frac{w_k'w_k-\|P_{B,w_1,\dots,w_{k-1}}w_k\|^2}
	{\|(I_d-P_{w_1,\dots,w_{k-1}})w_k\|^2}
\,\,
\frac{\|(I_d- P_{B,w_1,\dots,w_{k-1}})w_k\|^2}
	{\|(I_d- P_{B,w_1,\dots,w_{k-1}})w_k\|^2} \\
&=  \quad
\frac{\left(w_k'(I_d-P_{w_1,\dots,w_{k-1}})(I_d -
P_{B,w_1,\dots,w_{k-1}})w_k\right)^2}
	{\|(I_d-P_{w_1,\dots,w_{k-1}})w_k\|^2\|(I_d-
	P_{B,w_1,\dots,w_{k-1}})w_k\|^2}
	\quad=\quad (c_k'\beta_k)^2,
\end{align*}
where the second-to-last equality holds because of $P_{w_1,\dots,w_{k-1}}(I_d -
P_{B,w_1,\dots,w_{k-1}}) = 0$, and where the last equality
follows from the definition of $c_k$ and $\beta_k$.
Moreover, we also see that
$$
||(I_d - P_{w_1,\dots, w_{k-1}})w_k||^2 
\quad=\quad \det w_k'(I_d - P_{w_1,\dots,w_{k-1})}w_k 
\quad=\quad d\frac{\det S_k}{\det S_{k-1}}
$$
by again using the relation
$\det( [X,Y]' [X,Y]) = \det(X'X) \det(Y' (I-P_X)Y)$
for appropriate matrices $X$ and $Y$.
So putting the pieces together, we obtain
that the conditional density
$\varphi_x(w_k|w_1,\dots,w_{k-1})/\phi(w_k)$ can be written as
\begin{align*}
\left(\frac{d}{2}\right)^{-\frac{p}{2}}
\frac{
\Gamma((d-k+1)/2)}{\Gamma((d-p-k+1)/2)}
\left( 
\prod_{i=1}^{k-1}c_i'\beta_i
\right)^{-1}
\left( 
\frac{\det S_k}{\det S_{k-1}}
\right)^{-\frac{p}{2}}
\left(
c_k'\beta_k
\right)^{d-p-k-1}
e^{\frac{\|x\|^2}{2}}.
\end{align*}
We now take the formula for 
$\varphi_x(w_1,\dots, w_{k-1}) / (\phi(w_1)\cdots \phi(w_{k-1}))$
given by Proposition~\ref{p2} (with $k-1$ replacing $k$). We can 
write $\varphi_x(w_1,\dots, w_{k-1}) / (\phi(w_1)\cdots \phi(w_{k-1}))$
as
\begin{align*}
\left(\frac{d}{2} \right)^{-\frac{(k-1)p}{2}}
\prod_{i=1}^{k-1} 
\frac{ \Gamma \left( (d-i+1)/2\right)}
	{\Gamma \left((d-p-i+1)/2\right)}
\det(S_{k-1})^{-\frac{p}{2}}
\left(\prod_{i=1}^{k-1} c_i'\beta_i \right)^{d-p-k}
e^{\frac{(k-1)}2 \|x\|^2}
\end{align*}
in view of  Lemma~\ref{Wj2}\ref{Wj2iii}.
Multiplying the expressions in the preceding two displays, we see that
the density 
$\varphi_x(w_1,\dots, w_{k}) / (\phi(w_1)\cdots \phi(w_{k}))$ is given by
\begin{align*}
\left(\frac{d}{2}\right)^{-\frac{kp}{2}}
\prod_{i=1}^{k} 
\frac{ \Gamma \left( (d-i+1)/2\right)}
	{\Gamma \left((d-p-i+1)/2\right)}
\det(S_{k})^{-\frac{p}{2}}
\left(\prod_{i=1}^k c_i'\beta_i \right)^{d-p-k-1}
e^{\frac{k}2 \|x\|^2},
\end{align*}
which, after another application of Lemma~\ref{Wj2}\ref{Wj2iii}, equals the 
density ratio in Proposition~\ref{p2}, as desired.
The property of the normalizing constant is now 
established by Lemma~\ref{normConst}.
\end{proof}


\section{Proofs for Section~\ref{s3.5}}
\label{appendixD}

\begin{lemma}\label{Taylor2} 
Fix $M>1$ and positive integers $k$ and $p$. For $d\in \N$ such that
$d>4(k+p+1)M^4$, fix $x\in \R^p$ such that $\|x\|\le M$,
and define $g_1(z)$ by
\begin{align*}
	g_1(z) \quad=\quad \left(
	1-\frac{\|x\|^2}{d}z
	\right)^{(d-p-k-1)/2}
	e^{\frac{k}{2}\|x\|^2}
\end{align*}
for $z < d/\|x\|^2$.
Then $g_1(z)$ can be expanded as
\begin{align*}
	g_1(z) \quad=\quad p_1(z-k) + r_1(z-k) + \delta_1,
\end{align*}
where $p_1(y) = 1 + \sum_{j=1}^k \left( -\|x\|^2/2\right)^j y^j/j!$; where
$r_1$ is a polynomial of degree $k$ whose coefficients depend on $d$, $p$, $k$
and $\|x\|^2$, and whose coefficients are bounded in absolute value by
$\frac{pM^{2(k+2)}}{d} c_1$, where $c_1=c_1(k)$ is a constant that depends
only on $k$; and where the remainder term $\delta_1$ satisfies
\begin{align*}
	\sup_{z:|z-k|<1} \frac{|\delta_1|}{|z-k|^{k+1}} \quad\leq\quad
	\frac{M^{2(k+1)}}{2^{k+1}(k+1)!}e^{\frac{k}{2}M^2}.
\end{align*}
\end{lemma}

\begin{proof}
We begin with a $k$-th order Taylor expansion of $g_1$ around $k$:
\begin{align*}
	g_1(z) \quad=\quad \sum_{j=0}^k 
	\frac{g_1^{(j)}(k)}{j!} (z-k)^j +
	\frac{g_1^{(k+1)}(\zeta)}{(k+1)!}(z-k)^{k+1}
\end{align*}
for some $\zeta$ between $z$ and $k$.
Now set
$\delta_1$ equal to the remainder term in the above expansion, i.e., 
$\delta_1 = \delta_1(z) = (g_1^{(k+1)}(\zeta))(z-k)^{k+1}/(k+1)!$, define
$p_1$ as in the lemma and set 
$r_1$ equal to $r_1(y)=\sum_{j=0}^k \frac{g_1^{(j)}(k)}{j!} y^j - p_1(y)$, 
such that  $g_1(z) = p_1(z-k) + r_1(z-k) + \delta_1$.
It remains to establish the claimed properties of $r_1$ and $\delta_1$.

Before we proceed, we first take a closer look at the $j$-th derivative of
$g_1$, i.e., 
\begin{align*}
	g_1^{(j)}(z) 
	&\quad=\quad 
	\left(
	-\frac{\|x\|^2}{2}
	\right)^j
	\frac{d-[p+k+1]}{d} \cdots \frac{d-[p+k+1+2(j-1)]}{d} \\
	&\quad\quad\quad \times
	\left(
	1-\frac{\|x\|^2}{d}z
	\right)^{(d-[p+k+1+2j])/2}
	e^{\frac{k}{2}\|x\|^2}.
\end{align*}
In this display and for $j=1\dots, k+1$,
note that $0 <d-[p+k+1+2(i-1)] < d$ for each $i=0,\dots,j+1$,
because $d > 4(p+k+1) > (p+k+1) + 2(k+1)$.

To deal with the remainder term $\delta_1$, fix $z\in \R$ such that $|z-k|<1$
and note that this entails 
$0\le \|x\|^2\zeta/d<M^2(k+1)/d<1/4$ by assumption. We therefore have
\begin{align*}
	\frac{|\delta_1|}{|z-k|^{k+1}} &\quad=\quad 
	\Bigg|
	\frac{(-\|x\|^2/2)^{k+1}}{(k+1)!}
	\Bigg|
	\Bigg| 
	\frac{g_1^{(k+1)}(\zeta)}{(-\|x\|^2/2)^{k+1}}
	\Bigg| \\
	&\quad\leq\quad
	\frac{(M^2/2)^{k+1}}{(k+1)!}
	\left(1-\frac{\|x\|^2}{d}\zeta
	\right)^{(d-p-3k-3)/2}
	e^{\frac{k}{2}\|x\|^2} \\
	&\quad\leq\quad
	\frac{M^{2(k+1)}}{2^{k+1}(k+1)!} e^{\frac{k}{2}M^2},
\end{align*}
where the first inequality follows from the formula for $g_1^{(k+1)}$
just derived, and where the second inequality easily follows from 
our assumptions on $d$. In particular,
the remainder term $\delta_1$ has the desired properties.

To deal with $r_1$, we first obtain convenient upper and lower bounds
on  $g_1^{(j)}(k)/(-\|x\|^2/2)^j$ for $j=0,\dots, k$.
Note that Lemma D.1 in \citep{Lee12b} applies with $t=
k\|x\|^2/d$, because our assumptions entail that $d> kM^2 \ge k\|x\|^2$
and hence $0\le k\|x\|^2/d < 1$.
This lemma yields
\begin{align}\label{LeebC1}
	\left(
	1- \frac{k\|x\|^2}{d}
	\right)^{k\|x\|^2/2}
	\quad\le\quad
	e^{k\|x\|^2/2} \left(
	1- \frac{k\|x\|^2}{d}
	\right)^{d/2}
	\quad\leq \quad 1.
\end{align}
Note that due to concavity of the logarithm, we have for $z\in(-1,0)$ and
$y\in[z,0]$, that $\log(1+y)\ge y \log(1+z)/z$, i.e., on $(z,0)$ the graph
of the function $y \mapsto\log(1+y)$ lies above the line segment connecting
$(z,\log(1+z))$ with the origin. Applying this inequality twice, 
the first time with $z=-1/4$ and $y=-kM^2/d$, and the second time 
with $z=-3/4$ and $y=-(p+3k-1)/d$, we get
\begin{align}
	&0 \,>\, \log\left( 1 - \frac{kM^2}{d}\right) 
	\,\ge\, \frac{4kM^2}{d}\log(\frac{3}{4}) 
	\,\ge\, -\frac{4}{3}\frac{kM^2}{d},	\text{ and} \label{help1}
	\\
	&0\,>\,\log\left( 1-\frac{p+3k-1}{d} \right) 
	\,\ge\,
	\frac{4(p+3k-1)}{3d}\log(\frac{1}{4})
	\,\ge\, -2 \frac{p+3k-1}{d}.		\label{help2}
\end{align}
With this we can bound the quantity 
$g_1^{(j)}(k)/( -\|x\|^2/2)^j$ 
in the $j$-th coefficient of $r_1$ with $j=0,1,\dots, k$,
from above as follows: 
\begin{align*}
	& \frac{g_1^{(j)}(k)}{ (-\|x\|^2 /2)^j}\quad\leq\quad
	\left(
	1-\frac{k\|x\|^2}{d}
	\right)^{-(p+k+1+2j)/2} 
	\quad\leq\quad
	\left(
	1-\frac{kM^2}{d}
	\right)^{-(p+3k+1)} 
	\\ &
	=\quad 
	\exp\left[ -(p+3k+1)\log \left( 1 -\frac{kM^2}{d} \right)\right] 
	\quad\leq\quad \exp\left[ (p+3k+1)\frac{4}{3}\frac{kM^2}{d}\right] \\
	& \leq \quad \exp\left[ \frac{pM^2}{d} \gamma_1(k)\right]
\end{align*}
for some constant $\gamma_1(k)>0$.
In this display, the first inequality  follows because, in the formula
for $g_1^{(j)}(z)$, the product of fractions is bounded by one,
and in view of \eqref{LeebC1}; the second inequality and the
first equality are trivial; the third inequality follows from
\eqref{help1}; and the last inequality holds by setting, say,
$\gamma_1(k) = 4 k (2 + 3 k)/3$.
And we can also bound $g_1^{(j)}(k)/( -\|x\|^2/2)^j$ for $j=0,1,\dots, k$
from below:
\begin{align*}
	&\frac{g_1^{(j)}(k)}{\left( -\|x\|^2/2\right)^j} 
	 \quad\ge \quad 
	\left(\frac{d-[p+k+1+2(j-1)]}{d}\right)^j 
	\left(
	1- \frac{k\|x\|^2}{d}
	\right)^{k\|x\|^2/2} 
	\\
	&\ge \quad \left(1-\frac{p+3k-1}{d}\right)^{k} 
	\left(
	1- \frac{kM^2}{d}
	\right)^{kM^2/2} 
	\\
	&=\quad  \exp\left[
	k\log\left(
	1-\frac{p+3k-1}{d}
	 \right) 
	 + \frac{kM^2}{2} 
	 \log\left(
		1- \frac{kM^2}{d}	 
	\right)	 
	\right] \\
	&\ge  \quad \exp\left[
	-k2 \frac{p+3k-1}{d}	
	 - \frac{kM^2}{2} 
	\frac{4}{3}\frac{kM^2}{d} 
	\right] \\
	&\ge\quad
	\exp\left[ 
	-\frac{pM^4}{d} \gamma_2(k)
	\right]
\end{align*}
for some $\gamma_2(k)>0$.
Here, 
the first inequality follows from the formula for $g_1^{(j)}(z)$
derived earlier and from \eqref{LeebC1};
the second inequality holds because $j \leq k$ and because $\|x\|\leq M$;
the equality is trivial;
the third inequality is obtained from \eqref{help1} and \eqref{help2};
and the last inequality holds for, say, 
$\gamma_2(k) = (6+2/3)  k^2$.

For $r_1$, we note that $r_1(y)$ is a polynomial of degree $k$ in $y$,
i.e., $r_1(y) = \sum_{j=0}^k \rho_j y^j$, where the $j$-th coefficient
satisfies
\begin{align*}
	| \rho_j | \quad=\quad \left| 
	\frac{g_1^{(j)}(k)}{j!} - \frac{\left( -\|x\|^2/2\right)^j}{j!}
	\right| \quad\leq\quad
	\frac{\left( M^2/2\right)^j}{j!}
	\left| 
	\frac{g_1^{(j)}(k)}{\left( -\|x\|^2/2\right)^j} - 1
	\right|.
\end{align*} 
Setting 
$\gamma(k) = \max\{\gamma_1(k),\gamma_2(k)\}$, we thus get
\begin{align*}
	|\rho_j| \quad& \leq\quad
	M^{2k} 
	\max\left\{ 
	1 - \exp\left[ 
			-\frac{pM^4}{d} \gamma(k)
		\right], 
	 \exp\left[ 
				\frac{pM^4}{d} \gamma(k)
			\right] - 1
	\right\} 
	\\
	&=\quad
	M^{2k} 
	\left( \exp\left[ 
				\frac{pM^4}{d} \gamma(k)
			\right] - 1 
	\right)
	\\
	&\leq\quad
	\frac{pM^{2(k+2)}}{d}
	\left( e^{\gamma(k)} - 1\right).
\end{align*}
for $j=0,\dots, k$.
In this display, the first inequality follows from the formula
for $|\rho_j|$, and from the upper and lower bounds
on $g_1^{(j)}(k)/( -\|x\|^2/2)^j$ derived earlier;
the equality is trivial;
and the last inequality follows because $\gamma(k)>0$ and 
$p M^4/d \in (0,1)$ and because
$e^{cy}-1 \le (e^c -1)y$  for $c \geq 0$ and $y\in (0,1)$
by the convexity of the exponential. Setting $c_1(k) = e^{\gamma{(k)}}-1$,
we see that $r_1$ has the desired properties.
\end{proof}

\begin{lemma}\label{Taylor3}
Let $k$ and $p$ be positive integers and consider the function $g_2(z) = z^{-p/2}$ for $z> 0$. Then $g_2$ can be expanded as
\begin{align*}
	g_2(z) \quad=\quad p_2(z-1) + \delta_2,
\end{align*}
where $p_2(y) = 1 +\sum_{j=1}^k \left[\frac{(-1/2)^j}{j!}  \prod_{i=0}^{j-1}
(p+2i)\right]  y^j$ is a polynomial of degree $k$ whose $j$-th coefficient 
can be bounded in absolute value by $p^j$, and where the remainder term 
$\delta_2$ satisfies 
\begin{align*}
	\sup_{z:|z-1|<\frac{1}{2p}} \frac{|\delta_2|}{|z-1|^{k+1}} 
	\quad \leq \quad
	p^{k+1}
	d_2
\end{align*}
with $d_2 = d_2(k)$ depending only on $k$.
\end{lemma}

\begin{proof}
A $k$-th order Taylor expansion of $g_2$ around one shows that
$g_2(z)$ can be written as $p_2(z-1) + \delta_2$, 
where $p_2$ is as claimed, and where $\delta_2$ is the remainder term.
It is now elementary to verify that
the coefficients of $p_2$ are bounded as claimed.
The remainder term $\delta_2$ is given by
$g_2^{(k+1)}(\zeta) (z-1)^{k+1}/(k+1)!$ for some $\zeta$ between $z$ and $1$.
For $z$ such that $|z-1| < 1/(2 p)$, we thus have
\begin{align*}
	&\frac{|\delta_2|}{|z-1|^{k+1}} 
	\quad=  \quad
	\zeta^{-p/2-(k+1)}\frac{(1/2)^{k+1}}{(k+1)!} \prod_{i=0}^{k} (p+2i)  
	\\ 
	&\le \quad
	\left( 1- \frac{1}{2p}\right)^{-p/2}
	\left( 1- \frac{1}{2p}\right)^{-(k+1)} p^{k+1}  
	\quad\le \quad 
	(2e)^{1/4} 
	2^{k+1} p^{k+1} ,
\end{align*}
where the last inequality is obtained by using 
Lemma D.1 in \citep{Lee12b} with $t=1/(2p)$ to obtain that
$(1-1/(2 p))^{-p/2} \leq (2 e)^{1/4}$.
Setting $d_2 = (2 e)^{1/4} 2^{k+1}$ completes the proof.
\end{proof}

\begin{lemma}\label{Taylor4}
Fix $M>1$, positive integers $k$ and $p$, and $x\in \R^p$ such that $\|x\|\le
M$. Choose $d\in \N$ such that $d>4(k+p+1)M^4$,
and define $g_1(\cdot)$ as in Lemma~\ref{Taylor2}. 
Take $d$-vectors $w_1,\dots, w_k$ so that the $k\times k$ matrix
$S_k = (w_i'w_j/d)_{i,j=1}^k$ satisfies $||S_k - I_k|| < 1/(2k)$.
We then have
$$
	g_1(\iota' S_k^{-1}\iota)
	\quad=\quad 
	p_1\left( \sum_{j=1}^k \iota' (I_k-S_k)^j \iota\right)
	\;+\;
	r_1\left( \sum_{j=1}^k \iota' (I_k-S_k)^j \iota\right)
	\;+\; \
	\Delta_1,
$$
where the polynomials $p_1(\cdot)$ and $r_1(\cdot)$ are as in
Lemma~\ref{Taylor2},
and where the remainder term $\Delta_1$  satisfies
$$
	|\Delta_1|\quad\leq\quad D_1 ||S_k  - I_k||^{k+1}
	M^{2(k+2)}e^{\frac{k}{2}M^2}
$$
for some constant $D_1=D_1(k)$ that depends only on $k$.
\end{lemma}

\begin{proof}
Set $z=\iota'S_k^{-1}\iota$, and note that $g_1(z)$ is well-defined
because $z \|x\|^2/d<1$.
[Indeed, all eigenvalues of $S_k$ are strictly between $1/2$ and $3/2$ by
assumption, so that $||S_k^{-1}|| < 2$ and $\iota' S_k^{-1}\iota < 2 k$;
our assumptions on $d$ and $x$ give $z \|x\|^2/d < 1$.]
It is now easy to see that Lemma~\ref{Taylor2} applies, so that we can
write $g_1(z)$ as 
\begin{align*}
	g_1(\iota'S_k^{-1}\iota) \quad=\quad p_1(\iota'S_k^{-1}\iota-k)
	+r_1(\iota'S_k^{-1}\iota-k)+\delta_1,
\end{align*}
such that $\delta_1$ satisfies 
\begin{align*}
	|\delta_1| \quad&\le \quad \frac{M^{2(k+1)}}{2^{k+1}(k+1)!}
		e^{\frac{k}{2}M^2}
	|\iota'S_k^{-1}\iota
	-k|^{k+1} \\
	\quad&\le\quad \frac{M^{2(k+1)}}{2^{k+1}(k+1)!}
		e^{\frac{k}{2}M^2} (2 k)^{k+1} ||S_k -
	I_k||^{k+1},
\end{align*}
where the inequalities in the preceding display
hold because 
$|\iota' S_k^{-1}\iota - k| = |\iota' (S_k^{-1} - I_k) \iota| \le k ||S_k^{-1}
- I_k|| = 
k ||S_k^{-1}(I_k-S_k) || \leq
k ||S_k^{-1}|| \; ||S_k - I_k|| < 2  k ||S_k - I_k|| \le 1$.
For later use, we also 
note that the functions $p_1(\cdot)$ and $r_1(\cdot)$
in the second-to-last display are polynomials of degree $k$
whose coefficients can be bounded, in absolute value, by 
$L(k,M) = M^{2(k+2)}\max\{1,c_1\}$, where $c_1=c_1(k)$ is the constant from
Lemma~\ref{Taylor2}
that depends only on $k$. [Indeed, the $j$-th coefficient of $p_1$ is
given by $(-\|x\|^2/2)^j/j!$ and hence is bounded, in absolute value, by
$M^{2k}$, and
the coefficients of $r_1$ are bounded, in absolute value,
by $\frac{pM^{2(k+2)}}{d} c_1 < c_1 M^{2(k+2)} $.]

Now expand $\iota'S_k^{-1}\iota-k$ as in Corollary D.5(i) in \citep{Lee12b}
with $k$ replacing $m$, i.e., as
\begin{align*}
	\iota'S_k^{-1}\iota-k \quad=\quad 
	\sum_{j=1}^k \iota' (I_k - S_k)^j \iota
	\quad+\quad r,
\end{align*}
where the main term satisfies $|\sum_{j=1}^k \iota' (I_k - S_k)^j \iota| \le 1$
by assumption and $|r|< 2k\|S_k-I_k\|^{k+1}\le 1$.
Moreover, recall that $p_1$ is a polynomial of degree $k$
whose coefficients are bounded, in absolute value, by the constant $L = L(k,M)$
defined in the preceding paragraph.
In view of these considerations, we see that Lemma~D.6 of \citep{Lee12b}
applies, with $1$, $p_1$, $k$, $L(k,M)$, $1$, and $2 k \|S_k-I_k\|^{k+1}$
replacing $m$, $\rho_1$, $l$, $L$, $B$, and $\epsilon$.
From that lemma, and from the last line of its proof, we obtain that
$$
	p_1(\iota' S_k^{-1}\iota - k) \quad=\quad 
	p_1\left( \sum_{j=1}^k \iota'(I_k - S_k)^j \iota\right) \; + \;
	\delta_p
$$
with $|\delta_p|  < \|S_k-I_k\|^{k+1} M^{2(k+2)} 
[ 2 k (k+1+2^{k+1}) \max\{c_1,1\} ]$.
In a similar fashion, we also obtain that 
$$
	r_1(\iota' S_k^{-1}\iota - k) \quad=\quad 
	r_1\left( \sum_{j=1}^k \iota'(I_k - S_k)^j \iota\right) \; + \;
	\delta_q
$$
with $|\delta_q|  < \|S_k-I_k\|^{k+1} M^{2(k+2)} 
[ 2 k (k+1+2^{k+1}) \max\{c_1,1\} ]$.
In view of the preceding two paragraphs, we see that 
$g_1(\iota'S_k^{-1}\iota)$ can be expanded as claimed, with 
$\Delta_1 = \delta_1 + \delta_p + \delta_q$.
\end{proof}

\begin{lemma}\label{Taylor5} 
Fix  positive integers $k$, $p$ and $d$ such that $1\le k\le d-p$. Then there
exists 
a constant $\xi(k)>2$ depending only on $k$ for which the following holds true.
For
any collection of $d$-vectors $w_1,\dots, w_k$, for which the $k\times k$-matrix
$S_k = (w_i' w_j/d)_{i,j=1}^k$
satisfies $||S_k - I_k || < 1/(p\, \xi(k))$, 
the determinant $\det S_k^{-p/2}$ can be expanded as
$$
	\det S_k^{-p/2} \quad = \quad Q_2(S_k-I_k) 
	\;+\;\Delta_2,
$$
where $Q_2$ and $\Delta_2$ have the following properties:
$Q_2(S_k-I_k)$ is a polynomial in the elements of $S_k-I_k$ whose degree
depends only on $k$ 
and whose coefficients are bounded in absolute value by $p^k  C_2$, where $C_2
= C_2(k)$ is a constant that depends only on $k$. The remainder term $\Delta_2$
satisfies
$$
	|\Delta_2| \quad\leq \quad D_2 \,p^{k+1} ||S_k - I_k ||^{k+1}
$$
for a constant $D_2 = D_2(k)$ that depends only on $k$.
\end{lemma}

\begin{proof}
We begin with a collection of $d$-vectors $w_1,\dots, w_k$ such that
$\|S_k-I_k\|<1/2$.
First note that the determinant of $S_k$ can be written as
$$
	\det S_k \quad=\quad \prod_{i=1}^k\left( 
	\frac{1}{d} || (I_d - P_{N_{(i-1)}}) w_i||^2 \right),
$$
where $P_{N_{(i-1)}}$ denotes the matrix of the orthogonal projection 
on the column space of $N_{(i-1)} = [w_1,\dots,w_{i-1}]$ in case $i>1$ and is
to be interpreted
as the $d\times d$ zero matrix in case $i=1$. 
[Recall the relation
$\det( [X,Y]' [X,Y]) = \det(X'X) \det(Y' (I-P_X)Y)$
for appropriate matrices $X$ and $Y$, and use this
recursively with $X=N_{(i-1)}$ and $Y=w_i$.]
The first factor on the right hand side of the preceding display equals
$w_1'w_1/d$, while for $i>1$ the factor corresponding to index $i$ 
can be written as 
\begin{align*}
	\frac{w_i'w_i}{d}
	\,-\, \sum_{j=0}^k \frac{w_i'N_{(i-1)}}{d}
	\left( 
	I_{i-1} - \frac{N_{(i-1)}'N_{(i-1)}}{d}
	\right)^j
	\frac{N_{(i-1)}'w_i}{d} \,+\, r_i
\end{align*}
with $|r_i| < 2\|S_k-I_k\|^{k+3}$
in view of Corollary D.5(ii) in \citep{Lee12b}.
Write $T_i(S_k-I_k)$ as shorthand for
\begin{align*}
	1 + \left(\frac{w_i'w_i}{d} -1\right)
	- \sum_{j=0}^k \frac{w_i'N_{(i-1)}}{d}
	\left( 
	I_{i-1} - \frac{N_{(i-1)}'N_{(i-1)}}{d}
	\right)^j
	\frac{N_{(i-1)}'w_i}{d}
\end{align*}
for $i=1,\dots, k$, 
where the sum is to be interpreted as zero in case $i=1$.
Note that $T_i(S_k-I_k)$ 
is a polynomial in the entries of $S_k-I_k$ of degree 1 in
case $i=1$
and of degree $k+2$ in case $i>1$, whose coefficients depend only on $k$.
Clearly, the coefficient of the linear term $(w_i'w_i/d-1)$ in $T_i(S_k-I_k)$
is one, and the coefficient of each higher order term
originating from the $j$-th term in the sum
is bounded, in absolute value, by a constant depending only on $k$.
The coefficients of
$T_i(S_k-I_k)$ thus are bounded, in absolute value, by
a constant depending only on $k$.
Also note that $|T_i(S_k-I_k)|$ is bounded by some constant, $B(k)$ say, that
depends only on $k$. [Indeed, we have
$|w_i'w_i/d - 1| = |e_i'(S_k-I_k)e_i| \le \|S_k-I_k\| < 1/2$. 
And, in case $i>1$,
the absolute value of the $j$-th term in the sum in $T_i$ can be bounded in a
similar way by $\| N_{(i-1)}'N_{(i-1)}/d - I_{i-1}\|^j \|N_{(i-1)}'w_i/d\|^2 \le
\|S_k-I_k\|^j \sum_{l=1}^{i-1} (e_l'(S_k-I_k)e_i)^2 \le (i-1)\|S_k-I_k\|^{j+2}
\le k \|S_k-I_k\|^{j+2} < k/4$.
We may thus choose $B(k) = 1+1/2 + k(k+1)/4$.] 
With this we can write the determinant of $S_k$ as
\begin{align*}
	\det S_k \quad=\quad \prod_{i=1}^k \left[ T_i(S_k-I_k) + r_i \right].
\end{align*}
We now apply Lemma D.6 in \citep{Lee12b} with $m=k$, $\rho_i$ 
equal to the identity
function, 
$x_i=T_i(S_k-I_k)$, $u_i = r_i$, $l=L=1$, $B=B(k)$ and
$\epsilon = 2\|S_k-I_k\|^{k+3}$.
From that lemma, and from the last line of its proof, we see that
\begin{align*}
	\det S_k \quad=\quad \prod_{i=1}^k  T_i(S_k-I_k)  + r,
\end{align*}
where $r$ satisfies $|r| < \kappa(k) \|S_k-I_k\|^{k+3}$
for $\kappa(k) = 2^{k+1} 3^k B(k)^{k-1}$.
For later use, note that $\prod_{i=1}^k T_i(S_k-I_k)$ is a polynomial
in $S_k-I_k$ whose degree depends only on $k$, and whose coefficients
are bounded in absolute value by a constant depending only on $k$.
Expanding the product in the preceding display, we get
\begin{align*}
	\det S_k \quad=\quad  1 + \trace{(S_k-I_k)} + R(S_k-I_k)  + r,
\end{align*}
where $R(S_k-I_k)$ is a polynomial in the elements of $S_k-I_k$ whose degree
and coefficients depend only on $k$ and that contains no constant or linear
term. Note that $|\trace{(S_k-I_k)} + R(S_k-I_k)|\le
C(k) \|S_k-I_k\|$ for some constant $C(k)$ depending only on $k$,
e.g., $C(k) = (1+k^2)^k$. [Bounding the terms of $T_i(S_k-I_k)$
as in the preceding discussion, we see that
$|1-T_i(S_k-I_k)| \leq (1+(k+1)(k-1))\|S_k-I_k\| = k^2\|S_k-I_k\|$.
By this, it is easy to see that $|\prod_{i=1}^k T_i- 1|$,
is bounded by  $(1+k^2)^k \|S_k-I_k\|$.]

Next we define the constant $\xi$ by 
$\xi(k) = 2[ C(k)+\kappa(k)]$, note that $\xi(k)>2$, 
and take
$d$-vectors $w_1,\dots, w_k$ such that $\|S_k-I_k\|<1/(p\, \xi(k))$. Note that
for this new collection of vectors everything we have shown so far still holds
true since $\|S_k-I_k\|<1/(p\, \xi(k))<1/2$.
To finish the proof we now evaluate
$z^{-p/2}$ for $z = \det S_k$ 
(which is well-defined because $S_k$ is positive definite by
assumption). Now $|z-1| \le |\trace{(S_k-I_k)} +R(S_k-I_k)| + |r| \le [C(k) +
\kappa(k)]\|S_k-I_k\| < 1/(2p)$. Hence Lemma~\ref{Taylor3} entails that
\begin{align} \nonumber
	 \det S_k ^{-p/2} 
	 \quad=\quad
	 p_2 \Big( \trace{(S_k-I_k)} +R(S_k-I_k)  + r \Big) +
	 \delta_2
\end{align}
with $|\delta_2| \le d_2(k) p^{k+1} |z-1|^{k+1} \le
d_2(k)(\xi(k)/2)^{k+1}  p^{k+1}\|S_k-I_k\|^{k+1}$. 
Abbreviate the coefficients of $p_2$ by $\rho_j$ for $j=1,\dots, k$ and set
$t=\trace{(S_k-I_k)}$. Then we can evaluate the polynomial in the preceding
display as
\begin{align*}
	&p_2( t + R(S_k-I_k)  + r ) 
	\quad=\quad
	1 + \sum_{j=1}^k \rho_j ( t + R(S_k-I_k)
	+ r )^j  \\
	&\quad = \quad
	1 + \sum_{j=1}^k \rho_j \left( \sum_{i=0}^j \binom{j}{i}
	(t+R(S_k-I_k))^{i}  r^{j-i} \right) \\
	&\quad = \quad
	1 + \sum_{j=1}^k \rho_j (t+R(S_k-I_k))^j + 
	\sum_{j=1}^k \rho_j \left( \sum_{i=0}^{j-1} \binom{j}{i}
	(t+R(S_k-I_k))^{i}  r^{j-i} \right) \\
	&\quad = \quad
	p_2(t+R(S_k-I_k)) + \tilde{r},
\end{align*}
where the remainder term $\tilde{r}$ satisfies
\begin{align*}
	|\tilde{r}| 
	&\quad \le \quad
	\sum_{j=1}^k p^j \sum_{i=0}^{j-1}\binom{j}{i}
	C(k)^{i} \kappa(k)^{j-i} \|S_k-I_k\|^{(j-i)(k+3)+i} \\
	& \quad\le \quad
	p^{k+1} \|S_k-I_k\|^{k+3} \sum_{j=1}^k
	\sum_{i=0}^{j-1}\binom{j}{i} C(k)^{i} \kappa(k)^{j-i}. 
\end{align*}
Now set $Q_2(S_k-I_k) = p_2(\trace{(S_k-I_k)} + R(S_k-I_k))$
and set $\Delta_2 = \tilde{r}+\delta_2$.
By the discussion in the preceding paragraph and by Lemma~\ref{Taylor3},
we see that $Q_2(S_k-I_k)$ is a polynomial in $S_k-I_k$
whose degree depends only on $k$,
and whose coefficients are bounded in absolute value
by $p^k C_2(k)$ for some constant $C_2(k)$ that depends only on $k$.
The bound on $|\delta_2|$ derived earlier and the 
bound in $|\tilde{r}|$ in the preceding display entail
that $|\Delta_2| \leq D_2(k) p^{k+1}\|S_k-I_k\|^{k+1}$
for $D_2(k)$ depending only on $k$.
\end{proof}


\begin{proof}[\bf Proof of Proposition~\ref{p3}]
Any expansion is trivially correct unless we specify some properties of the
remainder term. 
So choose the constant $\xi(k)$ equal to $\xi(k) = \xi'(k) + 2k>2k$, where
$\xi'(k)$ is the constant $\xi$ given by Lemma~\ref{Taylor5}. Now take a
collection of $d$-vectors $w_1,\dots, w_k$ such that $\|S_k-I_k\| < 1/(p\,
\xi(k))$ and note that this also permits the determinant expansion given by
Lemma~\ref{Taylor5} with the error term as described there. Note that
$\|S_k-I_k\|<1/(2p)\leq 1/2$ and thus all eigenvalues of $S_k$ are 
strictly between $1/2$
and $3/2$. Therefore the largest eigenvalue of $S_k^{-1}$, i.e.,
$\|S_k^{-1}\|$,
is less than $2$. We conclude that $\|x\|^2\iota'S_k^{-1}\iota/d \le
M^2k\|S_k^{-1}\|/d \le 2kM^2/d < 1$ by assumption. In view of
Proposition~\ref{p2}, we thus can write the density ratio of interest as
\begin{align*}
	\frac{\varphi_x(w_1,\dots,w_k)}{\phi(w_1)\cdots\phi(w_k)} 
	\quad=\quad
	\eta(d,p,k) \; \gamma_1(S_k) \; \gamma_2(S_k),
\end{align*}
for $\eta(d,p,k)$  as in Proposition~\ref{p2}, for 
\begin{align*}
	\gamma_1(S_k) 
	\quad=\quad
	\left(
	1-\frac{\|x\|^2}{d}\iota'S_k^{-1}\iota
	\right)^{(d-p-k-1)/2}
	e^{\frac{k}{2}\|x\|^2},
\end{align*}
and for $\gamma_2(S_k) = \det(S_k)^{-p/2}$.
Recall that $\eta(d,p,k)$ is bounded as 
described by Proposition~\ref{p2}, and hence it is also bounded by
$\exp(2p^2k^2/(3d) ) \le \exp(2k^2/3)$, because our assumptions on $d$, $p$ and
$k$ entail that $(1 - \frac{p+k-1}{d})^{-1} \le 4/3$ and $p^2/d < 1$. 

Note that we have $\gamma_1(S_k) = g_1(\iota' S_k^{-1}\iota)$ for
$g_1$ as in Lemma~\ref{Taylor2}. And
since $\|S_k-I_k\|< 1/(p \xi(k)) < 1/(2k)$, Lemma~\ref{Taylor4}
entails that
\begin{align*}
	\gamma_1(S_k) \quad=\quad
	p_1\left( \sum_{j=1}^k\iota'(I_k-S_k)^j\iota\right) 
	+ r_1\left( \sum_{j=1}^k\iota'(I_k-S_k)^j\iota\right) 
	+ \Delta_1,
\end{align*}
where $\Delta_1$ satisfies $|\Delta_1|\le D_1(k) \|S_k  - I_k\|^{k+1}
M^{2(k+2)}e^{\frac{k}{2}M^2}$. Using the properties of $p_1$ and $r_1$ 
as given by Lemma~\ref{Taylor2} we
can rewrite this as
\begin{align*}
	\gamma_1(S_k) 
	\quad=\quad
	Q_1(S_k-I_k) + \Delta_1.
\end{align*}
where $Q_1(S_k-I_k)$ is a polynomial in the elements of $S_k-I_k$ whose degree
depends only on $k$ and whose coefficients are bounded in absolute value by
$M^{2(k+2)} C_{Q_1}$ for a constant $C_{Q_1}=C_{Q_1}(k)$ that depends only on
$k$.

As already mentioned at the beginning of the proof, also Lemma~\ref{Taylor5}
applies and entails that
\begin{align*}
	\gamma_2(S_k) 
	\quad=\quad
	Q_2(S_k-I_k) + \Delta_2,
\end{align*}
where $Q_2(S_k-I_k)$ is a polynomial in the elements of $S_k-I_k$ whose degree
depends only on $k$ and whose coefficients are bounded in absolute value by
$p^k C_{Q_2}$ where $C_{Q_2} = C_{Q_2}(k)$ is a constant that depends
only on $k$, and where the remainder term $\Delta_2$ satisfies
$$
	|\Delta_2| \quad\leq \quad D_2 \; p^{k+1} \; ||S_k - I_k ||^{k+1}
$$
for a constant $D_2 = D_2(k)$ that depends only on $k$.

Now expanding the product of $\gamma_1(S_k)$ and $\gamma_2(S_k)$ yields a sum
of 4 terms. Denote the product of the two polynomials $Q_1$ and $Q_2$ by
$\tilde{\psi}_x(S_k-I_k) = Q_1(S_k-I_k)Q_2(S_k-I_k)$ and note that this is a
polynomial in the elements of $S_k-I_k$ whose degree depends only on $k$ and
whose coefficients are bounded by $p^kM^{2(k+2)}\tilde{C}$, where $\tilde{C} =
\tilde{C}(k)$ is a constant that depends only on $k$. Moreover, denote the sum
of the three remaining terms by $\tilde{\Delta} = \tilde{\Delta}(S_k-I_k)$ and
note that it satisfies $|\tilde{\Delta}| \le
p^{k+1}M^{2(k+2)}e^{kM^2/2}\|S_k-I_k\|^{k+1} \tilde{D}$, where $\tilde{D} =
\tilde{D}(k)$ is a constant that depends only on $k$.
All together, we have obtained that
\begin{align*}
	&\frac{\varphi_x(w_1,\dots,w_k)}{\phi(w_1)\cdots\phi(w_k)} = 
	\tilde{\psi}_x(S_k-I_k) + \tilde{\Delta},
\end{align*}
where $\eta(d,p,k)$ has been incorporated into the coefficients of
$\tilde{\psi}_x$ and the remainder term $\tilde{\Delta}$ without affecting
their boundedness properties. 
Now it is no loss of generality to assume that the degree of $\tilde{\psi}_x$
is no more than $k$, since all the terms which are of order $k+1$ or higher,
can simply be added to the remainder term $\tilde{\Delta}$ without affecting
its boundedness property (only the constant $\tilde{D}$ needs adjustment),
because $\|S_k-I_k\| < 1$ and thus $\|S_k-I_k\|^{k+1+l} \le \|S_k-I_k\|^{k+1}$,
for all $l\ge 0$.

The quantities $\tilde{\psi}_x(S_k-I_k)$ and $\tilde{\Delta}$ 
discussed in the preceding paragraph
have all the properties that we need to 
derive for $\psi_x(S_k-I_k)$ and $\Delta$,
respectively, except for permutation invariance. But the expression
on the left-hand side of the preceding display is invariant under
permutations of the $w_i$ in view of Proposition~\ref{p2}.
Set $\psi_x(S_k-I_k)$ and $\Delta$ equal to the average of
$\tilde{\psi}_x(S_k-I_k)$ and $\tilde{\Delta}$, respectively,
taken over all permutations of the $w_i$.
Clearly,  the relation in the preceding display continues to
hold with $\psi_x(S_k-I_k)$ and $\Delta$ replacing
$\tilde{\psi}_x(S_k-I_k)$ and $\tilde{\Delta}$, respectively,
and $\psi_x(S_k-I_k)$ and $\Delta$ have all the required properties,
including permutation invariance.
\end{proof}


\begin{lemma}\label{p4.l1}
Let $d$, $p$ and $k$ be positive integers and suppose that 
$Z_1,\dots,Z_k$  are i.i.d. random $d$-vectors  that
satisfy the bound~\ref{c.density} with $k$ as chosen here. 
If $\zeta>0$ and $d>\max\{2k+4p\zeta,p^2\}$, then 
\begin{align*}
	\E\left[ \left(
		\frac{d}{\|(I_d- P_{Z_1,\dots,Z_{k-1}})Z_k\|^2}
		\right)^{p\zeta}\right]
	\quad\le\quad
	C (4\pi e D^2)^{p\zeta},
\end{align*} 
where $C = C(k,\zeta)$ depends only on $k$ and $\zeta$ and where $D$ is the
constant from \ref{c.density}.
\end{lemma}

\begin{proof}
Write $R_d$ as shorthand for $R_d = \|(I_d- P_{Z_1,\dots,Z_{k-1}})Z_k\|^2/d$,
and take $L>0$ to be chosen later. The expectation in question can be
calculated as
\begin{align*}
	\E\left[R_d^{-p\zeta}\right] 
	\quad=\quad\int_0^\infty \P\left( R_d^{-p\zeta} >
	t\right)\;dt 
	\quad\leq \quad
	L + \int_L^\infty \P\left( R_d < t^{-\frac{1}{p\zeta}}\right)\;dt.
\end{align*}
Now use the exact same argument as in the proof of Proposition E.1 in
\citep{Lee12b} (with $p\zeta$ replacing $\alpha$ in that reference), which
involves \ref{c.density} (or, equivalently, condition (t2) in that
reference), to bound the integral on the far
right-hand-side of the previous display by
\begin{align*}
	L\frac{2p\zeta}{d-k+1-2p\zeta}
	\left( \frac{D^2\pi d}{L^{\frac{1}{p\zeta}}}\right)^{\frac{d-k+1}{2}} 
	\frac{d^{\frac{k-1}{2}}}{\Gamma\left( \frac{d-k+3}{2}\right)}.
\end{align*}
By Theorem 1 of \citep{Kec71a} with $x\geq 2$ and with $y=2$,
we obtain that $\Gamma(x) \geq x^{x-1} \exp(2-x)/2 > x^{x-1} \exp(-x)/2$.
Using this to lower-bound the gamma-function in the preceding display,
it is elementary to verify that
\begin{align}
	\E\left[R_d^{-p\zeta}\right] \quad\le\quad
	L + L^{1-\frac{d-k+1}{2p\zeta}} 4p\zeta  \left( 2D^2\pi
	e\frac{d}{d-k+3}\right)^{\frac{d-k+1}{2}}
	d^{\frac{k-1}{2}} e. \label{eq:l1.bound}
\end{align}
Write $K$ to abbreviate the factor in \eqref{eq:l1.bound} starting with
$4p\zeta\dots$, ending with $\dots e$, and note that it does not depend on $L$.
Also note that the upper bound in \eqref{eq:l1.bound} is a convex function in
$L$ that is minimized at
\begin{align*}
	L_0 \quad=\quad \left(
	\frac{d-k+1}{2p\zeta}-1\right)^{\frac{2p\zeta}{d-k+1}}
	K^{\frac{2p\zeta}{d-k+1}}.
\end{align*}
Plugging this back into \eqref{eq:l1.bound} yields the optimized bound
\begin{align*}
	L_0 + L_0 \left(\frac{d-k+1}{2p\zeta} - 1\right)^{-1} 
	\quad<\quad 2L_0.
\end{align*}
It remains to show that $2 L_0$ is of the form
$C(k,\zeta) (4 \pi e D^2)^{p \zeta}$ as claimed.
But this follows immediately upon writing out the formula for $2 L_0$
and noting that 
$\frac{d}{d-k+3}\le 2$, that
$2 p \zeta / (d-k+1) < 1/2$, and that the following quantities
are all bounded by a constant that depends only on $k$ and $\zeta$:
\begin{align*}
	 \left(\frac{d-k+1}{2p\zeta} -1\right)^{\frac{2p\zeta}{d-k+1}} 
	&\quad=\quad 
	\exp\left( \frac{2p\zeta}{d-k+1} \log\left(\frac{d-k+1}{2p\zeta}
	-1\right)\right)
	\quad\le \quad e, 
\\
	(4p\zeta)^{\frac{2p\zeta}{d-k+1}} 
	&\quad=\quad
	\exp\left(\frac{p}{\sqrt{d}} \frac{\log(p) + \log(4\zeta)}{\sqrt{d}}
	\frac{2\zeta}{1-k/d+1/d} \right),
\\
	d^{\frac{(k-1)p\zeta}{d-k+1}} 
	&\quad=\quad
	\exp\left( \frac{(k-1)\zeta}{1-k/d+1/d}\frac{p}{d} \log d \right) 
	\\
	&\quad\le \quad
	\exp\left( 2(k-1)\zeta \frac{p}{\sqrt{d}} \frac{\log d}{\sqrt{d}}
	\right).
\end{align*}
\end{proof}


\begin{proof}[\bf Proof of Proposition~\ref{p4}]
Set $h = \deg(H)$ and fix $x \in \mathcal S_{M,p}$.  It suffices to show that 
\begin{equation}\nonumber
	\E\left[ d^\frac{h+l}{2} \| S_l- I_l \|^{h}\; \left|
	\frac{\varphi_{x}(Z_1,\dots, Z_l)}{\phi(Z_1)\cdots\phi(Z_l)} - 
	\psi_{x}(S_l - I_l)\right|\right]
\end{equation}
is bounded as claimed.
Note that Proposition~\ref{p3} applies with $l$ replacing $k$.
Write $U_d$ for the integrand in the preceding display and note that
$\E[U_d]$ can also be written as
\begin{equation}\label{pp4.1}
	\E\left[ U_d \left\{ \|S_l - I_l\| < \frac{1}{p\,\xi(l)}\right\} \right]
	+
	\E\left[ U_d \left\{ \|S_l - I_l\| \geq  \frac{1}{p\,\xi(l)}\right\}
	\right],
\end{equation}
where $\xi(l)>2l$ is the constant from Proposition~\ref{p3}.
We will show that both terms in \eqref{pp4.1} can be bounded
appropriately.

For the first term in \eqref{pp4.1}, we use Proposition~\ref{p3} to conclude
that $U_d \{ ||S_l - I_l||<1/(p\,\xi(l))\}$ is bounded from above by
\begin{align*}
	& C_\Delta 
	d^{(h+l)/2} 
	p^{l+1}M^{2(l+2)}e^{\frac{l}{2}M^2}
	||S_l - I_l||^{h+l+1} \\
	&\quad \le\quad
	C_\Delta p^{k+1}M^{2(k+2)}e^{\frac{k}{2}M^2}d^{-1/2} 
	||\sqrt{d}(S_l - I_l)||^{h+l+1}, 
\end{align*}
where $C_\Delta$ is a constant that depends only on $k$. 
Note that $\E[ ||\sqrt{d}(S_l - I_l)||^{h+l+1}]$ 
is bounded by the constant $\alpha$ from \ref{c.Sk}.(\ref{c.sup})
in view of Lyapunov's inequality.

The second term in \eqref{pp4.1} can further be bounded from above by
\begin{align}
	\label{pp4.2}
	&\E\left[ d^\frac{h+l}{2} \| S_l- I_l \|^{h}\; \left|
	\psi_{x}(S_l - I_l)
	\right|\left\{ \|S_l - I_l\| \geq\frac{1}{p \, \xi(l)}\right\} \right]
	\\
	\label{pp4.3}
	&
	+ \quad 
	\E\left[ d^\frac{h+l}{2} \| S_l- I_l \|^{h}\; \left|
	\frac{\varphi_{x}(Z_1,\dots, Z_l)}{\phi(Z_1)\cdots\phi(Z_l)}  
	\right|\left\{ \|S_l - I_l\| \geq\frac{1}{p \, \xi(l)}\right\} \right].
\end{align}
Concerning \eqref{pp4.2}, recall that $l\le k$ and that $\psi_{x}(S_l - I_l)$
is a polynomial 
of degree $l$ in the entries of $S_l - I_l$ whose coefficients are bounded by
$p^k M^{2(k+2)}$ times a constant $C_{\psi} = C_{\psi}(k)$ that depends only on
$k$ (cf. Proposition~\ref{p3}). Therefore, \eqref{pp4.2} can be bounded by 
\begin{align} \nonumber
	p^k  M^{2(k+2)}
	d^{\frac{h+l}{2}}
	\sum_{\ell=0}^l  
	\E \left[
	 \| S_l - I_l\|^{h+\ell} 
	\left\{
	\|S_l - I_l\|\geq \frac{1}{p \, \xi(l)}
	\right\}
	\right]
\end{align}
times a constant that depends only on $k$ ($C_\psi$ times some factor bounding
the number of different monomials in $\psi_x(S_l-I_l)$).
Now observe that 
\begin{align*}
	\|S_l-I_l\|^{h+l + 1} 
	&\quad\ge \quad
	\|S_l-I_l\|^{h+\ell} \|S_l-I_l\|^{l-\ell+1} \left\{ \|S_l-I_l\|\ge
	\frac{1}{p\, \xi(l)} \right\} \\
	&\quad\ge \quad
	\|S_l-I_l\|^{h+\ell}
	\left(\frac{1}{p\, \xi(l)}\right)^{l-\ell+1}
	\left\{ \|S_l-I_l\|\ge \frac{1}{p\, \xi(l)} \right\}, 
\end{align*}
and use this to further bound the expression in the second to last display
by
\begin{align*}
	& M^{2(k+2)}
	\sum_{\ell=0}^l  
	\xi(l)^{k-\ell+1} 
	p^{k+l - \ell + 1}
	d^{-1/2}
	\E \left[
	 \| \sqrt{d}(S_l - I_l)\|^{h+l + 1} 
	\right] \\
	&\quad \le\quad
	(k+1) \xi(l)^{k+1}  p^{2k + 1} M^{2(k+2)} d^{-1/2}
	\alpha
\end{align*}
where, again, we have bounded
$\E[\| \sqrt{d}(S_l-I_l)\|^{h+l+1}]$ by the constant $\alpha \geq 1$ 
from \ref{c.Sk}.(\ref{c.sup}) using Lyapunov's inequality.

To deal with \eqref{pp4.3}, we first 
note that $\varphi_{x}(Z_1,\dots, Z_l) / (\phi(Z_1)\cdots \phi(Z_l))$
is bounded by $\eta \det S_l^{-p/2} \exp(l M^2/2)$, 
in view of Proposition~\ref{p2}, where $0 \le \eta\le e^{2l^2/3}$ holds
because, under our present assumptions, $p^2(1-(p+l-1)/d)^{-1}/d \le 4/3$. Using
H\"{o}lder's inequality with $a>1$ and $b$ 
such that $\frac{1}{a}+\frac{1}{b} = 1$,
an upper bound for \eqref{pp4.3} is given by
\begin{align}
	&e^{\frac{2 k^2}{3}+\frac{k M^2}{2}} 
	\left( \E\left[ \det S_l^{-\frac{b p}{2}}\right]\right)^{\frac{1}{b}}
	\left(\E\left[ d^{\frac{a(h+l)}{2}} \|S_l-I_l\|^{ah}
	\left\{ \|S_l-I_l\|\ge \frac{1}{p\, \xi(l)}
	\right\}\right]\right)^{\frac{1}{a}}.\label{pp4.4}
\end{align}
For $\eps \in[0,1/2]$ as in \ref{c.Sk}.(\ref{c.sup}), fix
$\delta\in(0,1/2]$ and use a similar argument as in the preceding paragraph to
obtain
\begin{align*}
	\|S_l-I_l\|^{h} \left\{ \|S_l-I_l\| \ge \frac{1}{p \xi(l)} \right\}
	\quad \le \quad
	(p\xi(l))^{l+\eps+\delta} \|S_l-I_l\|^{h+l+\eps+\delta}.
\end{align*}
With this, and by choosing $a=\frac{h+l+\eps+2\delta}{h+l+\eps+\delta} > 1$, we
can bound the last factor in \eqref{pp4.4} with the $1/a$ exponent as
\begin{align*}
	&\left(
	\E\left[ d^{\frac{a(h+l)}{2}} \|S_l-I_l\|^{ah}
	\left\{ \|S_l-I_l\|\ge \frac{1}{p\, \xi(l)} \right\}\right]
	\right)^{\frac{1}{a}} \\
	&\le \quad
	(p\xi(l))^{l+\eps+\delta} d^{-\frac{\eps+\delta}{2}} 
	\left(
	\E\left[ \|\sqrt{d}(S_l-I_l)\|^{a(h+l+\eps+\delta)}\right]
	\right)^{\frac{1}{a}} \\
	&\le \quad
	(p\xi(l))^{k+\eps+\delta} d^{-\frac{\eps+\delta}{2}} 
	\max\left\{1,
	\E\left[ \|\sqrt{d}(S_k-I_k)\|^{(h+k+\eps+2\delta)}\right]
	\right\},
\end{align*}
where the last expression in parentheses is, again,  bounded by $\alpha\ge 1$
in view of \ref{c.Sk}.(\ref{c.sup}). At the end of the proof we will
see that it is optimal to choose $\delta=1/2$.
It remains to control the factor $(\E \det S_l^{-b p /2})^{1/b}$
in \eqref{pp4.4}.
First, in view of Lyapunov's inequality, increasing $b =
a/(a-1) = (h+l+\eps+2\delta)/\delta$ to $\tilde{b} =
(2k+2\delta +1)/\delta$ only increases this term. We can therefore work
with $\tilde{b}$ instead, which, after fixing $\delta$, depends only on $k$.
Next, decompose $\det S_l$ 
as $\det S_l = \prod_{i=1}^l ( || (I_d - P_{Z_1,\dots, Z_{i-1}}) Z_i||^2/d) $
(cf. the beginning of the proof of Lemma~\ref{Taylor5}), and set 
$$
	X_i 
	\quad=\quad 
	\left(\frac{d}{|| (I_d - P_{Z_1,\dots,
	Z_{i-1}}) Z_i||^2}\right)^{\tilde{b}p}
$$ 
to obtain that
\begin{align*}
	\E \det S_l^{-\tilde{b}p} \quad = \quad
	\E \prod_{i=1}^l X_i 
	\quad \leq \quad
	\prod_{i=1}^l  \left( 	
		\E X_i^{2^i}
		\right)^{\frac{1}{2^i}}
\end{align*}
upon repeatedly using the Cauchy-Schwarz  inequality.
Now the expectation in the $i$-th factor on the far right-hand-side of the
preceding display is bounded by
\begin{align*}
	\E X_i^{2^i}
	&\;=\;
	\E\left( \frac{d}{\|(I_d-P_{Z_1,\dots,Z_{i-1}})Z_i\|^2}
		\right)^{p\tilde{b}2^i}
	\;=\;
	\E\left( \frac{d}{\|(I_d-P_{Z_1,\dots,Z_{i-1}})Z_k\|^2}
		\right)^{p\tilde{b}2^i}
	\\
	&\;\leq \;
	\E\left( \frac{d}{\|(I_d-P_{Z_1,\dots,Z_{k-1}})Z_k\|^2}
		\right)^{p\tilde{b}2^i}
	\;\leq \;
	C(k,\tilde{b} 2^i) (4\pi e D^2)^{p\tilde{b}2^i},
\end{align*}
provided that $d> \max\{2k + 4p\tilde{b}2^i, p^2\}$ in view of
Lemma~\ref{p4.l1}. Since $\tilde{b} = (2k+2\delta+1)/\delta$, 
this will certainly hold for all $i=1,\dots, k$,
if $d> 2k + 2^{k+2} p(2 k+2\delta+1)/\delta$. This lower bound is minimized
over $(0,1/2]$ for $\delta=1/2$, leading to the requirement that $d > 2k +
p(2k+2)2^{k+3}$, which holds by assumption. We therefore have
\begin{align*}
	\E \det S_l^{-\tilde{b}p} 
	\; \leq \;
	\prod_{i=1}^l \left( C(k,\tilde{b} 2^i)^{1/2^i}
	\left(4\pi e D^2\right)^{p\tilde{b}} \right)
	\;=\; 
	\left(4\pi e D^2\right)^{p\tilde{b}l}  \prod_{i=1}^l
	C(k,\tilde{b} 2^i)^{1/2^i}.
\end{align*}
and hence
\begin{align*}
	\left( \E\left[ \det S_l^{-\tilde{b}p/2}\right] \right)^{1/\tilde{b}}
	\quad\le\quad
	\left(2D\sqrt{\pi e}\right)^{pk} \prod_{i=1}^l
	C(k,\tilde{b} 2^i)^{\frac{1}{\tilde{b}2^{i+1}}}
\end{align*}
because 
$\E[ \det S_l^{-\tilde{b}p/2}] \leq (\E[ \det S_l^{-\tilde{b}p}])^{1/2}$.
After putting the pieces together,
it is now elementary to verify that the resulting upper bound
for \eqref{pp4.4} has the desired properties.
\end{proof}

\section{Proofs for Section~\ref{s3.6}}
\label{appendixE}

\begin{lemma}\label{normalzero}
For positive integers $p< d$ and a positive even integer $k$ such that $k \le
d-p$, let $V_1,\dots, V_k$ be i.i.d. standard Gaussian $d$-vectors. For
$x\in\R^p$, let $\varphi_x$ denote the density given by Proposition~\ref{p2}.
Then the expressions
$$
	\E \left[  \left(\prod_{i=1}^m V_{j_{i-1}+1}'V_{j_{i-1}+2}\cdots
	V_{j_i-1}' V_{j_i} \right) 
	\frac{\varphi_x(V_1, \dots, V_l)}{\phi(V_1) \cdots \phi(V_l)}\right]  -
	\|x\|^{2(j_m-m)} 
$$
are all equal to zero,
for each $l=1,\dots, k$, for each $m\ge 0$ and for each set of indices
$j_0\dots, j_m$ that satisfies $j_0=0$, $j_m \le l$ and $j_{i-1} +1 < j_i$
whenever $0 < i \le m$. Moreover, 
$$
	\sum_{j=1}^k (-1)^{k-j}\binom{k}{j} \E \Bigg[  \left(V_1'V_2\cdots
	V_jV_j' V_1 -d +p- 1 \right) 
	\frac{\varphi_x(V_1,\dots
	V_k)}{\phi(V_1)\cdots\phi(V_k)}\Bigg] - (1-\|x\|^2)^k
$$
also equals zero.
\end{lemma}

\begin{proof}
Let $W_1,\dots, W_k$ be defined by \eqref{Wj} with $\mathbf B$ uniformly
distributed on $\mathcal V_{d,p}$ and note that the $W_i$ are conditionally
independent given $\mathbf B$, $\E[W_i\|\mathbf B] = \mathbf Bx$ and $\E[W_i
W_i' \| \mathbf B] = \mathbf Bxx'\mathbf B' + (I_d-\mathbf B \mathbf B')$ for
all $i=1,\dots, k$. 

For the first statement, note that if $m=0$ then also $j_m
= j_0 = 0$ and so we have $1-1=0$,
because $\varphi_x$ is a density w.r.t. Lebesgue measure and hence
integrates to one.
In case $m>0$, we may write the expected value of interest as
\begin{align*}
	&\E \left[\prod_{i=1}^m W_{j_{i-1}+1}'W_{j_{i-1}+2}\cdots
	W_{j_i-1}' W_{j_i} \right] \\
	&\quad =\quad
	\E \left( \prod_{i=1}^m \E \left[
	W_{j_{i-1}+1}'W_{j_{i-1}+2}W_{j_{i-1}+2}'\cdots W_{j_i-1}' W_{j_i} \|
	\mathbf B\right]  \right)\\
	&\quad =\quad
	\E \left(
	\prod_{i=1}^m 
	x'\mathbf B' 
		\left[ 
			(\mathbf Bxx'\mathbf B')^{j_i-j_{i-1}-2} 
			+ (I_d-\mathbf B \mathbf B') 
		\right] 
	\mathbf Bx 
	\right) \\
	&\quad =\quad
	\prod_{i=1}^m \|x\|^{2(j_i-j_{i-1}-1)} = \|x\|^{2(j_m-m)}
\end{align*}
in view of Proposition~\ref{p2} and \eqref{Wj}, 
where we have used the convention that 
$A^0 = I_d$
for a $d\times d$ matrix $A$.

For the second statement, each expected value in the sum can be computed as
\begin{align*}
	& 
	\E \Big[  \left(W_1'W_2\cdots W_jW_j' W_1 -d +p- 1 \right) \Big] \\
	&\quad =\quad
	\E \left( \E \Big[  \trace W_1W_1'W_2\cdots W_jW_j' \Big\| \mathbf B
	\Big]  \right)   -d +p- 1 \\
	&\quad =\quad
	\|x\|^{2j} + \E \left( \trace (I_d-\mathbf B \mathbf B')  \right)   -d
	+p- 1 
	 \quad = \quad \|x\|^{2j} -1
\end{align*}
again in view of Proposition~\ref{p2} and \eqref{Wj}. 
For $j=0$ the above expression is equal to zero, so we get
\begin{align*}
	&\sum_{j=1}^k (-1)^{k-j}\binom{k}{j} \left( \|x\|^{2j} -1 \right)
	 \quad=\quad
	\sum_{j=0}^k (-1)^{k-j}\binom{k}{j} \left( \|x\|^{2j} -1 \right) \\
	&\quad =\quad
	\left( \|x\|^2 - 1 \right)^k =  \left( 1 - \|x\|^2 \right)^k
\end{align*}
since $k$ is even. 
\end{proof}


Define $S_k$ as in Section~\ref{sec:main},
and let $G$ be a monomial of degree $g$ in the elements of $S_k-I_k$.
We note that $G$ corresponds to a graph with $g$ edges, where
each divisor of $G$ of the form $(S_k-I_k)_{i,j}$ 
corresponds to an edge $(i,j)$
in the graph; multiple divisors correspond to multiple edges. 
A linear factor of $G$ corresponds to a single edge $(i,j)$,
i.e., to a situation where the vertices $i$ and $j$ are connected 
by exactly one edge.

\begin{proof}[\bf Proof of Proposition~\ref{p5}]
As an auxiliary consideration that will be used repeatedly in the
following, we note that condition~\ref{c.Sk} is always satisfied for any
$k\in\N$, with constants $\alpha^\star$ and $\beta^\star$
replacing $\alpha$ and $\beta$, $\alpha^\star$ and $\beta^\star$ not depending
on $d$, and $\eps = \xi = 1/2$, provided
that $Z$ is replaced by a standard Gaussian random vector $V$; cf. 
Example~A.1 in \citep{Lee12b}. Moreover, we will use the inequality $|H| \le
\|S_k - I_k\|^h$.
Also, it will be convenient to assume, 
throughout the proof and without loss of generality,
that both $G$ and $H$ are monomials in those elements of $S_k - I_k$
that lie on or above the diagonal. 
Finally, we may also assume without loss of generality that,
in the graph corresponding to $G$ (resp. $G^\star$), each vertex is visited by
at least one edge from the graph corresponding to $H$ (resp. $H^\star$),
in view of Lemma~F.1 in \citep{Lee12b} (cf. also the proof of Proposition~2.6
in that reference).

To derive part~\ref{p5.i}, first note that if $H$, and hence also $H^\star$,
has degree zero,
the difference of interest is equal to zero. In the case where $h>0$,
it follows from  condition~\ref{c.Sk}.(\ref{c.sup}) that
$|\E[H]| = d^{-h/2} |\E d^{h/2}H| \le d^{-h/2} \alpha$. Similarly, we obtain
that
$|\E[H^\star]| \le d^{-h/2} \alpha^\star$, such 
that the difference in question is bounded
as claimed.
 
For part~\ref{p5.ii}, we first note that if $m=0$, we have
$G = G^\star = 1$ and thus the expression in \eqref{p5.1} is equal to zero.
Otherwise, we have $\E G = 0$, because if $G$ is as in \eqref{openc} and 
$m\ge 1$,
then $G$ contains the factor $(S_k - I_k)_{1,2} = Z_1'Z_2/d$ and $Z_1$ has no
other 
occurrence in $G$. Therefore $\E Z_1 = 0$ and independence of the $Z_i$ 
entails
that $\E[G] = 0$. 
In particular, the expression in \eqref{p5.1} reduces to
$\E[d^g G H] - \E[d^g G^\star H^\star]$.
Now the cases $g<h$, $g=h$ and $g>h$ need separate treatment.

In the case where $g<h$, it follows from \ref{c.Sk}.(\ref{c.sup}) and our
auxiliary 
consideration that $|\E[d^g GH]| = d^{(g-h)/2} |\E[d^{(g+h)/2} GH]| \le
d^{-1/2} \alpha$ and, similarly, $|\E[d^g G^\star H^\star]| \le d^{-1/2}
\alpha^\star$. Hence \eqref{p5.1} is bounded by $d^{-1/2}(\alpha +
\alpha^\star)$.

In the case where $g=h$,  consider first the subcase where 
$G = H$. Because $G$ is composed solely of linear factors of $S_k-I_k$
above the diagonal, it follows
that $G H$ consists only of quadratic factors above the diagonal,
and condition~\ref{c.Sk}.(\ref{c.monom1}) entails that $|\E d^g G H - 1| \le \beta d^{-\xi}$.
In the Gaussian case, we even have $|\E d^g G^\star H^\star - 1| \le
\beta^\star d^{-1/2}$,
such that the expression in \eqref{p5.1} is bounded by 
$|\E d^gGH - \E d^g G^\star H^\star| \le d^{-\min\{\xi,1/2\}}(\beta +
\beta^\star)$.
Next, in the subcase where $g = h$ and $G \neq H$,
it follows that the graph corresponding to $G$ has an edge
that is not matched by an edge in the graph corresponding to $H$.
In other words, the monomial $G H$ has a linear factor, and
Condition~\ref{c.Sk}.(\ref{c.monom1}) entails that $|\E d^g G H| \le \beta
d^{-\xi}$.
Because the same applies to $\E d^g G^\star H^\star$ with the bound
$\beta^\star d^{-1/2}$, the claim follows.

Finally, consider the case where $g > h$. 
Because each vertex in the graph corresponding to $G$ is visited
by at least one edge from the graph corresponding to $H$, 
and because $k\leq 4$,
it is easy to see that this can only occur if $g=3$ and $h=2$. 
In particular, we see that $G$ is given by
$G = Z_1'Z_2 Z_2'Z_3 Z_3'Z_4/d^3$,  and that $H$ is equal to 
either 
$H_1 = Z_1'Z_2 Z_3'Z_4/d^2$, or 
$H_2 = Z_1'Z_3 Z_2'Z_4/d^2$, or
$H_3 = Z_1'Z_4 Z_2'Z_3/d^2$.
In either case, the vertices $1$ and $4$
have degree two in the graph corresponding to $G H$. 
With this, we obtain that
$\E G H_2 = \E G H_3 = d^{-2} \E G_\circ H_\circ$
with $G_\circ = Z_2'Z_3/d$ and $H_\circ =G_\circ^2$,
upon using Lemma~F.1(ii) in \citep{Lee12b} twice.
And we obtain that
$\E G H_1 = d^{-2} \E G_\circ H_\odot$,
with $G_\circ$  as before and with $H_\odot =  (Z_2'Z_2/d-1)(Z_3'Z_3/d-1)$,
by using parts~(i) and~(iii) of Lemma~F.1 in \citep{Lee12b}.
Either way, we end up with a monomial $G_\circ$ of degree one
and a monomial $H_\circ$ or $H_\odot$ of degree two, and that case
has already been dealt with earlier.

For part~\ref{p5.iii}, we note that
the case where $g < h$ is treated
as in the proof of part~\ref{p5.ii}, mutatis mutandis.

In the case where $g=h$ and $G \neq H$, 
it is easy to see that $G H$ has a linear factor (upon separately
treating the three subcases $g=1$, $g=2$, and $g>2$),
so that either $G$ or $H$ also has a linear factor.
Therefore, we have $|\E d^g  GH| \le  \beta d^{-\xi}$ and either
$|\E d^{g/2}  G| \le \beta d^{-\xi}$ and $| \E d^{h/2}H| \le \alpha$,
or 
$|\E d^{g/2} G| \le \alpha$ and $|\E d^{h/2} H| \le \beta d^{-\xi}$,
in view of \ref{c.Sk}.(\ref{c.sup}-\ref{c.monom1}).
Because this also holds for $G^\star$ and $H^\star$ with $\xi=1/2$ and
$\alpha^\star$ and $\beta^\star$ replacing $\alpha$ and $\beta$, respectively, 
we see that \eqref{p5.1} is bounded in, absolute value, by 
$d^{-\min\{\xi,1/2\}}(\beta+\beta^\star+\alpha\beta+\alpha^\star\beta^\star)$.

In the case where $g=h$ and $G = H$, we again consider the three 
subcases where $g=1$, where $g=2$, and where $g>2$.
In the subcase where $g=1$ (and $G=H$), we must have
$G = H = (S_k-I_k)_{1,1}$, which corresponds to the case (a)
in Proposition~\ref{p5}\ref{p5.iii}. Note that, here,   $\E G = 0$ 
and that $d \E G H = d \E (Z_1'Z_1/d-1)^2 = \Var[Z_1'Z_1]/d$. And for
$G^\star$ and $H^\star$, the first four moments of the Gaussian
give that $\E G^\star = 0$ and that
$d \E G^\star H^\star = 2$. In particular, \eqref{p5.1}
is equal to $\Var[Z_1'Z_1]/d-2$, as claimed.
In the subcase where $g=2$ (and $G=H$), we see that 
case (c) occurs and that
$d^2 \E (G-\E G) H = d^2 \Var[G] = \Var[(Z_1'Z_2)^2]/d^2$;
moreover, it is elementary to verify that
$d^2 \E (G^\star-\E G^\star) H^\star = d^2 \Var[G^\star] = 2+6/d$.
In particular, \eqref{p5.1} here equals  $\Var[(Z_1'Z_2)^2]/d^2 -2(1+3/d)$.
And in the subcase where $g>2$ (and $G=H$), we see that
both $G$ and $H$ are composed only of linear factors
above the diagonal of $S_k-I_k$, and
$G H$ consists only of quadratic factors above the diagonal,
so that $|\E d^g G H -1| \le \beta d^{-\xi}$ and 
$|d^g \E G \E H| = |\E d^{g/2}G|\,|\E d^{h/2} H| \le \beta^2 d^{-2\xi}$
by condition~\ref{c.Sk}.(\ref{c.monom1}).
Because a similar statement also holds with $\xi=1/2$ and $G^\star$, $H^\star$,
$\alpha^\star$ and $\beta^\star$ replacing 
$G$, $H$, $\alpha$ and $\beta$, respectively, \eqref{p5.1} is bounded in
absolute value by
$d^{-\min\{\xi,1/2\}}(\beta+\beta^2+\beta^\star+{\beta^\star}^2)$.

Lastly, in the case where $g>h$, the fact that each vertex in
the graph corresponding to $G$ is visited by an edge from the
graph corresponding to $H$ entails that
we either have $h=1$ and $g=2$, or $2 \leq h < g\leq k$.
If $h=1$ and $g=2$, we must have $G = (Z_1'Z_2/d)^2$ and
$H = Z_1'Z_2/d$. In particular, we see that case (b) of
Proposition~\ref{p5}\ref{p5.iii} occurs. For such $G$ and $H$, we obviously
have $d^2 \E G H = \E (Z_1'Z_2)^3/d$ and $d^2 \E G \E H = 0$.
And for $G^\star$ and $H^\star$, it is easy to see that
$d^2 \E G^\star H^\star = d^2 \E G^\star \E H^\star = 0$. Here, \eqref{p5.1}
is equal to $\E (Z_1'Z_2)^3/d$, as claimed.
Finally, consider the case where $2 \leq h < g \leq k$.
Here, we obtain that $|\E d^g G H| \le \beta d^{-\xi}$, in view of condition
\ref{c.Sk}.(\ref{c.monom2}) (because each vertex of $G$ is visited by an edge
from $H$, it follows that $H$ depends at least on $Z_1,\dots, Z_g$),
and also that $|\E d^g G^\star H^\star| \le \beta^\star d^{-1/2}$.
And we obtain that $\E H = \E H^\star = 0$ by Lemma~F.1(i) in \citep{Lee12b},
because the number of vertices of $H$ is at least $g$ (as edges in $H$
visit all vertices in $G$), and because $h < g$. Altogether, we see that 
the expression in \eqref{p5.1} is bounded in absolute value by
$d^{-\min\{\xi,1/2\}} (\beta + \beta^\star)$.
\end{proof}

\end{appendix}


{\small

}

\end{document}